\documentclass[10pt]{article}
\newcommand\tab[1][1cm]{\hspace*{#1}}
\usepackage{enumitem}
\usepackage{amsthm,amsfonts}
\usepackage{amsmath,amssymb,wasysym}
\usepackage{graphicx}
\usepackage{wrapfig}
\usepackage{comment}
\usepackage{caption}
\usepackage{subcaption}
\usepackage{tikz}
\usepackage{breqn}
\usepackage{booktabs}
\usepackage{authblk}
\usepackage{layout}
\usepackage[makeroom]{cancel}

\usepackage{pgfplots}
\pgfplotsset{compat=1.18}

\usepackage[notcite,notref]{}
\usepackage[%
  colorlinks=true,%
	linkcolor	=blue,%
	citecolor =blue,%
  urlcolor=blue%
]{hyperref}

\newcommand{\ControlC}{\textsf{\textup{\textbf{c}}}}
\newcommand{\ControlS}{\textsf{\textup{\textbf{s}}}}

\usepackage{hypcap}%importante
\usepackage{float}% Antes nofloat importante
\DeclareMathOperator*{\esssup}{ess\,sup}
\usepackage[left=2.8cm,right=2.8cm, top=3.1cm, bottom=2.8cm]{geometry}
\newtheorem{theo}{Theorem}[section]
\newtheorem{lemm}[theo]{Lemma}

\newtheorem{prop}[theo]{Proposition}

\newtheorem{defi}[theo]{Definition}
\newtheorem{remark}[theo]{Remark}
\usepackage{stackengine}%para hacer inmersión compacta
\newcommand\dhookrightarrow{\mathrel{%
		\ensurestackMath{\stackanchor[.1ex]{\hookrightarrow}{\hookrightarrow}}
}}%hasta aqu\'i inmersión compacta
\newcommand{\leu}{\lambda_\varepsilon (u_\varepsilon)}
\newcommand{\R}{\hbox{\rm I \kern -5pt R}}     % simbolo de Reales
\newcommand{\p} {\hbox{\rm I \kern -5pt P}}
\def\z        {\textit{\textbf{z}}}
\def\w        {\textit{\textbf{w}}}
\def\u        {\textit{\textbf{u}}}
\def\v        {{\textit{\textbf{v}}}}
\def\f        {\textit{\textbf{f}}}
\def\x        {\textit{\textbf{x}}}
\def\W        {\textit{\textbf{W}}}
\def\X        {\textit{\textbf{X}}}
\def\V        {\textit{\textbf{V}}}
\def\H        {{\boldsymbol H}}
\def\L       {{\boldsymbol L}}

\def\Om       {\Omega}
\def\div      {\nabla\cdot}
\def\d        {\Delta}
\def\n        {\nabla}

\def\l        {\lambda}

\def\e        {\quad}

\def\rdd      {\rho^{n+1}_{h}}
\def\rd       {\rho^{n}_{h}}
\def\dss      {\eta^{n+1}}
\def\ds       {\eta^{n}}
\def\uss      {\textit{\textbf{u}}^{n+1}}
\def\us       {\textit{\textbf{u}}^{n}}
\def\udd      {\textit{\textbf{u}}^{n+1}_{h}}
\def\ud       {\textit{\textbf{u}}^{n}_{h}}
\def\pr       {p^{n+1}_{h}}

\def\fdd      {\bar{\rho}_{h}}
\def\fdc      {\bar{\rho}}
\def\fvd      {\bar{{\textit{\textbf{u}}}}_h}
\def\fvc      {\bar{{\textit{\textbf{u}}}}}
\def\auxvv   {\textit{\textbf{u}}_{h,k}}
\def\auxv    {\widehat{\textit{\textbf{u}}}_{h,k}}
\def\auxvvl   {\textit{\textbf{u}}_{h,k,\l}}
\def\auxvl   {\widehat{\textit{\textbf{u}}}_{h,k,\l}}
\def\auxdd    {\rho_{h,k}}
\def\auxd    {\widehat{\rho}_{h,k}}
\def\auxddl    {\rho_{h,k,\l}}
\def\auxdl    {\widehat{\rho}_{h,k,\l}}
\def\auxdc    {\widetilde{\rho}_{h,k}}
\def\auxdcl    {\widetilde{\rho}_{h,k,\l}}

\numberwithin{equation}{section}
\definecolor{ao(english)}{rgb}{0.0, 0.5, 0.0}
\definecolor{aqua}{rgb}{0.0, 1.0, 1.0}
\definecolor{bazaar}{rgb}{0.6, 0.47, 0.48}
\definecolor{britishracinggreen}{rgb}{0.0, 0.26, 0.15}

\newcommand{\imgthree}{0.32\textwidth}
\newcommand{\imgfour}{0.24\textwidth}

\title{Optimal control of therapies related to an oxytaxis glioblastoma model}

\author[a]{Juan J. Forero-Herna\'ndez\thanks{E-mail: \href{mailto: juan2258036@correo.uis.edu.co}{ juan2258036@correo.uis.edu.co}}}
\author[b]{Francisco Guill\'en-Gonz\'alez\thanks{E-mail: \href{mailto:guillen@us.es}{guillen@us.es}}}
\author[a]{\'Elder J. Villamizar-Roa\thanks{E-mail (Corresponding author) \href{mailto:jvillami@uis.edu.co}{jvillami@uis.edu.co}}}

\affil[a]{Universidad Industrial de Santander, Escuela de Matem\'{a}ticas,  Carrera 29 Calle 9, A.A. 678, CP 680002, Bucaramanga, Colombia.}
\affil[b]{Universidad de Sevilla, Dpto. Ecuaciones Diferenciales y Análisis Numérico and IMUS, Facultad de Matemáticas, C. Tarfia, s/n, 41012, Sevilla, España.}

\renewcommand\Authands{, }
\renewcommand\Affilfont{\fontsize{9}{10.8}\itshape}
\date{}

\begin{document}
\maketitle

%%%%%%%%%%%%%%%%%%%%%%%%%%%%%%%%%%%%%%%%%%%%%%%%%%%%%%%%%%%%%
%%%%%%%%%%%%%%%%%%%%%%%%%%%%%%%%%%%%%%%%%%%%%%%%%%%%%%%%%%%%%
% RESUMEN
%%%%%%%%%%%%%%%%%%%%%%%%%%%%%%%%%%%%%%%%%%%%%%%%%%%%%%%%%%%%%
%%%%%%%%%%%%%%%%%%%%%%%%%%%%%%%%%%%%%%%%%%%%%%%%%%%%%%%%%%%%%

\begin{abstract}
We propose and analyze an optimal control problem associated with a Keller-Segel type parabolic system with chemoattraction, modeling the glioblastoma growth 
 in a bi-dimensional bounded domain, influenced by the presence of oxygen where the controls are   two different (chemotherapy and antiangiogenic) therapies. The model considers the random diffusion of tumor cells and oxygen, the movement of cells towards   the oxygen gradient (oxytaxis), and reaction terms describing the interaction between cells and oxygen. We establish a mathematical framework to analyze the existence and uniqueness of weak-strong solution of the model and subsequently we analyze an optimal control problem considering a cost functional that minimizes both the tumor growth and the oxygen concentration.
 We prove the existence of a global optimal solution and derive necessary first-order optimality conditions. Finally, we propose a methodology for approximating the optimal therapies. We use the gradient of the reduced cost functional through the adjoint scheme, and minimize the cost functional  implementing the Adam gradient optimization method. Some numerical experiments are provided to demonstrate the effectiveness
  of the proposed scheme.

\vspace{0.3cm}

\vspace{0.3cm}

%\noindent{\bf AMS subject classifications.} 35K55; 35Q35; 35Q92; 92C17
\end{abstract}

%%%%%%%%%%%%%%%%%%%%%%%%%%%%%%%%%%%%%%%%%%%%%%%%%%%%%%%%%%%%%
%%%%%%%%%%%%%%%%%%%%%%%%%%%%%%%%%%%%%%%%%%%%%%%%%%%%%%%%%%%%%
% CAPÍTULO 1: INTRODUCCIÓN
%%%%%%%%%%%%%%%%%%%%%%%%%%%%%%%%%%%%%%%%%%%%%%%%%%%%%%%%%%%%%
%%%%%%%%%%%%%%%%%%%%%%%%%%%%%%%%%%%%%%%%%%%%%%%%%%%%%%%%%%%%%
\section{Introduction}
%{\color{britishracinggreen}
Cell migration driven by chemical gradients, referred to as chemotaxis, constitutes a fundamental biological mechanism underlying processes such as tissue morphogenesis, immune system function, and tumor progression. From a mathematical perspective, this phenomenon has been extensively characterized through Keller–Segel type models, which remain among the most widely employed frameworks for representing directed motion in biological systems \cite{kellersegel}. In recent decades, the classical Keller–Segel formulation has undergone numerous extensions designed to more accurately capture the complexity of biological processes \cite{ Bellomo2015, Hillen2009, winkler}. Such extensions have been applied to describe specific biomedical phenomena, including tumor invasion and cell–extracellular matrix interactions, as exemplified in haptotaxis models \cite{Chaplain2000, Chaplain2006, Celis}. Furthermore, the foundational Keller–Segel system has provided the basis for specialized models that address various pathologies, notably gliomas, a prevalent class of highly aggressive and morphologically heterogeneous brain tumors. Several authors have proposed versions of the Keller–Segel model adapted to the study of glioblastoma, in order to reproduce its infiltrative behavior and its interaction with the brain microenvironment, see, for instance, \cite{Fernandez,Leo}. At the same time, there has been a recent interest in optimal control problems in Keller-Segel type models \cite{PerusatoKS}. These control problems aim to formulate optimal interventions, such as the administration of drugs in tumor models, among others, under constraints imposed by PDE systems (see \cite{DaiLiu, De_Araujo2015}).\\

In this paper, we formulate and analyze an optimal control problem associated with a mathematical model describing the glioblastoma growth. Glioblastoma stands out the most aggressive subtype of brain tumors, characterized by its highly infiltrative nature with a pronounced resistance to conventional therapies \cite{Menze2015}. The complex dynamics of tumor cell infiltration into brain tissue present substantial challenges from both clinical and scientific perspectives, for which mathematical modeling has emerged as a critical tool to elucidate propagation mechanisms and to investigate potential therapeutic strategies \cite{anderson1999}.\\

Some studies (see \cite{Gomez2017, Leo2, Swanson}) have highlighted the role of the oxygen in the cellular proliferation, as well as in the aggressiveness and invasiveness of glioblastoma. In this sense, without considering control terms, in \cite{Leo}, a Keller-Segel type parabolic system  was proposed as a generalization of the reaction-diffusion model presented in \cite{Gomez2017}, to predict the proliferation-invasion structure of glial cells in the brain, including the mechanisms of invasion and proliferation of gliomas, and taking into account the attractive taxis mechanism. Inspired by \cite{Leo}, in this paper, we propose and analyze  an optimal control problem to deal with the treatment of the tumor invasion, by using chemotherapic and antiangiogenic controls where the state equations describe the evolution of glial cells in the presence of oxygen. Explicitly, given a bounded and sufficiently regular domain $\Omega \subset \mathbb{R}^2$, $0<T\leq \infty$, $Q:= \Omega \times (0,T)$ and $\Sigma=\partial \Omega \times [0,T),$ the state-control model is given by the following parabolic system:
\begin{equation}
	\begin{cases}
		\partial_t u= \overbrace{D_u\Delta u}^{\mbox{\tiny{Random Diffusion}}}+ \overbrace{\rho(\sigma)u(\alpha-u)}^{\mbox{\tiny{Logistic Growth}}} - \overbrace{\kappa \ControlC u \sigma}^{\mbox{\tiny{Chemotherapy}}} - \overbrace{\chi \nabla \cdot \left( u \nabla \sigma\right)}^{\mbox{\tiny{Oxytaxis}}},& \text{in } Q,\\
		\partial_t \sigma= \overbrace{D_\sigma \Delta \sigma}^{\mbox{\tiny{Random Diffusion}}}+\overbrace{P_{er}S_{v}(\beta-\sigma)}^{\mbox{\tiny{Reaction-Production}}}-\overbrace{{\frac{A_{ox}u\sigma}{k_{ox}+\sigma}}}^{\mbox{\tiny{Consumption}}}+\overbrace{(1-\ControlS) S_c u}^{\mbox{\tiny{\shortstack{Antiangiogenic\\Therapy}}}
},& \text{in }Q, \\
        \partial_{\mathbf{n}}u=\partial_{\mathbf{n}}\sigma = 0, & \text{on }\Sigma, \\
        u(0)= u_0, \sigma(0)= \sigma_0, & \text{in }\Omega,
	\end{cases}
\label{eq:SistemaControlado}
\end{equation}
where $u$ and $\sigma$ represent the cell density and the oxygen concentration, respectively. In addition, $u_0$ and $\sigma_0$ are given functions defined on $\Omega$ that represent the initial conditions of the unknowns $u$ and $\sigma$. The term $D_u$ is the diffusion coefficient of tumor cells, while \( D_\sigma \) represents the diffusion rate of oxygen in the tissue. The parameters \( A_{ox} \) and \( k_{ox} \) describe the oxygen consumption rate by tumor cells according to Michaelis-Menten kinetics given by $\frac{A_{ox}u\sigma}{k_{ox}+\sigma}$. The coefficient $P_{er}$ refers to the vascular permeability, which controls the amount of oxygen that can pass through the blood vessel walls; $S_v$ represents the vascular density, assumed constant; and $S_c$ denotes the rate of nutrient supply associated with the release of angiogenic factors by cancerous tissue. Finally, the supply of oxygen from the vasculature is modeled by a constant source term proportional to \( \beta \).\\

The oxytaxis term $\chi \nabla \cdot \left( u \nabla \sigma \right)$ models the directed movement of tumor cells towards regions with higher oxygen concentration, where \( \chi \geq 0 \) quantifies the sensitivity of cells to oxygen gradients. The tumor cell proliferation is described by a logistic term \( \rho(\sigma)u(\alpha-u) \), where $\alpha \geq 1$ is the carrying capacity; the function \( \rho(\sigma) \) modulates the growth rate depending on oxygen availability; and \(u(\alpha-u) \) accounts for saturation due to overcrowding. A particular feature of the model is that it captures the hypoxia-induced phenotypic switch from a proliferative to a migratory behavior.  In this paper, we assume that $\rho:\mathbb{R}^+ \to \mathbb{R}$ is a function of class $C^1[0,+\infty)$, with bounded derivative, satisfying the following bounds: there exist constants $\rho_{\min}, \mu_1, \mu_2 > 0$ such that $0 < \rho_{\min} \leq \rho(s) \leq \mu_1s + \mu_2$, for all $s\ge 0$.  A prototype function in the literature (see, e.g. \cite{Gomez2017, Leo2}) is given by  the following affine linear function
\[
\rho(\sigma) = \frac{\hat \rho}{\alpha} \left( \frac{\sigma}{\beta} + b \left( 1 - \frac{\sigma}{\beta} \right) \right),
\]
whose parameter \( \hat\rho \) represents the maximum proliferation rate under normoxic conditions, and \( b \in [0,1] \) modulates the extent to which low oxygen levels reduce proliferation. This formulation enables the model to reproduce the biologically observed decrease in cell division under hypoxic conditions, accompanied by increased motility.\\

System (\ref{eq:SistemaControlado}) incorporates two control functions $[\ControlC, \ControlS]$, which can be interpreted as therapeutic interventions. The function $\ControlC(x,t)$, with $0 \leq \ControlC \leq 1$ represents the administration of a cytotoxic drug (chemotherapy) which reduces the tumor cell population. This control is modeled with the decreasing (trilinear) term $-\kappa \ControlC u \sigma$ in the tumor equation where the parameter $\kappa > 0$ represents the cytotoxic therapeutic intensity and the dependence of oxygen concentration is introduced to reflect that the efficacy of the cytotoxic drug can be modulated by the availability of oxygen. The second control $\ControlS(x,t)$, with $0 \leq \ControlS \leq 1$, represents the effect of an antiangiogenic therapy by means of the (bilinear) term $(1-\ControlS)S_c u$, which modulates the supply of angiogenic factors released by the tumor. The modeling of the control system is a first contribution to be highlighted in this paper.\\

From a biomedical perspective, the objective is to regulate or prevent tumor invasion into surrounding healthy tissue and its spread to other regions of the body, through combined treatments such as chemotherapy and angiogenesis inhibitors. However, these therapeutic approaches are associated with significant side effects, including fatigue, nausea, vomiting, hair loss, and reduced blood cell production, which may lead to immunosuppression \cite{Joshi2009}. Therefore, it is reasonable to formulate an optimal control problem aimed at minimizing tumor growth while simultaneously regulating oxygen concentration toward a desired state 
%over the entire time horizon $[0,T]$ and at the final time $t=T,$
 with the additional objective of minimizing drug dosages to mitigate adverse effects of therapy.\\

Before establishing the optimal control problem, we clarify the setting for the existence of solutions. In this line of thought, we introduce the main notation and some preliminaries that will be employed in this paper. We denote by $W^{m,p}(\Omega)$ and $L^p(\Omega),$ $1\leq p\leq \infty,$ the Sobolev and Lebesgue spaces with norms $\Vert \cdot\Vert_{W^{m,p}}$ and $\Vert \cdot\Vert_{L^p},$ respectively. In particular, we denote by $H^k(\Omega):=W^{k,2}(\Omega)$ the usual Hilbert--Sobolev space. On $L^2(\Omega)$, we denote by $(\cdot,\cdot)$ its inner product. We denote by $L^p(X)=L^p(0,T;X)$ the space of Bochner integrable functions defined on the interval $[0,T]$ with values in a Banach space $X$, endowed with the usual norm $\Vert \cdot\Vert_{L^p(X)};$ in particular, we denote by $L^p(Q):=L^p(0,T;L^p(\Omega)),$ $Q\equiv \Omega\times (0,T).$ In general for a Banach space $X$ the duality product we will denote by $\langle \cdot , \cdot \rangle$. Also, we consider the Banach space $C(X)=C([0,T];X)$ of continuous functions from $[0,T]$ into $X,$ with norm $\Vert \cdot\Vert_{C(X)}$. In addition, frequently we will use the following  Banach spaces:
\[
W_2 := \{s \in C(L^2)\cap L^2( H^1), \quad \partial_t s \in L^2(H^{-1})\},
\]
and
\[
X_p := \{s \in C(W^{2-2/p, p}) \cap L^p(W^{2,p}), \quad \partial_t s \in L^p(Q) \}, \quad (p > 1).
\]
In \cite{Leo}, considering $\textbf{c}=0,$ and $\textbf{s}=0,$  the authors proved the existence and uniqueness of strong solutions in bounded domains in $\mathbb{R}^2$, and proposed a nonlinear fully discrete finite element (FE) for approximate the solutions of the continuous model. Indeed, some qualitative properties of the numerical scheme, including the well-posedness, some uniform estimates, the positivity for the discrete oxygen concentration and approximate positivity for the discrete tumor cells were demonstrated. Some numerical simulations were shown to validate the results. In this paper we consider the well-posedness of (\ref{eq:SistemaControlado}) including the control terms, focusing our analysis on a weak-strong setting, which is more interesting from the point of view of the optimal control theory.

% ===================================================
% ============ DEFINICIÓN DE SOLUCIÓN ===============
% ===================================================
\begin{defi}\label{defSolFuerte}
    Consider $[\ControlC, \ControlS] \in [L^\infty(Q)]^2$ with $0\leq \ControlS\leq 1$ a.e. in $Q$, $u_0 \in L^2 (\Omega),$ $\sigma_0 \in H^1(\Omega)$ with $u_0 \geq 0$ and $\sigma_0 \geq 0$ a.e. in $\Omega$. A pair $[u,\sigma]$ is called a weak-strong solution of problem \eqref{eq:SistemaControlado} in $[0,T]$, if $u \geq 0,$ $\sigma \geq 0$  a.e. in $Q$,
    $$u \in W_2,\quad \sigma \in X_{2},$$ the equation \eqref{eq:SistemaControlado}$_1$ and the boundary condition for $u$ are satisfied in the variational sense, equation \eqref{eq:SistemaControlado}$_2$ and the boundary condition for $\sigma$ pointwise a.e. in $Q$, and the initial conditions \eqref{eq:SistemaControlado}$_3$ and \eqref{eq:SistemaControlado}$_4$ in the sense  $L^2(\Omega)$ and $H^1(\Omega)$, respectively.
\end{defi}

Our main result of existence and uniqueness of solutions of \eqref{eq:SistemaControlado} is given by the following

 \begin{theo}\label{teoSolFuerte}
    Let $u_0 \in L^2(\Omega)$, $\sigma_0 \in H^1(\Omega)$ with $u_0 \geq 0$ and $\sigma_0 \geq 0$ a.e. in $\Omega$, and $[\ControlC,\ControlS] \in [L^{\infty}(Q)]^2$ with $0 \leq \ControlS \leq 1$ a.e. in $Q$. Then there exists a unique weak-strong solution $[u, \sigma]$ of the system \eqref{eq:SistemaControlado} in the sense of Definition \ref{defSolFuerte}. 
 \end{theo}

Now we establish the main results of the optimal control related to the states \eqref{eq:SistemaControlado}.  First, in order to balance the clinical benefits and negative consequences of therapies, we introduce the following cost functional including positive and bounded weighting functions $k_i(x,t) \in L^\infty(Q)$ ($i = 1,2,3,4$), 
and $l_j(x) \in L^\infty(\Omega)$ ($j=1,2$): 
\begin{equation}\label{funcionalObjetivo}
    \begin{split}
        J([u(\ControlC,\ControlS), \sigma(\ControlC,\ControlS), \ControlC,\ControlS]) 
        &= \int_{Q} \Big[\frac{k_1}{2}u^2 + \frac{k_2}{2} \left(\sigma - \sigma_Q \right)^2 + k_3\ControlC + k_4\ControlS\Big](x,t)dxdt \\ 
        & + \int_\Omega \Big[\frac{l_1}{2}u^2 + \frac{l_2}{2}\left(\sigma-\sigma_\Omega\right)^2\Big](x,T)dx,
    \end{split}
\end{equation}
where $\sigma_Q \in L^2(Q)$ and $\sigma_\Omega \in L^2(\Omega)$
 represents the desired state for oxygen levels in $Q$ and in $\Omega$ at the final time $T$, respectively. Additionally, we define the set of admissible controls $[\ControlC,\ControlS]$, as
\begin{equation}\label{Uad}
\mathcal{U}_{ad} = \left\{ [\ControlC,\ControlS] \in [L^\infty(Q)]^2 : 0 \leq \ControlC, \ControlS \leq 1  \text{ a.e. in }Q ,  \int_Q \ControlC(x,t)dxdt \leq c_{max} \right\},
\end{equation}
where $c_{max}>0$ represents the feasible limit for the global therapeutic quantity that can realistically be achieved with common drugs. Since $c_{max}< T|\Omega|$ is imposed, it is not possible to use the maximum of therapy along the whole time. We also introduce the set of admissible states and controls by
\begin{equation}\label{Pad}
\begin{split}
    \mathcal{P}_{ad} = & \{ [u,\sigma,\ControlC,\ControlS] : [u,\sigma] \text{ is the unique weak-strong solution of \eqref{eq:SistemaControlado}} \text{ associated with } [\ControlC,\ControlS] \in \mathcal{U}_{ad} \}.
    \end{split}
\end{equation}
Then, the optimization problem we want to analyze is posed as follows: Minimize the objective functional $J([u,\sigma,\ControlC,\ControlS])$ subject to the state system \eqref{eq:SistemaControlado}; that is, to find a {\it global optimal} solution $[u^*,\sigma^*,\ControlC^*,\ControlS^*] \in \mathcal{P}_{ad}$ such that
\begin{equation}\label{eq:ProblemaDeControl}
    J([u^*,\sigma^*,\ControlC^*,\ControlS^*]) \le J([u, \sigma, \ControlC,\ControlS]),
    \quad \forall\, [u,\sigma,{\ControlC},\ControlS]\in \mathcal{P}_{ad}.
\end{equation}
If the minimum is just in a neighborhood, then it is a {\it local optimal control.} 

\begin{theo}\label{exis2}
Under the hypotheses of Theorem \ref{teoSolFuerte}, there is at least a global optimal solution $[u^*,\sigma^*,\ControlC^*,\ControlS^*] \in \mathcal{P}_{ad}$ for the controlled system \eqref{eq:SistemaControlado} such that
\[
    J([u^*,\sigma^*,\ControlC^*,\ControlS^*]) = \min_{[u,\sigma,{\ControlC},\ControlS]\in \mathcal{P}_{ad}} J([u, \sigma, \ControlC,\ControlS]).
\]
\end{theo}

\begin{theo}[First-order necessary optimality condition]\label{Teo:CNPO}
Let $[\ControlC^*, \ControlS^*] \in \mathcal{U}_{ad}$ be a local optimal control for \eqref{eq:ProblemaDeControl}, with associated state $[u^*, \sigma^*]$. Then, there exists a unique adjoint state $[p_1, p_2]$ satisfying the following adjoint system:
\[
\begin{cases} 
-\partial_t p_1 - D_u\Delta p_1 - \rho(\sigma^*)(\alpha - 2u^*)p_1+ \kappa \sigma^* \ControlC^*p_1 - \chi\nabla \sigma^* \cdot \nabla p_1 + \dfrac{A_{ox}\sigma^*p_2}{k_{ox}+\sigma^*} - (1-\ControlS^*)S_cp_2= k_1u^*,& \text{ in }Q, \\[0.3em]
-\partial_t p_2 - D_\sigma\Delta p_2 - \rho'(\sigma^*)u^*(\alpha - u^*)p_1 + \kappa u^*\ControlC^* p_1 + \chi \nabla \cdot (u^* \nabla p_1) + \dfrac{A_{ox}u^* p_2}{(k_{ox}+\sigma^*)^2} + P_{er}S_v  p_2 \\ \hspace{12cm}= k_2(\sigma^* - \sigma_Q),& \text{ in }Q,\\[0.3em]
\partial_{\mathbf{n}} p_1 = \partial_{\mathbf{n}}p_2 = 0, & \text{ on } \Sigma, \\ 
p_1(x,T) = l_1(x)u^*(x,T),\quad p_2(x,T) = l_2(x)(\sigma^*(x,T) - \sigma_\Omega), & \text{ in }\Omega.
\end{cases}
\]
Moreover, the following variational inequality holds:
\begin{equation}\label{eq:CondicionesOptimalidad}
    \int_Q \Big((k_3- \kappa u^* \sigma^* p_1)(\overline\ControlC - \ControlC^*) 
    + (k_4- S_c u^*p_2)(\overline{\ControlS}-\ControlS^*)\Big)(x,t)\, dxdt \geq 0,
    \quad \forall\, [\overline\ControlC, \overline\ControlS] \in \mathcal{U}_{ad}.
\end{equation}
This variational inequality characterizes the first-order necessary optimality condition and represents the gradient of the reduced cost in the feasible direction in the convex set of admissible controls $\mathcal{U}_{ad}$.
\end{theo}
The content of this paper is organized as follows: In Section 2, we present the proof of the existence and uniqueness of weak-strong solution of system \eqref{eq:SistemaControlado}. In addition, we prove a stability result of solutions with respect to the controls. In Section 3,  we prove the existence of optimal solutions restricted to state system \eqref{eq:SistemaControlado}, and derive first-order necessary optimality conditions, analyzing the differentiability of the control-to-state operator and solving the adjoint problem. In Section 4, we present a numerical scheme to solve the proposed optimal control problem, including the Adam minimization algorithm. We also present several numerical results  to illustrate the well behavior of the proposed algorithm and its applicability to simulate some  situations involving the cost functional, including the role of the weight of parameters, the robustness of the minimization problem in terms of the initial controls, the stability of the cost functional, among others. Finally,  we include an appendix section addressing the analysis of the linearized system.\\

Without loss of generality, for convenience in the analysis 
%of the continuous system \eqref{eq:SistemaControlado},
given in Sections 2 and 3 we assume that the positive parameters satisfy $D_u=D_\sigma=k_{ox}=A_{ox}=\chi=S_c=\kappa=1$ and $\gamma = P_{er}S_v$.

%%%%%%%%%%%%%%%%%%%%%%%%%%%%%%%%%%%%%%%%%%%%%%%
% CAPÍTULO 3: EXISTENCIA Y UNICIDAD DE SOLUCION DEBIL-FUERTE
%%%%%%%%%%%%%%%%%%%%%%%%%%%%%%%%%%%%%%%%%%%%%%%

\section{Existence and Uniqueness. Proof of Theorems \ref{teoSolFuerte}.}\label{sec2}

 The aim of this section is to prove the existence and uniqueness of weak-strong solution of \eqref{eq:SistemaControlado}. For the existence, we will use the Leray-Schauder Fixed Point Theorem combined with a regularization strategy. We prove a stability result for the solutions with respect to the initial data (Theorem \ref{teo:estabilidad} below), which in particular implies uniqueness.

\begin{proof}
We first consider that $u_0 \in L^2(\Omega)$, $\sigma_0 \in W^{1+,2+}(\Omega)$ and we will prove that there exists a unique weak-strong solution $[u, \sigma]$ of the system \eqref{eq:SistemaControlado} in $W_2 \times X_{2+}$. In fact, improving a little the regularity of $\sigma_0 $ (from $W^{1,2}(\Omega)$ to $W^{1+,2+}(\Omega)$), we improve the regularity of the solution $\sigma$ from $X_2$ to  $X_{2+}$. For that,
we use the Leray–Schauder Fixed Point Theorem. We introduce the following auxiliary Banach spaces
$$
\mathcal{X}_u := L^{4-}(Q), \quad \quad \mathcal{X}_\sigma := L^{\infty}(Q),
$$
and let  the map $\Gamma : \mathcal{X}_u \times \mathcal{X}_\sigma \to W_2 \times X_{2+} \hookrightarrow \mathcal{X}_u \times \mathcal{X}_\sigma$ defined by $\Gamma[\overline u,\overline \sigma] = [u,\sigma]$ where $[u,\sigma]$ is the solution of the decoupled linear problem
\begin{equation}\label{eq:demSol-linealizado}
\begin{cases}
    \displaystyle\int_0^T \langle \partial_t u, \varphi \rangle + \int_Q \nabla u \cdot \nabla \varphi + \int_Q \rho(\bar \sigma_+)\overline u_+ u \varphi + \int_Q \ControlC u_+ \overline\sigma \varphi= \alpha\int_Q \rho(\bar\sigma_+)\overline u_+ \varphi + \int_Q  \overline u_+\nabla \sigma \cdot\nabla \varphi,\\\displaystyle
    \partial_t \sigma - \Delta \sigma =  \gamma(\beta - \overline \sigma_+) -\frac{\overline u_+\overline \sigma_+}{1+\overline \sigma_+} + (1-\ControlS)\overline u_+, & \text{in }Q,\\
    \partial_{\mathbf{n}} \sigma = 0, & \text{on }\Sigma,\\
    u(0) = u_0, \sigma(0) = \sigma_0. & \text{in }\Omega,
\end{cases}
\end{equation}
for all $\varphi \in L^2(H^1)$. In general, we denote $\overline m_+ := \max\{\overline m, 0\} \geq 0$. Next, we verify the hypotheses of Leray–Schauder Fixed Point Theorem.\\

\noindent
\textit{\textbf{Step 1:} The operator $\Gamma : \mathcal{X}_u \times \mathcal{X}_\sigma \to \mathcal{X}_u \times \mathcal{X}_\sigma$ is well-defined and compact.}\\

Since $\ControlS \in L^{\infty}(Q)$, $\overline u \in L^{4-}(Q)$ and $\overline{\sigma} \in L^{\infty}(Q)$, it holds that $ \gamma(\beta - \overline \sigma_+) -\frac{\overline u_+\overline \sigma_+}{1+\overline \sigma_+} + (1-\ControlS)\overline u_+ \in L^{2+}(Q)$. Therefore, by using the maximal parabolic regularity (see \cite{FeireislNovotny}, Theorem 10.22), there exists a unique solution $\sigma \in X_{2+}$ of \eqref{eq:demSol-linealizado}$_2$-\eqref{eq:demSol-linealizado}$_3$ and $\sigma(0) = \sigma_0$. On the other hand, using that $\sigma \in X_{2+}$, it holds that $\nabla \sigma \in W_{2+} \hookrightarrow L^{4+}(Q)$; moreover, taking into account that $\overline u \in L^{4-}(Q)$, it follows that $-\rho(\bar\sigma_+) \overline u_+ - \ControlC\overline \sigma + \alpha\rho(\bar \sigma_+)\overline u_+ + \overline u_+ \nabla \sigma \in L^2(Q)$, therefore there exists a unique solution $u\in W_2$ of \eqref{eq:demSol-linealizado}$_1$. Consequently, $\Gamma$ is well-defined. The compactness of $\Gamma$ follows from the compact embedding from $W_2 \times X_{2+} $ 
%\hookrightarrow 
into $\mathcal{X}_u\times \mathcal{X}_\sigma$.\\

\noindent
\textit{\textbf{Step 2:} The set
$
G_\theta = \{[u,\sigma] \in W_2 \times X_{2+} : [u,\sigma] = \theta \,\Gamma[u,\sigma] \text{ for some } \theta \in [0,1]\},
$
is bounded in $\mathcal{X}_u \times \mathcal{X}_\sigma$ (independently of $\theta \in [0,1]$).}\\

Let $[u,\sigma] \in G_\theta$ for $\theta \in (0,1]$ (the case $\theta = 0$ is trivial). Then, the pair $[u,\sigma] \in W_2 \times X_{2+}$ satisfies the following system
\begin{equation}\label{eq:DemSolDebilFuerte-10}
\begin{cases}
        \displaystyle\int_0^T \langle \partial_t u, \varphi \rangle + \int_Q \nabla u \cdot \nabla \varphi + \int_Q \rho(\sigma)u_+u\varphi + \int_Q \ControlC u_+ \sigma\varphi  = \theta \alpha\int_Q \rho(\sigma) u_+\varphi + \int_Q u_+ \nabla \sigma \cdot \nabla \varphi,\\
        \displaystyle\partial_t \sigma - \Delta \sigma = \theta\gamma(\beta-\sigma_+) - \theta\frac{u_+\sigma_+}{1+\sigma_+} + \theta(1-\ControlS) u_+, & \text{in }Q,\\
        \partial_{\mathbf{n}} \sigma = 0, & \text{on }\Sigma,\\
    u(0) = u_0, \sigma(0) = \sigma_0, & \text{in }\Omega.
\end{cases}
\end{equation}
for all $\varphi \in L^2(H^1)$. The boundedness of $[u,\sigma]$ is obtained in the following steps.\\

\noindent
\textit{\textbf{Step 2.1:} Non-negativity: $u,\sigma \geq 0$}.\\

Taking $\varphi = u_{-} := \min\{u, 0\} \leq 0$ in \eqref{eq:DemSolDebilFuerte-10}$_1$ (that is possible because $u \in L^2(H^1)$), and considering that $u_- = 0$ if $u \geq 0$, $\nabla u_- = \nabla u$ if $u \leq 0$, and $\nabla u_- = 0$ if $u > 0$, we have
$$
\frac{1}{2}\frac{d}{dt}\|u_-\|^2_{L^2} + \|\nabla u_-\|^2_{L^2} + \int_\Omega \rho(\sigma)u_+ u u_- + \int_\Omega \ControlC u_+ \sigma u_- = \theta\alpha\int_\Omega \rho(\sigma)u_+u_-+ \int_\Omega u_+ \nabla \sigma \cdot \nabla u_-  = 0.
$$
Thus $u_- \equiv 0$ and, consequently, $u \geq 0$ a.e. in $Q$. Similarly, since $0 \leq \ControlS \leq 1$, then multiplying \eqref{eq:DemSolDebilFuerte-10}$_2$ by $\sigma_- = \min \{\sigma,0\} \leq 0$ and integrating on $\Omega$, it holds
$$
\frac{1}{2}\frac{d}{dt}\|\sigma_-\|^2_{L^2} + \|\nabla \sigma_-\|^2_{L^2} =  \theta\gamma \int_\Omega (\beta-\sigma_+)\sigma_- - \theta \int_\Omega \frac{u_+\sigma_+}{1+\sigma_+}\sigma_- + \theta \int_\Omega (1-\ControlS) u_+ \sigma_-\leq 0,
$$
which implies $\sigma_- \equiv 0$ and then $\sigma \geq 0$ a.e. in $Q$. \\

\noindent
\textit{\textbf{Step 2.2:} Boundedness of $u$ in $L^2(Q).$}\\

Taking $\varphi = 1$ in \eqref{eq:DemSolDebilFuerte-10}$_1$, we obtain
\begin{equation}\label{eq:Teo1_nuevaEsti}
    \frac{d}{dt}\int_\Omega u + \int_\Omega \rho(\sigma)u^2 + \int_\Omega \ControlC u \sigma = \alpha \int_\Omega \rho(\sigma) u \leq \frac{1}{2}\int_\Omega \rho(\sigma)u^2 + C\int_\Omega \rho(\sigma).
\end{equation}
Integrating \eqref{eq:DemSolDebilFuerte-10}$_2$  on $\Omega$ and using that $0\leq \ControlS \leq 1$ we get
\begin{equation}\label{eq:Teo1_nuevaEsti2}
\frac{d}{dt}\int_\Omega \sigma 
%+ \gamma \int_\Omega \sigma 
+ \int_\Omega \frac{u\sigma}{1+\sigma} \leq \gamma \beta |\Omega| + \int_\Omega u.
\end{equation}
Adding \eqref{eq:Teo1_nuevaEsti} and \eqref{eq:Teo1_nuevaEsti2}, we obtain
\[
\frac{d}{dt}\left(\int_\Omega u +\int_\Omega \sigma\right) 
+ \frac12\int_\Omega \rho(\sigma) u^2 
%+ \int_\Omega \ControlC u \sigma 
%+ \gamma \int_\Omega \sigma 
%+ \int_\Omega \frac{u\sigma}{1+\sigma} 
\leq \gamma \beta |\Omega| + \int_\Omega u 
%+ \frac{1}{2}\int_\Omega \rho(\sigma)u^2 
+ C\int_\Omega \rho(\sigma).
\]
Applying the Gronwall inequality and using that $\rho(\sigma) \leq \mu_1 \sigma + \mu_2$, $\rho(\sigma)\geq \rho_{min}>0$ and $u,\sigma\geq 0,$ then
\[
u, \sigma \in L^\infty(L^1), \quad u \in L^2(Q).
\]

\noindent
\textit{\textbf{Step 2.3:} Boundedness of $\sigma$ in $ L^\infty(H^1) \cap L^2(H^2)$.} \\

Multiplying \eqref{eq:DemSolDebilFuerte-10}$_2$ by $\sigma-\Delta \sigma,$ integrating by parts on $\Omega$ and using that $\|z\|_{H^2} \leq C\, \|z-\Delta z\|_{L^2}$ and that $\theta \in (0,1]$, we obtain
\[
\begin{split}
\frac{1}{2}\frac{d}{dt}\|\sigma\|^2_{H^1} + C\|\sigma\|_{H^2} &= \theta \gamma \beta \! \int_\Omega (\sigma-\Delta \sigma) - \theta \gamma \int_\Omega \sigma (\sigma-\Delta\sigma) - \theta \int_\Omega \frac{u\sigma}{1+\sigma}(\sigma-\Delta \sigma) + \!\theta\! \int_\Omega (1-\ControlS)u(\sigma-\Delta \sigma) \\
& \leq 4\delta \|\sigma\|^2_{H^2} + C_\delta \Big(\gamma ^2 \beta ^2 |\Omega| + \gamma^2\|\sigma\|^2_{L^2} + \|u\|^2_{L^2} + \|1-\ControlS\|^2_{L^\infty}\|u\|^2_{L^2}\Big),
\end{split}
\]
taking $\delta$ small enough, we obtain that $\sigma$ is bounded in $L^\infty(H^1)\cap L^2(H^2).$ \\

\noindent
\textit{\textbf{Step 2.4:} Boundedness of $u$ in $L^\infty(L^2) \cap L^2(H^1)$.} \\

Testing \eqref{eq:DemSolDebilFuerte-10}$_1$ by $u$, using the H\"older and Young inequalities, we get
\[
\begin{split}
    \frac{1}{2}\frac{d}{dt}\|u\|^2_{L^2}+\|\nabla u \|^2_{L^2} & + \int_\Omega \rho(\sigma)u^3 + \int_\Omega \ControlC u^2 \sigma = \alpha \int_\Omega \rho(\sigma)u^2 + \int_\Omega u \nabla \sigma \cdot \nabla u \\
    &\leq \frac{1}{2}\int_\Omega \rho(\sigma)u^3 + C\int_\Omega \rho(\sigma) + \int u\nabla \sigma \cdot \nabla u \\ 
    &\leq \frac{1}{2}\int_\Omega \rho(\sigma)u^3 + C\int_\Omega \rho(\sigma) + \|u\|_{L^4}\|\nabla \sigma\|_{L^4}\|\nabla u\|_{L^2} \\
    &\leq \frac{1}{2}\int_\Omega \rho(\sigma)u^3 + C\int_\Omega \rho(\sigma) + C\Big(\| u\|^{1/2}_{H^1}\|u\|^{1/2}_{L^2}+\|u\|_{L^2}\Big) \|\nabla u\|_{L^2}\|\nabla \sigma\|_{L^2} \\ 
    &\leq \frac{1}{2}\int_\Omega \rho(\sigma)u^3 + C\int_\Omega \rho(\sigma) + 2\delta \|\nabla u\|^2_{L^2}+C_\delta C^4 \|u\|^2_{L^2}\|\nabla \sigma\|^4_{L^4}+C_\delta C^2 \|u\|^2_{L^2}\|\nabla \sigma\|^2_{L^4}.
\end{split}
\]
Taking $\delta$ small enough, applying the Gronwall Lemma and using that $\rho(\sigma) \leq \mu_1 \sigma + \mu_2$ and \textit{Step 2.3}, we obtain that $u$ is bounded in $L^\infty(L^2)\cap L^2(H^1)$. \\

\noindent
\textit{\textbf{Step 2.5:} Boundedness of $\sigma$ in $L^\infty(W^{1+,2+})\cap L^{2+}(W^{2,2+})$.} \\

By interpolation, using \textit{Steps 2.4} and \textit{2.5}, we also obtain that $\sigma \in L^{\infty-}(Q)$ and $u \in L^{4}(Q)$. Consequently, $\gamma\!\left(\beta-\sigma\right) -\frac{u\sigma}{1+\sigma} +(1-\ControlS)u \in L^{2+}(Q).$ Then, by the regularity result for the heat equation with Neumann boundary conditions, it follows that $\sigma \in X_{2+}$, with the corresponding estimate in $X_{2+}$ depending on $\|\sigma_0\|_{W^{1+,2+}(\Omega)}$.\\

\noindent
\textit{\textbf{Step 3:} The operator $\Gamma:\mathcal{X}_u \times \mathcal{X}_\sigma \to \mathcal{X}_u \times \mathcal{X}_\sigma$ defined in \eqref{eq:demSol-linealizado}, is continuous.}\\

Consider a sequence $\{[\overline u^l, \overline \sigma^l]\}_{l \in \mathbb{N}} \subset \mathcal{X}_u \times \mathcal{X}_{\sigma}$ such that
\begin{equation}\label{eq:conti-1}
    [\overline u^l, \overline \sigma^l]\to [\overline u, \overline \sigma] \quad \text{in } \mathcal{X}_u \times \mathcal{X}_{\sigma}, \quad \text{as }l \to +\infty.
\end{equation}
Since the sequence $\{[\overline u^l, \overline \sigma^l]\}_{l\in \mathbb{N}}$ is bounded in $\mathcal{X}_u \times \mathcal{X}_{\sigma}$, it follows from \textit{Step 2.4 - Step 2.5} that the corresponding sequence $[u^l, \sigma^l] = \Gamma [\overline u^l, \overline \sigma^l]$ is bounded in $W_2 \times X_{2+}$. By the compact embedding $W_{2} \times X_{2+} \hookrightarrow \mathcal{X}_u \times \mathcal{X}_{\sigma}$, there exists a subsequence (still denoted by $[u^l, \sigma^l]$) and a limit pair $[\hat u, \hat \sigma] \in W_2\times X_{2+}$ such that
\begin{equation}\label{eq:conti-2}
\Gamma[\overline u^l, \overline \sigma^l] \to [\hat u,\hat \sigma] \quad \text{weakly in }W_2 \times X_{2+}\quad  \text{and} \quad \text{strongly in } \mathcal{X}_u \times \mathcal{X}_{\sigma}.
\end{equation}
Now, combining \eqref{eq:conti-1} and \eqref{eq:conti-2}, we may pass to the limit as $l\to \infty$ in equation \eqref{eq:demSol-linealizado} (with $[\overline u^l,\overline \sigma^l]$ and $[u^l, \sigma^l]$ in place of $[\overline u, \overline \sigma]$ and $[u,\sigma]$, respectively) yields $[\hat u, \hat \sigma] = \Gamma [\overline u, \overline \sigma]$. Consequently, $\{\Gamma [\overline u^l, \overline \sigma^l]\}_{l\in \mathbb{N}}$ converges strongly in $\mathcal{X}_u \times \mathcal{X}_{\sigma}$ to $\Gamma[\overline u, \overline \sigma]$. Hence, the operator $\Gamma$ is continuous on $L^{4-}(Q) \times L^\infty(Q)$. Finally, the Leray-Schauder Fixed Point theorem guarantees the existence of a weak-strong solution $u \in W_2$ $\sigma \in X_{2+}$ of system (\ref{eq:SistemaControlado}).\\

\noindent
Now, we consider the case where we only have $\sigma_0\in H^1(\Omega).$ For any $\varepsilon>0,$ let $\sigma_0^\varepsilon\in W^{1+,2+}(\Omega),$ 
a regularization of $\sigma_0$ such that $\sigma_0^\varepsilon\rightarrow \sigma_0$ in $H^1(\Omega),$ and then, we regularize system (\ref{eq:SistemaControlado}) replacing $[u,\sigma]$ by $[u^\varepsilon, \sigma^\varepsilon],$ with initial data $u^\epsilon(0)=u_0$ and $\sigma^\varepsilon(0)=\sigma^\varepsilon_0.$ Then, for each $\varepsilon>0,$ from the first part of the proof, we obtain the existence of a solution $[u^\varepsilon,\sigma^\varepsilon]\in W_2\times X_{2+},$ of the regularized system, which satisfies the following $\varepsilon $-uniform estimates:
% with respect to the parameter $:$
 \begin{equation}\label{uniform_estimates}
\begin{cases}
        \{u^\varepsilon\}\ \mbox{bounded in}\ L^2(Q),\\
        \{ \sigma^\varepsilon \}\ \mbox{bounded in}\  L^\infty(H^1) \cap L^2(H^2)\},\\
        \{ u^\varepsilon \}\ \mbox{bounded in}\  L^\infty(L^2) \cap L^2(H^1)\}.
\end{cases}
\end{equation}
From (\ref{uniform_estimates}) we deduce that there exists limit functions $(u,\sigma)\in W_2\times X_2$ such that, for some subsequence of $\{[u^\varepsilon,\sigma^\varepsilon]\}_{\varepsilon>0},$ still denoted by $\{[u^\varepsilon,\sigma^\varepsilon]\}_{\varepsilon>0},$ the following convergences holds as $\varepsilon\rightarrow 0,$
\begin{eqnarray}\label{e1}
[u^\varepsilon,\sigma^\varepsilon]\rightarrow\ [u,\sigma]\ \mbox{weakly in}\ W_2\times X_2.
\end{eqnarray}
It holds that $[u,\sigma]$ is a weak-strong solution of \eqref{eq:SistemaControlado}. From (\ref{uniform_estimates}) it holds that $\{ u^\varepsilon \nabla\sigma^\varepsilon\}$ is bounded in $L^2(Q),$ and there exists $z$ such that $u^\varepsilon \nabla\sigma^\varepsilon\rightarrow z$ weakly in $L^{2-}(Q).$ Also, from (\ref{e1}), $\{ u^\varepsilon\}_{\varepsilon>0}$ is relatively compact in $L^q(Q),$ for all $q<2,$ and $\nabla\sigma^\varepsilon\rightarrow \nabla\sigma$ in $L^4(Q).$ Thus, by uniqueness of the limit, $z=u\nabla\sigma,$ and $u^\varepsilon \nabla\sigma^\varepsilon\rightarrow u\nabla\sigma$ in $L^2(Q).$ A standard limit argument allows us to pass to the limit as $\varepsilon\rightarrow 0,$ in the remaining terms, which finishes the proof of the existence. The uniqueness is consequence of Theorem \ref{teo:estabilidad} below.
\end{proof}
\begin{theo}{\label{teo:estabilidad}}
    Let $[u_i, \sigma_i]$ be two solutions of \eqref{eq:SistemaControlado} corresponding to controls $[\ControlC_i, \ControlS_i] \in L^\infty(Q) \times L^\infty (Q)$, with $u_{0,i} \in L^2(\Omega)$, $\sigma_{0,i} \in H^1(\Omega)$, $u_{0,i} \geq 0$, $\sigma_{0,i} \geq 0$ a.e. in $\Omega$, $i=1,2$. Then the following estimate holds
    \[
    \|u_1-u_2\|_{W_2} + \|\sigma_1-\sigma_2\|_{X_{2}} \leq C(\|u_{0,1}-u_{0,2}\|_{L^2(\Omega)}+\|\sigma_{0,1}-\sigma_{0,2}\|_{H^1(\Omega)} + \|\ControlC_1 - \ControlC_2\|_{L^4(L^{4/3+})} + \|\ControlS_1 - \ControlS_2\|_{L^{4}(Q)}),
    \]
    where $C>0$ is a constant dependent on the norms of $u_1, u_2$ in $W_2$, $\sigma_2$ in $X_{2}$, and $\ControlC_i, \ControlS_i$ in $L^\infty(Q)$, $i=1,2$. 
\end{theo}
\begin{proof}
    Let $[u_1, \sigma_1], [u_2,\sigma_2] \in W_2 \times X_{2}$ be two solutions of system \eqref{eq:SistemaControlado}. Defining their difference $[u,\sigma] = [u_1-u_2, \sigma_1-\sigma_2]$ and $[\ControlC, \ControlS] = [\ControlC_1-\ControlC_2, \ControlS_1- \ControlS_2]$ 
    %and  subtracting the equations satisfied by each pair yields 
    the following system holds
    \begin{equation}\label{eq:unicidad_1}
        \begin{cases}
            \displaystyle\int_0^T \langle \partial_t u, \varphi \rangle - \int_Q \nabla u \cdot \nabla \varphi = \alpha\int_Q u\rho(\sigma_1)\varphi + \alpha \int_Q u_2(\rho(\sigma_1)-\rho(\sigma_2))\varphi \\ \qquad\qquad\qquad \displaystyle- \int_Q u(u_1+u_2)\rho(\sigma_1)\varphi+\int_Q u_2^2(\rho(\sigma_1)-\rho(\sigma_2))\varphi - \int_Q \ControlC u_1\sigma_1\varphi \\ \qquad\qquad\qquad \displaystyle- \int_Q \ControlC_2 u \sigma_1\varphi  - \int_Q \ControlC_2 u_2 \sigma \varphi + \int_Q u_1\nabla \sigma \cdot \nabla \varphi+ \int_Q u\nabla\sigma_2 \cdot \nabla \varphi, & \forall \varphi \in L^2(H^1),\\
            \displaystyle\partial_t \sigma - \Delta \sigma = -\frac{u\sigma_1}{1+\sigma_1} - u_2\left(\frac{\sigma_1}{1+\sigma_1} - \frac{\sigma_2}{1+\sigma_2}\right) - \gamma \sigma - \ControlS u_1 + (1-\ControlS_2)u, & \text{in }Q,\\
            \partial_{\mathbf{n}}u = \partial_{\mathbf{n}}\sigma = 0, & \text{on }\Sigma, \\
            u(0) = u_0, \sigma(0) = \sigma_0, & \text{in }\Omega.
        \end{cases}
    \end{equation}
     Testing \eqref{eq:unicidad_1}$_1$ by $u$ and \eqref{eq:unicidad_1}$_2$ by $\sigma-\Delta \sigma$, using that $\|z\|_{H^2} \leq C\, \|z-\Delta z\|_{L^2}$, we obtain
    \[
    \begin{split}
    \frac{1}{2}&\frac{d}{dt}\Big(\|u\|^2_{L^2} + \|\sigma\|^2_{H^1}\Big)  + \|\nabla u\|_{L^2}^2 + C\|\sigma\|_{H^2}^2 = \alpha(u\rho(\sigma_1),u)+\alpha(u_2(\rho(\sigma_1)-\rho(\sigma_2)),u) \! - \! (u(u_1+u_2)\rho(\sigma_1),u) \\ & +(u^2_2 (\rho(\sigma_1)-\rho(\sigma_2)), u) -(\ControlC u_1 \sigma_1, u) -(\ControlC_2 u \sigma_1, u) - (\ControlC_2 u_2 \sigma, u) + (u_1\nabla \sigma, \nabla u) \\ & + (u\nabla \sigma_2, \nabla u) -\Big(\frac{u\sigma_1}{1+\sigma_1},\sigma-\Delta \sigma\Big)- \Big(u_2\left(\frac{\sigma_1}{1+\sigma_1} - \frac{\sigma_2}{1+\sigma_2}\right), \sigma-\Delta \sigma\Big) - \gamma \Big(\sigma,\sigma-\Delta \sigma\Big)\\& - (\ControlS u_2, \sigma-\Delta \sigma) + ((1-\ControlS_2)u, \sigma-\Delta\sigma).
    \end{split}
    \]
    Applying the H\"older and Young inequalities, using the 2D interpolation estimates, and that $\rho$ is bounded by a linear function, we obtain
\begin{equation}\label{eq:unicidad_2}
\begin{split}
\alpha(u\rho(\sigma_1),u) &\leq \delta \|u\|^2_{H^1} + C_\delta \alpha^2 \|\rho(\sigma_1)\|^2_{L^{2+}}\|u\|^2_{L^2}, \\
(u(u_1+u_2)\rho(\sigma_1), u) & \leq \delta \|u\|^2_{H^1} + C_\delta \|u_1+u_2\|^2_{L^{2+}}\|\rho(\sigma_1)\|^2_{L^{\infty-}}\|u\|^2_{L^2}.
\end{split}
\end{equation}
Likewise, using the fact that $\rho$ is Lipschitz, we get
\begin{equation}\label{eq:unicidad_2_1}
\begin{split}
(\alpha u_2(\rho(\sigma_1)-\rho(\sigma_2)),u) &\leq \delta \|\sigma\|^2_{H^2} + C_\delta \alpha \|u_2\|^2_{L^2}\|u\|^2_{L^2}, \\
(u_2^2(\rho(\sigma_1)-\rho(\sigma_2)),u) & \leq \delta \|\sigma\|^2_{H^2} + C_\delta C\|u_2\|^4_{L^4}\|u\|^2_{L^2}.
\end{split}
\end{equation}
Similarly, it is easy to see that
\begin{equation}\label{eq:unicidad_2_3}
\begin{split}
(\ControlC u_1 \sigma_1, u) &\leq \delta \|u\|^2_{H^1} + C_\delta \|u_1\|^2_{L^4}\|\sigma_1\|^2_{L^{\infty-}}\|\ControlC\|^2_{L^{4/3+}},\\
(\ControlC_2 u \sigma_1,u) &\leq  \delta \|u\|^2_{H^1} + C_\delta \|\ControlC_2\|^2_{L^\infty}\|\sigma_1\|^2_{L^{2+}}\|u\|^2_{L^2},\\
(\ControlC_2 u_2 \sigma, u) & \leq \delta \|\sigma\|^2_{H^2} + C_\delta \|\ControlC_2\|^2_{L^\infty}\|u_2\|^2_{L^2}\|u\|^2_{L^2}.
\end{split}
\end{equation}
Additionally, we obtain
\begin{equation}\label{eq:unicidad_6}
\begin{split}
&(u_1\nabla \sigma, \nabla u)  \leq C\|u_1\|_{L^4}\|\nabla \sigma\|^{1/2}_{L^2}\|\nabla \sigma\|^{1/2}_{H^1}\|\nabla u\|_{L^2} \leq \delta (\|\nabla \sigma\|^2_{H^1} + \|\nabla u\|^2_{L^2})+ C_\delta \|u_1\|^4_{L^4}\|\nabla \sigma\|^2_{L^2}, \\
&(u\nabla \sigma_2, \nabla u) \leq C\|u\|^{1/2}_{L^2}\|u\|^{1/2}_{H^1}\|\nabla \sigma_2\|_{L^4}\|\nabla u\|_{L^2} \leq \delta \|u\|^2_{H^1} + C_\delta \|\nabla \sigma_2\|^4_{L^4}\|u\|^2_{L^2.}
\end{split}
\end{equation}
Now, taking account the global Lipschitz inequality $\left|\frac{\sigma_1}{1+\sigma_1}-\frac{\sigma_2}{1+\sigma_2}\right|\leq |\sigma_1-\sigma_2| = |\sigma|$, then
\begin{equation}\label{eq:unicidad_7}
\begin{split}
    \Big(u_2\Big(\frac{\sigma_1}{1+\sigma_1} - \frac{\sigma_2}{1+\sigma_2}\Big), \sigma-\Delta\sigma\Big) & \leq \|\sigma\|_{H^2}\|\sigma\|_{L^4}\|u_2\|_{L^4} \leq \delta \|\sigma\|^2_{H^2} + C_\delta \|\sigma\|^2_{H^1}\|u_2\|^2_{L^4},
     \end{split}
\end{equation}
and using that $\frac{|x|}{1+|x|} \leq 1$, we have
\begin{equation}\label{eq:unicidad_9}
    \Big(\frac{u\sigma_1}{1+\sigma_1},\sigma-\Delta \sigma\Big) \leq \delta \|\sigma\|_{H^2}^2+C_\delta \|u\|^2_{L^2}.
\end{equation}
Moreover, 
\begin{equation}\label{eq:unicidad_10}
    (\ControlS u_1, \sigma- \Delta \sigma) \leq \delta \|\sigma\|_{H^2}^2 + C_\delta \|u_1\|^2_{L^4}\|\ControlS\|^2_{L^4},
\end{equation}
\begin{equation}\label{eq:unicidad_11}
    ((1-\ControlS_2) u, \sigma - \Delta \sigma) \leq \delta \|\sigma\|^2_{H^2} + C_\delta \|1-\ControlS_1\|^2_{L^\infty}\|u\|^2_{L^2}.
\end{equation}
Therefore, using the estimates \eqref{eq:unicidad_2}-\eqref{eq:unicidad_11} in \eqref{eq:unicidad_1}, and since $u_1, u_2 \in W_2 \hookrightarrow L^4(Q)$, $\nabla \sigma_2 \in W_{2}\hookrightarrow L^{4}(Q)$, $\ControlC_i, \ControlS_i \in L^\infty(Q)$, $i=1,2$, taking $\delta > 0$ small enough, and using the Gronwall inequality  and a comparison argument to obtain the time-regularity, we can obtain
\[
\|u\|^2_{W_2} + \|\sigma\|^2_{X_2} \leq C(\|u_{0,1}-u_{0,2}\|_{L^2(\Omega)}^2+\|\sigma_{0,1}-\sigma_{0,2}\|^2_{W^{1,2}(\Omega)}+\|\ControlC\|^2_{L^4(L^{4/3+})}+\|\ControlS\|^2_{L^4(Q)}),
\]
completing the proof. 
\end{proof}

%%%%%%%%%%%%%%%%%%%%%%%%%%%%%%%%%%%%%%%%%%%%%%%%%%%%%%%%%%%%%
%%%%%%%%%%%%%%%%%%%%%%%%%%%%%%%%%%%%%%%%%%%%%%%%%%%%%%%%%%%%%
% CAPÍTULO 3: EXISTENCIA DE SOLUCIÓN OPTIMA GLOBAL Y CONDICIONES DE OPTIMALIDAD
%%%%%%%%%%%%%%%%%%%%%%%%%%%%%%%%%%%%%%%%%%%%%%%%%%%%%%%%%%%%%
%%%%%%%%%%%%%%%%%%%%%%%%%%%%%%%%%%%%%%%%%%%%%%%%%%%%%%%%%%%%%

\section{Existence of Global Optimal Solution and Optimality Conditions}

\subsection{Existence of Global Optimal Solution. Proof of Theorem \ref{exis2}.}
In this subsection, we address the proof of the Theorem \ref{exis2} related to the existence of solution for the optimal control problem \eqref{eq:ProblemaDeControl}.
\begin{proof}

    From Theorem \ref{teoSolFuerte}, it follows that the set $\mathcal{P}_{ad}$ is nonempty. Moreover, due to the nonnegativity of the variables, the functional $J$ is bounded from below. Thus, there exists a minimizing sequence  $ \{[u_n, \sigma_n,\ControlC_n, \ControlS_n]\}_{n\in \mathbb{N}} \subset \mathcal{P}_{ad}$ such that
    $$
    \lim_{n\to\infty} J([u_n, \sigma_n,\ControlC_n, \ControlS_n]) = \inf_{[u,\sigma,\ControlC,\ControlS] \in \mathcal{P}_{ad}} J([u,\sigma, \ControlC, \ControlS]).
    $$
    For each $n\in\mathbb{N},$ the pair $[u_n, \sigma_n]$ is the weak-strong solution of system \eqref{eq:SistemaControlado} with data  $[\ControlC_n, \ControlS_n] \in \mathcal{U}_{ad}$. Consequently, since $[\ControlC_n, \ControlS_n]$ is bounded in $L^\infty(Q)$ then $[u_n, \sigma_n]$ is bounded in $W_2\times X_2$, hence we can extract a subsequence of $\{[u_n, \sigma_n, \ControlC_n, \ControlS_n]\}_{n\geq 1}$, still denoted by itself, such that
    \begin{equation}\label{eq:DemExistencia - convergencia1}
        [\ControlC_n, \ControlS_n] \overset{*}{\rightharpoonup} [\ControlC^*, \ControlS^*] \quad \text{weak-* in }  (L^\infty(Q))^2,
    \end{equation}
    \begin{equation}\label{eq:DemExistencia - convergencia2}
        [u_n,\sigma_n] \rightharpoonup [u^*,\sigma^*] \quad \text{weakly in } L^2(H^1)\times L^{2}(H^2),
    \end{equation}
    \begin{equation}\label{eq:DemExistencia - convergencia3}
        [u_n,\sigma_n] \overset{*}{\rightharpoonup} [u^*,\sigma^*] \quad \text{weak-* in } L^\infty(L^2 \times H^1),
    \end{equation}
    \begin{equation}\label{eq:DemExistencia - convergencia4}
        [\partial_t u_n, \partial_t\sigma_n] \rightharpoonup [\partial_t u^*, \partial_t \sigma^*] \quad \text{weakly in } L^2((H^1)')\times L^{2}(Q).
    \end{equation}
    By the Aubin-Lions Lemma (see \cite{16}, Theorem 5.1, p.58) and using the Corollary 4 of \cite{23}, we have that
    \begin{equation}\label{eq:DemExistencia - convergencia5}
        [u_n,\sigma_n] \to [u^*, \sigma^*] \quad \text{strongly in } C^0([0,T]; ((H^1)'\times L^2)),
    \end{equation}
    \begin{equation}\label{eq:DemExistencia - convergencia6}
        \sigma_n \to \sigma^* \quad \text{strongly in }L^{\infty-}(Q),
    \end{equation}
    \begin{equation}\label{eq:DemExistencia - convergencia7}
        u_n \to u^* \quad \text{strongly in } L^{4-}(Q),
    \end{equation}
    \begin{equation}\label{eq:DemExistencia - convergencia8}
        \nabla \sigma_n \to \nabla \sigma^* \quad \text{strongly in }L^{4-}(Q).
    \end{equation}
    Since for every $n$ one has $u_n(0) = u_0$ and $\sigma_n(0) = \sigma_0$, it follows directly that the limit functions satisfy $u^*(0) = u_0$ and $\sigma^*(0)=\sigma_0$. Indeed, by convergence \eqref{eq:DemExistencia - convergencia5}, the sequence $(u_n(0), \sigma_n(0))$ approaches $(u^*(0), \sigma^*(0))$ in the space $(H^1(\Omega))' \times L^2(\Omega)$. Consequently, $[u^*,\sigma^*]$ satisfies the initial condition in \eqref{eq:SistemaControlado}.
    Since $[\ControlC_n, \ControlS_n] \in \mathcal{U}_{ad}$, they are bounded in $[L^2(Q)]^2$, and hence
    \begin{equation}\label{eq:DemExistencia - convergencia9}
    [\ControlC_n, \ControlS_n] \rightharpoonup [\ControlC^*, \ControlS^*] \quad \text{weakly in } L^2(Q).
    \end{equation}
    Note that $\mathcal{U}_{ad}$ is a convex and closed subset of $[L^2(Q)]^2$, and thus weakly closed, implying $[\ControlC^*, \ControlS^*] \in \mathcal{U}_{ad}$. 
    \\ \\
    Taking into account that $\left|\frac{x}{1+x}-\frac{y}{1+x}\right|\leq |x-y|$ and using that $\sigma_n \rightarrow \sigma^*$ in $L^{\infty-}(Q)$, $\nabla \sigma_n \rightarrow \nabla\sigma^*$ weakly in $L^{4}(Q)$, $u_n \rightarrow u^*$ strongly in $L^{4-}(Q)$ and $u_n \nabla \sigma_n \in L^2(Q)$, then
     \begin{equation}\label{eq:DemExistencia - convergencia10}
     u_n  \nabla \sigma_n \to u^* \nabla \sigma^*, \quad \text{weakly in }L^{2}(Q),
     \end{equation}
     \begin{equation}\label{eq:DemExistencia - convergencia11}
     \frac{\sigma_nu_n}{1+\sigma_n} \to \frac{\sigma^* u^*}{1+\sigma^*} \quad \text{weakly in }L^{4-}(Q).
     \end{equation}
    Passing to the limit  as $n \to \infty$ and using \eqref{eq:DemExistencia - convergencia1}-\eqref{eq:DemExistencia - convergencia11}, it follows that $[u^*, \sigma^*]$ solves \eqref{eq:SistemaControlado} associated to $[\ControlC^*, \ControlS^*]$. Consequently,
    $$
    \lim_{n\to \infty} J([u_n, \sigma_n, \ControlC_n, \ControlS_n]) = \inf_{[u,\sigma, \ControlC, \ControlS] \in \mathcal{P}_{ad}} J([u,\sigma, \ControlC, \ControlS]) \leq J([u^*, \sigma^*, \ControlC^*, \ControlS^*]).
    $$
    Since $J$ is weakly lower-semicontinuous on $\mathcal{P}_{ad}$, we obtain
    $$
    J([u^*, \sigma^*, \ControlC^*, \ControlS^*]) \leq \underset{n\to \infty}{\text{ lim inf }} J([u_n,\sigma_n, \ControlC_n, \ControlS_n]),
    $$
    and hence, we conclude that
    $$
    \inf_{[u,\sigma, \ControlC, \ControlS]\in \mathcal{P}_{ad}} J([u,\sigma, \ControlC, \ControlS]) \leq J([u^*,\sigma^*, \ControlC^*, \ControlS^*]) \leq \underset{n\to\infty}{\text{ lim inf }} J([u_n, \sigma_n, \ControlC_n, \ControlS_n]) = \inf_{[u,\sigma, \ControlC, \ControlS] \in \mathcal{P}_{ad}} J([u,\sigma,\ControlC,\ControlS]),
    $$
    implying that $J$ reaches a global minimum at $[u^*, \sigma^*,\ControlC^*, \ControlS^*]$.
    
\end{proof}

\subsection{Linearized and Adjoint Systems}

In this subsection, we derive necessary first-order optimality conditions for a local optimal solution $[u^*,\sigma^*,\ControlC^*,\ControlS^*] \in \mathcal{P}_{ad}$ of Problem \eqref{eq:ProblemaDeControl}. We introduce the control-to-state mapping
\[
G : \mathcal{U}_{ad} \subset [L^\infty (Q)]^2 \to W_2 \times X_{2},
\]
defined by $G([\ControlC, \ControlS]) = [u(\ControlC,\ControlS), \sigma(\ControlC, \ControlS)]$, where $[u(\ControlC,\ControlS), \sigma(\ControlC, \ControlS)]$ is the (unique) solution of the controlled system \eqref{eq:SistemaControlado} corresponding to the controls $[\ControlC, \ControlS]\in \mathcal{U}_{ad}$. From Theorem \ref{teoSolFuerte} it follows that the mapping $G$ is well-defined. Moreover, from Theorem \ref{teo:estabilidad}, we have the stability of solutions 
$G([\ControlC, \ControlS]) = [u(\ControlC,\ControlS), \sigma(\ControlC, \ControlS)]$
 with respect to the control $[\ControlC,\ControlS]$. Next proposition provides the directional derivative in any point $[\ControlC^*, \ControlS^*]\in \mathcal{U}_{ad}$ in any direction towards  the convex  $[\bar \ControlC-\ControlC^*, \bar \ControlS -\ControlS^*]$ of the control-to-state mapping $G$, for any $[\overline\ControlC, \overline\ControlS]\in \mathcal{U}_{ad}$.
\begin{prop}
    Under the hypotheses of Theorem  \ref{teoSolFuerte}, the directional derivative $G'(\ControlC^*, \ControlS^*)[\bar \ControlC-\ControlC^*, \bar \ControlS-\ControlS^*]= [z_1,z_2]$ exists in $W_2 \times X_{2}$, with $[z_1,z_2]$ the unique solution of the following (sensitivity) problem
    \begin{equation}\label{eq:SistemaLinealizado}
    \begin{cases}
    \displaystyle\int_0^T \langle \partial_t z_1, \varphi \rangle 
    +\int_Q \nabla z_1 \cdot \nabla \varphi = \int_Q \rho(\sigma^*)(\alpha - 2u^*)z_1\varphi + \int_Q \rho'(\sigma^*)u^*(\alpha - u^*)z_2\varphi
     \\ \quad
      \displaystyle -\int_Q\sigma^* \ControlC^*z_1\varphi
       -\int_Qu^*\ControlC^* z_2 \varphi 
       + \int_Q u^* \nabla z_2 \cdot \nabla \varphi + \int_Q z_1 \nabla \sigma^* \cdot \nabla \varphi
       - \int_Q u^*\sigma^*(\bar\ControlC-\ControlC^*)\varphi , 
       & \forall \varphi \in L^2(H^1), \\
    \displaystyle\partial_t z_2 -\Delta z_2 = -\frac{z_1\sigma^*}{1+\sigma^*}- \frac{u^*z_2}{(1+\sigma^*)^2} - \gamma z_2  + (1-\ControlS^*)z_1
    - (\bar \ControlS - \ControlS^*)u^*,
     & \text{ in }Q, \\
   % \partial_{\mathbf{n}}z_1 =
     \partial_{\mathbf{n}}z_2 = 0, & \text{ on }\Sigma, \\
    z_1(0) = 0, z_2(0) = 0, & \text{ in }\Omega.
    \end{cases}
    \end{equation}
\end{prop}
\begin{proof}
    We apply Theorem \ref{teo:Append_derivada} (to the controls  $[\ControlC^*,\ControlS^*]\in\mathcal{U}_{ad}$ and its associated state 
    $[u^*,\sigma^*]\in W_2\times X_2$), considering the nonlinearities
    \begin{align*}
    &f_1(u^*,\sigma^*)=\rho(\sigma^*)u^*(\alpha-u^*), && f_2(u^*,\sigma^*)=-u^*\sigma^*,\\
    & g_1(u^*,\sigma^*)=\gamma(\beta-\sigma^*)-\frac{u^*\sigma^*}{1+\sigma^*}+u^*, && 
    g_2(u^*,\sigma^*)=-u^*.
    \end{align*}
    Recalling that $\rho(\cdot)$ is bounded by a linear function, then it is straightforward  to check that $f_i,g_i\in C^2([0,\infty)^2)$ for $i=1,2$.  
    Using the regularity $u^*\in W_2$, $\sigma^*\in X_2$, it follows that the hypotheses of the Theorem \ref{teo:Append_derivada} are satisfied. Consequently, $G'(\ControlC^*, \ControlS^*)[\bar \ControlC-\ControlC^*, \bar\ControlS - \ControlS^*] = [w,z]$
    where $[w,z]$ solves the linearized system \eqref{eq:SistemaLinealizado}.
\end{proof}
Now we introduce the adjoint system associated to the linearized one. Let $[u^*, \sigma^*]$ solution of system \eqref{eq:SistemaControlado} associated with the optimal control $[\ControlC^*, \ControlS^*]$. Then the adjoint variables $[p_1, p_2]$ satisfy the following system
\begin{equation}\label{eq:sistemaAdjunto} 
\begin{cases} 
-\partial_t p_1 - \Delta p_1 - \rho(\sigma^*)(\alpha - 2u^*)p_1+ \sigma^* \ControlC^*p_1 - \nabla \sigma^* \cdot \nabla p_1 + \frac{\sigma^*p_2}{1+\sigma^*} - (1-\ControlS^*)p_2= k_1u^*,& \text{ in }Q, \\ 
-\partial_t p_2 - \Delta p_2 - \rho'(\sigma^*)u^*(\alpha - u^*)p_1 + u^*\ControlC^* p_1 + \nabla \cdot (u^* \nabla p_1) + \frac{u^* p_2}{(1+\sigma^*)^2} + \gamma  p_2 = k_2(\sigma^* - \sigma_Q),& \text{ in }Q,\\

\partial_{\mathbf{n}} p_1 = \partial_{\mathbf{n}}p_2 = 0, & \text{ on } \Sigma, \\ 
p_1(x,T) = l_1(x)u^*(x,T),\quad p_2(x,T) = l_2(x)(\sigma^*(x,T) - \sigma_\Omega), & \text{ in }\Omega,
\end{cases}
\end{equation}
where $[u^*, \sigma^*,\ControlC^*,\ControlS^*]$ is the optimal control-state and $k_j(x,t) \in L^\infty(Q)$, $(j=1,2,3,4)$ and $l_i(x) \in L^\infty(\Omega)$ $(i = 1,2)$ are the positive bounded weighting functions from the objetive functional \eqref{funcionalObjetivo}.
% \\ \\
%{\color{red}Assumption: $k_1, k_2 \in L^2(Q)$, $l_1 \in W^{1+,2+}(\Omega)$, $l_2 \in L^2(\Omega)$}
\begin{theo} Let $[u^*, \sigma^*, \ControlC^*, \ControlS^*] \in \mathcal{P}_{ad}$ be a local optimal solution of problem \eqref{eq:ProblemaDeControl}. Then 
\begin{enumerate}
    \item If $l_1(x) = 0$, there exists a unique solution $[p_1, p_2]$ in $X_2 \times W_2$ of the adjoint system \eqref{eq:sistemaAdjunto}.
    \item If $l_1(x) \neq 0$, there exists a unique solution $[p_1, p_2]$ in $W_2 \times (L^2(Q)\cap L^\infty((H^1)'))$ of \eqref{eq:sistemaAdjunto}.
\end{enumerate}
\end{theo}
\begin{proof}
    First assume $l_1(x)=0.$ By substituting $t$ by $T - t$, and setting $\eta_i(x,t)= p_i (x, T-t),$ $i= 1,2$ in $Q$, system \eqref{eq:sistemaAdjunto} can be rewritten as the following forward parabolic system:
    \begin{equation}\label{eq:sistemaAdjunto hacia adelante}
        \begin{cases}
            \partial_t \eta_1 - \Delta \eta_1 =   \rho(\sigma^*)(\alpha - 2u^*)\eta_1  - \sigma^*\ControlC^* \eta_1 + \nabla \sigma^*\cdot \nabla \eta_1 - \frac{\sigma^*}{1+\sigma^*}\eta_2 + (1 - \ControlS^*)\eta_2+k_1u^*,& \text{ in }Q, \\
            \partial_t\eta_2 - \Delta \eta_2 =  \rho'(\sigma^*)u^*(\alpha - u^*)\eta_1 - u^*\ControlC^* \eta_1 - \nabla \cdot (u^* \nabla \eta_1) - \frac{u^*}{(1+\sigma^*)^2}\eta_2 - \gamma \eta_2 + k_2(\sigma^* - \sigma_Q),&\text{ in }Q,\\
            \partial_{\mathbf{n}} \eta_1 = \partial_{\mathbf{n}}\eta_2 = 0, & \text{ on } \Sigma, \\
            \eta_1(x,0) = 0,\quad \eta_2(x,0) = l_2(x)(\sigma^*(x,T) - \sigma_\Omega), & \text{ in }\Omega.
        \end{cases}
    \end{equation}
    The existence of a solution to \eqref{eq:sistemaAdjunto hacia adelante} follows from the Leray–Schauder Fixed Point Theorem, applied to the operator
    \[
    \Gamma : L^{\infty-}(Q) \times L^{4-}(Q) \to X_{2} \times W_2 \hookrightarrow L^{\infty-}(Q) \times L^{4-}(Q),
    \]
    such that $\Gamma[\overline \eta_1, \overline \eta_2] = [\eta_1, \eta_2]$ is the solution of the following decoupled linear problem:
    \begin{equation}\label{eq_append: LemmaEq2}
         \begin{cases}
            \partial_t \eta_1 - \Delta \eta_1 =   \rho(\sigma^*)(\alpha - 2u^*)\eta_1  - \sigma^*\ControlC^* \eta_1 
            + \nabla \sigma^*\cdot \nabla \eta_1 - \frac{\sigma^*}{1+\sigma^*}\bar \eta_2 
            + (1 - \ControlS^*)\bar\eta_2+k_1u^*,& \text{ in }Q, \\
            \displaystyle\int_0^T \langle \partial_t\eta_2,\varphi\rangle 
            + \int_Q\nabla \eta_2 \cdot\nabla \varphi =  \int_Q\rho'(\sigma^*)u^*(\alpha - u^*)\eta_1\varphi 
            - \int_Q u^*\ControlC^* \eta_1\varphi 
            + \int_Qu^* \nabla \eta_1\cdot \nabla \varphi \\ \qquad\qquad\displaystyle- \int_Q\frac{u^*}{(1+\sigma^*)^2}\eta_2\varphi - \int_Q\gamma \eta_2\varphi + \int_Q k_2(\sigma^* - \sigma_Q)\varphi,&\forall \varphi \in L^2(H^1),\\
            \partial_{\mathbf{n}} \eta_1 
            %= \partial_{\mathbf{n}}\eta_2 
            = 0, & \text{ on } \Sigma, \\
            \eta_1(x,0) = 0,\quad \eta_2(x,0) = l_2(x)(\sigma^*(x,T) - \sigma_\Omega), & \text{ in }\Omega.
        \end{cases}
    \end{equation}
    Since $u^* \in W_{2} \hookrightarrow L^4(Q)$, $\eta_2 \in L^{4-}(Q)$, $\sigma^* \in X_{2},$ which gives that consequently $\nabla \sigma^* \in W_{2} \hookrightarrow L^{4}(Q)$, we apply Lemma \ref{lema_adjunto} with $f_1 = \rho(\sigma^*)(\alpha - 2u^*) - \sigma^* \ControlC^* \in L^4(L^{4-})$ , $f_2 = \frac{\sigma^*}{1+\sigma^*}\bar \eta_2 + (1-\ControlS^*)\bar \eta_2 + k_1u^* \in L^{4-}(Q)$ and $\mathbf{v} = \nabla \sigma^* \in L^4(Q)$, in order to obtain the existence and  uniqueness of solution $\eta_1 \in X_2$ of \eqref{eq_append: LemmaEq2}$_1$. Moreover, since $X_2 \hookrightarrow L^{\infty}(L^{\infty-})$, then $\rho'(\sigma^*)u^*(\alpha - u^*)\eta_1 - u^*\ControlC^*\eta_1 - u^*\nabla \eta_1 - \frac{u^*}{(1+\sigma^*)^2} - \gamma + k_2(\sigma^*-\sigma_Q) \in L^{2}(L^{2-})$; therefore, exists a unique solution $\eta_2 \in W_2$ of \eqref{eq_append: LemmaEq2}. Hence, the operator $\Gamma$ is well defined. In addition, since $X_2 \times W_2$ is compactly embedded into $L^{\infty-}(Q) \times L^{4-}(Q)$, it follows that $\Gamma$ is a compact operator.
    \\ \\
    Now we prove that the set 
    \[
    G_\theta = \{[\eta_1,\eta_2] \in  X_{2} \times W_2: [\eta_1,\eta_2] = \theta \,\Gamma[\eta_1,\eta_2] \text{ for some } \theta \in [0,1]\},
    \]
    of possible fixed points of $\theta\Gamma,$ is bounded in $L^{\infty-}(Q) \times L^{4-}(Q),$ uniformly with respect to $\theta$. Indeed, if $[\eta_1, \eta_2] \in G_\theta$, then $[\eta_1, \eta_2] \in X_{2} \times W_{2}$ satisfying the following system:
    \begin{equation}\label{eq:linealizadoAdjuntoPrueba}
        \begin{cases}
            \partial_t \eta_1 - \Delta \eta_1 =   \rho(\sigma^*)(\alpha - 2u^*)\eta_1  - \sigma^*\ControlC^* \eta_1 + \nabla \sigma^*\cdot \nabla \eta_1 - \theta\frac{\sigma^*}{1+\sigma^*} \eta_2 + \theta(1 - \ControlS^*)\eta_2+\theta k_1u^*,& \text{ in }Q, \\
            \displaystyle\int_0^T \langle \partial_t\eta_2,\varphi\rangle
             + \int_Q\nabla \eta_2 \cdot\nabla \varphi 
             =  \int_Q\rho'(\sigma^*)u^*(\alpha - u^*)\eta_1\varphi - \int_Q u^*\ControlC^* \eta_1\varphi 
             + \int_Qu^* \nabla \eta_1 \cdot\nabla \varphi 
             \\ \qquad\qquad\displaystyle
             - \int_Q\frac{u^*}{(1+\sigma^*)^2}\eta_2\varphi - \int_Q\gamma \eta_2\varphi 
             + \theta\int_Q k_2(\sigma^* - \sigma_Q)\varphi,&\forall \varphi \in L^2(H^1),\\
            \partial_{\mathbf{n}} \eta_1 
            %= \partial_{\mathbf{n}}\eta_2 
            = 0, & \text{ on } \Sigma, \\
            \eta_1(x,0) = 0,\quad \eta_2(x,0) = l_2(x)(\sigma^*(x,T) - \sigma_\Omega), & \text{ in }\Omega.
        \end{cases}
    \end{equation}
    Now, multiplying \eqref{eq:linealizadoAdjuntoPrueba}$_1$ by $\eta_1 - \Delta \eta_1 \in L^2(Q),$ testing \eqref{eq:linealizadoAdjuntoPrueba}$_2$ by $\eta_2 \in L^2(H^1),$ using the estimates
    \[
        \begin{split}
        (u^*\nabla \eta_1, \nabla \eta_2) & \leq \delta \|\nabla \eta_2\|^2_{L^2} + C_\delta \|u^*\|^2_{L^4}\|\nabla \eta_1\|_{L^4}^2 \\ & \leq \delta \|\nabla \eta_2\|^2_{L^2} + C_\delta \|u^*\|^2_{L^4}\|\Delta\eta_1\|_{L^2}\|\nabla \eta_1\|_{L^2} \\
        & \leq \delta \|\nabla \eta_2\|^2_{L^2} + \delta \|\Delta \eta_1\|^2_{L^2} + C_\delta \|u^*\|^4_{L^2}\|\nabla \eta_1\|^2_{L^2},\\
        (\nabla \sigma^* \nabla \eta_1, \eta_1 - \Delta \eta_1) &\leq \delta \|\eta_1 - \Delta \eta_1\|^2_{L^2} + C_\delta \|\nabla \sigma^*\|^2_{L^4}\|\nabla \eta_1\|^2_{L^4} \\
        & \leq \delta \|\eta_1-\Delta\eta_1\|^2_{L^2}+\delta \|\Delta \eta_1\|^2_{L^2} + C_\delta \|\nabla \sigma^*\|^4_{L^4}\|\nabla \eta_1\|^2_{L^2},
        \end{split}
    \]
    and then, taking $\delta$ small enough, from the Gronwall Lemma we obtain that
    \[
    \begin{split}
    \|\eta_1\|_{X_2} + \|\eta_2\|_{W_2} \leq C\Big(\|l_2\|_{L^{\infty}(\Omega)}\|\sigma^*(T) - \sigma_\Omega\|_{L^2(\Omega)},\|u^*\|_{W_2}, \|\sigma^*\|_{X_2}, \|\ControlC^*\|_{L^\infty(Q)}, \|\ControlS^*\|_{L^\infty(Q)}, \|\sigma_Q\|_{L^{2+}(Q)}\Big).
    \end{split}
    \]
    Therefore, the Leray–Schauder Fixed Point Theorem implies the existence of solution $[\eta_1,\eta_2]\in X_{2}\times W_{2}$ of \eqref{eq:sistemaAdjunto} with $l_1(x) = 0$. The uniqueness follows from the linearity of \eqref{eq:sistemaAdjunto}.\\
    
    In the case $l_1(x)\neq 0,$ we only have $u^*(\cdot,T)\in L^2(\Omega),$ and thus, $p_1(\cdot,T)=\eta_1(\cdot,0)\in L^2(\Omega).$ Then, we can only obtain regularity $W_2$ for $\eta_1,$ which is not sufficient to control the 
    %right hand side of the 
    variational equation for $\eta_2.$ 
    Then, we consider the following "weak-very weak''  variational formulation of (\ref{eq:sistemaAdjunto hacia adelante}):
    \begin{equation}\label{very_weak}
         \begin{cases}
            \displaystyle\int_0^T\langle \partial_t\eta_1,\phi\rangle +\int_Q \nabla\eta_1\cdot \nabla\phi =   \int_Q \rho(\sigma^*)(\alpha - 2u^*)\eta_1\phi  - \int_Q\sigma^*\ControlC^* \eta_1\phi
            + \int_Q \phi\nabla \sigma^*\cdot \nabla \eta_1 - \int_Q\frac{\sigma^*}{1+\sigma^*} \eta_2\phi \\ \qquad \qquad \displaystyle
            +\int_Q (1 - \ControlS^*)\eta_2\phi+\int_Q k_1u^*\phi,\\
            \displaystyle\int_0^T \langle \partial_t\eta_2,\varphi\rangle 
            + \int_Q \eta_2(-\Delta \varphi) =  \int_Q\rho'(\sigma^*)u^*(\alpha - u^*)\eta_1\varphi 
            - \int_Q u^*\ControlC^* \eta_1\varphi 
            + \int_Qu^* \nabla \eta_1\cdot \nabla \varphi - \int_Q\gamma \eta_2\varphi \\ \qquad\qquad\displaystyle- \int_Q\frac{u^*}{(1+\sigma^*)^2}\eta_2\varphi  + \int_Q k_2(\sigma^* - \sigma_Q)\varphi,\\
            \eta_1(x,0) = l_1(x)u^*(x,T),\quad \eta_2(x,0) = l_2(x)(\sigma^*(x,T) - \sigma_\Omega),
        \end{cases}
    \end{equation}
for all $\phi \in L^2(H^1)$ and for all $\varphi \in L^2(H^2)\cap L^\infty(H^1)$ with $\partial_{\mathbf{n}}\varphi = 0$ on $\partial\Omega$. Now, we introduce the following elliptic boundary value problem
     \begin{equation}\label{very_weak_2}
         \begin{cases}
-\Delta q_2+q_2=\eta_2\ & \mbox{in}\ \Omega,\\
\partial_{\mathbf{n}} q_2 = 0\ &\mbox{on}\ \partial\Omega.
     \end{cases}
    \end{equation}
    Then, testing (\ref{very_weak})$_1$ by $\eta_1$ and (\ref{very_weak})$_2$ by $q_2,$ using the regularity of $u^*,\sigma^*,$ the linear boundedness of $\rho(\cdot),$ and the uniform bound of $\rho^\prime(\cdot),$ we can derive that $\eta_1$ is bounded in $W_2$ and $q_2$ is bounded in $L^2(H^2)\cap L^\infty(H^1),$ which implies that $\eta_2$ is bounded in $L^2(Q)\cap L^\infty((H ^1)').$ Consequently, $[p_1,p_2]\in W_2 \times (L^2(Q)\cap L^\infty((H ^1)'))$.
    
\end{proof}

\subsection{First-order Necessary Optimality Conditions. Proof of Theorem \ref{Teo:CNPO}.}
In this subsection, we derive the first-order necessary conditions, namely, the proof of Theorem \ref{Teo:CNPO}.
\begin{proof}

    Assume that $[u^*,\sigma^*,\ControlC^*,\ControlS^*] \in \mathcal{P}_{ad}$ is an optimal solution with $[\ControlC^*, \ControlS^*] \in \mathcal{U}_{ad}$. For all $[\overline{\ControlC}, \overline{\ControlS}] \in \mathcal{U}_{ad}$, we consider the admissible point $[u^\varepsilon, \sigma^\varepsilon, \ControlC^\varepsilon, \ControlS^\varepsilon] \in \mathcal{P}_{ad}$ with $[\ControlC^\varepsilon, \ControlS^\varepsilon] = [\ControlC^* + \varepsilon (\overline{\ControlC}-\ControlC^*), \ControlS^* + \varepsilon (\overline{\ControlS}-\ControlS^*)]\in \mathcal{U}_{ad}$ for any $\varepsilon \in [0,1]$. Then
    \[
    J([u^\varepsilon, \sigma^\varepsilon, \ControlC^\varepsilon, \ControlS^\varepsilon]) \geq J([u^*, \sigma^*, \ControlC^*, \ControlS^*]).
    \]
    This implies
    \[
    \begin{split}
    \int_Q \left(\frac{1}{2}k_1[(u^\varepsilon)^2-(u^*)^2] + \frac{1}{2}k_2[(\sigma^\varepsilon - \sigma_Q)^2 - (\sigma^*-\sigma_Q)^2] + k_3(\ControlC^\varepsilon - \ControlC^*) + k_4(\ControlS^\varepsilon - \ControlS^*)\right) (x,t)dxdt \\
    + \int_\Omega \left(\frac{1}{2}l_1[(u^\varepsilon)^2 - (u^*)^2] + \frac{1}{2}l_2[(\sigma^\varepsilon-\sigma_\Omega)^2-(\sigma^*-\sigma_\Omega)^2]\right)(x,T) dx \geq 0.
    \end{split}
    \]
    Multiplying by $1/\varepsilon > 0$ and assigning $z_1^\varepsilon = \frac{u^\varepsilon - u^*}{\varepsilon}$ and $z_2^\varepsilon=\frac{\sigma^\varepsilon - \sigma^*}{\varepsilon}$, it follows that
    \begin{eqnarray}
        \int_Q \left(\frac{1}{2}k_1 z_1^\varepsilon(u^\varepsilon+u^*) + \frac{1}{2}k_2z_2^\varepsilon(\sigma^\varepsilon +\sigma^*-2\sigma_Q) + k_3(\overline{\ControlC}-\ControlC^*) +k_4(\overline{\ControlS}-\ControlS^*)\right)(x,t)dxdt\nonumber\\
        + \int_\Omega \left(\frac{1}{2}l_1z_1^\varepsilon(u^\varepsilon+u^*) + \frac{1}{2}l_2z_2^\varepsilon(\sigma^\varepsilon + \sigma^* - 2\sigma_\Omega)\right)(x,T)dx \geq 0.\label{eq32-2}
    \end{eqnarray}
    Taking $\varepsilon \to 0$, yields
    \begin{equation}\label{eq32}
    \begin{split}
        \int_Q \Big(k_1 z_1u^* + k_2z_2(\sigma^*-\sigma_Q) + k_3(\overline{\ControlC}-\ControlC^*) +k_4(\overline{\ControlS}-\ControlS^*)\Big)(x,t)dxdt \\ + \int_\Omega \Big(l_1z_1u^* + l_2z_2(\sigma^*-\sigma_\Omega)\Big)(x,T)dx \geq 0.
    \end{split}
    \end{equation}
    Testing the first and second equations of the linearized system \eqref{eq:SistemaLinealizado} by $p_1$ and $p_2$ respectively, we obtain
    \begin{equation}\label{sisMult}
        \begin{cases}
        p_1\partial_t z_1 - p_1\Delta z_1 = p_1\rho(\sigma^*)(\alpha - 2u^*)z_1 + p_1\rho'(\sigma^*)u^*(\alpha - u^*)z_2 - p_1\sigma^* \ControlC^*z_1 -p_1u^*\ControlC^* z_2 \\\qquad \qquad\qquad- p_1u^*\sigma^*(\bar\ControlC-\ControlC^*) -p_1\nabla \cdot ( u^* \nabla z_2) - p_1\nabla \cdot(z_1 \nabla \sigma^*), & \text{in }Q, \\
    p_2\partial_t z_2 -p_2\Delta z_2 = -p_2\frac{\sigma^*}{1+\sigma^*}z_1- p_2\frac{u^*}{(1+\sigma^*)^2}z_2 - p_2\gamma z_2 - p_2(\bar \ControlS - \ControlS^*)u^* + p_2(1-\ControlS^*)z_1, & \text{in }Q.
        \end{cases}
    \end{equation}
    In addition, multiplying the two equations in the adjoint system \eqref{eq:sistemaAdjunto} by $z_1$ and $z_2$ respectively, we have
    \begin{equation}\label{adjMult}
        \begin{cases} 
            z_1\partial_t p_1 + z_1\Delta p_1 = - z_1\rho(\sigma^*)(\alpha - 2u^*)p_1+ z_1\sigma^* \ControlC^*p_1 - z_1\nabla \sigma^* \cdot \nabla p_1 + z_1\frac{\sigma^*}{1+\sigma^*}p_2 \\ \qquad \qquad \qquad- z_1(1-\ControlS^*)p_2- z_1k_1u^*,& \text{ in }Q, \\ 
            z_2\partial_t p_2 + z_2\Delta p_2= - z_2\rho'(\sigma^*)u^*(\alpha - u^*)p_1 + z_2u^*\ControlC^* p_1 + z_2\nabla \cdot (u^* \nabla p_1) + z_2\frac{u^*}{(1+\sigma^*)^2}p_2 \\ \qquad \qquad \qquad+ z_2\gamma  p_2 - z_2k_2(\sigma^* - \sigma_Q),& \text{ in }Q.
        \end{cases}
    \end{equation}
    Adding \eqref{sisMult} and \eqref{adjMult} gives
    \[
    \begin{split}
    \sum_{i=1}^2 \left(p_i \frac{\partial z_i}{\partial t} + z_i \frac{\partial p_i}{\partial t}\right) = & (p_1 \Delta z_1 - z_1\Delta p_1) + (p_2\Delta z_2 - z_2 \Delta p_2) \\ 
    & -p_1\nabla \cdot (u^* \nabla z_2) + z_2\nabla \cdot(u^*\nabla p_1) - p_1\nabla \cdot (z_1\nabla \sigma^*) - z_1 \nabla \sigma^* \cdot \nabla p_1 \\
    & -p_1u^*\sigma^*(\overline \ControlC - \ControlC^*) - p_2u^*(\overline \ControlS - \ControlS^*) - z_1k_1u^* - z_2k_2(\sigma^* - \sigma_Q).
    \end{split}
    \]
    Integrating by parts over $Q$ and applying the initial and boundary conditions of the linearized and adjoint systems, we obtain
    \begin{equation}\label{eq134}
        \begin{split}
        \int_\Omega \Big(l_1& z_1u^* +l_2z_2(\sigma^*-\sigma_\Omega)\Big)(x,T) dx \\ & =   \int_Q \Big(-p_1u^*\sigma^*(\overline \ControlC - \ControlC^*) - p_2u^*(\overline \ControlS - \ControlS^*) - k_1z_1u^* - k_2z_2(\sigma^*-\sigma_Q)\Big)(x,t)dxdt.
        \end{split}
    \end{equation}
    Combining \eqref{eq32} and \eqref{eq134} it follows that
    \begin{equation}\label{eq:DesigualdadVariacional}
        \int_Q \Big((k_3-u^*\sigma^*p_1)(\overline\ControlC - \ControlC^*) + (k_4-u^*p_2)(\overline{\ControlS}-\ControlS^*)\Big)(x,t) dxdt \geq 0.
    \end{equation}
\end{proof}

\begin{remark}
 %   From  Proposition \ref{prop:proyeccion} below, 
    Since the set of admissible controls $\mathcal{U}_{ad}$ is convex, from \eqref{eq:DesigualdadVariacional}, we have
    \[
    \ControlC^*(x,t) = \text{Proj}_{\mathcal{U}_{ad}}\Big(-u^*\sigma^*p_1\Big) \qquad \text{and} \qquad \ControlS^*(x,t) = \text{Proj}_{\mathcal{U}_{ad}}\Big(-u^*p_2\Big).
    \]
\end{remark}
In order to compute this projection, we state the following result.
\begin{prop}\label{prop:proyeccion}(See \cite{Fernandez-RestriccionGlobalControl})
   Assume that $u_{max}, y_{max} \in (0,+\infty)$ and $\rho \in L^2(Q)$ with $\rho(x,t) > 0$ a.e. $Q$ satisfying
   \[
   \int_Q \rho(x,t)dxdt > \frac{y_{max}}{u_{max}}.
   \]
   Let us consider the set
   \[
   \mathbb{K}=\Big\{u\in L^2(Q) : 0\leq u \leq u_{max} \text{ a.e. }(x,t)\in Q, \int_Q u(x,t)\rho(x,t) dxdt \leq y_{max}\Big\},
   \]
   and denote by Proj$_{\mathbb{K}}$ the projection operator onto $\mathbb{K}$. 
   For each $h \in L^2(Q)$, we have that
   \begin{itemize}
       \item[i)] If $\displaystyle\int_Q\text{Proj}_{[0,u_{max}]}(h(x,t))\rho(x,t)dxdt \leq y_{max}$, then Proj$_{\mathbb{K}} = $ Proj$_{[0,u_{max}]}(h)$.
       \item[ii)] If $\displaystyle\int_Q \text{Proj}_{[0,u_{max}]}(h(x,t))\rho(x,t)dxdt > y_{max}$, then there exists $\gamma > 0$ such that
       \[
       \text{Proj}_{\mathbb{K}}(h) = \text{Proj}_{[0,u_{max}]}(h-\gamma \rho) \quad \text{and} \quad \int_Q \text{Proj}_{\mathbb{K}}(h)(x,t)\rho(x,t) dxdt = y_{max}.
       \]
   \end{itemize}
\end{prop}

\section{Numerical Simulations}

This section describes the numerical strategy used to approximate the control problem related to the system (\ref{eq:SistemaControlado}). We show some numerical simulations, 
% that validate the model and the proposed control problem, 
demonstrating the effectiveness and
computational efficiency of the proposed scheme.

\subsection{Approximation of the States and Adjoint System}

To approximate the state and adjoint systems (\ref{eq:SistemaControlado}) and (\ref{eq:sistemaAdjunto}), we propose a  first-order, linear, decoupled, fully discrete numerical scheme, based on the finite element method for spatial discretization and finite differences for temporal discretization. 

We consider a partition of $[0,T]$ with time step $\Delta t = T/N : (t_n = n \Delta t)_{n=0}^{n=N}$
%.\\. \\ \noindent
%For the spatial discretization, we consider 
and a family of triangulations $\{\mathcal{T}_h\}_{h>0}$ of $\overline \Omega$, and the
$\mathbb{P}_1$ continuous  finite element space
% is defined as
\[
X_h = \{v_h \in C(\overline\Omega) : v_h |_K \in \mathbb{P}_1, \forall K \in \mathcal{T}_h\} \subset H^1(\Omega).
\]
%Thus, the state variables are spatially approximated in $X_h$. 
%\begin{comment}
%In addition, the nodal interpolation operator $I_h : C(\bar \Omega) \to X_h$ and the mass-lumping inner product
%\[
%(m_1, m_2)^h := \int_\Omega I_h(m_1m_2)dx,
%\]
%with its associated norm $|m|_h = \sqrt{(m,m)^h}$ are introduced.
%\end{comment}
Likewise, we consider the projection operator $\mathcal{Q}^h : L^2(Q) \to X_h$, defined by
\[
(\mathcal{Q} ^hz, \bar m) = (z, \bar m), \quad \forall \bar m \in X_h.
\]
%Thus, we consider the following first-order in time, linear, and decoupled numerical scheme
\\
{\bf Approximation for the state  system   (\ref{eq:SistemaControlado}):}
\begin{itemize}
\item \textbf{Initialization:} Let $[u^0_h, \sigma_h^0] = [\mathcal{Q}^h u_0, \mathcal{Q}^h\sigma_0] \in X_h \times X_h$.
\item \textbf{Time step $n$:} Given the vector $[u_h^{n-1}, \sigma_h^{n-1}] \in X_h \times X_h$, calculate $[u_h^n, \sigma_h^n] \in X_h \times X_h$ solving the following linear decoupled problem
\end{itemize}
\[
\begin{cases}
    (\delta_t u_h^n, \bar u) + D_u(\nabla u_h^n, \nabla \bar u) 
    = (\rho(\sigma^{n-1}_h)u^n_h(\alpha - u_h^{n-1}), \bar u)
    -\kappa(\ControlC u^n_h\sigma^{n-1}_h, \bar u)
    +\chi(u^n_h \nabla \sigma^{n-1}_h, \nabla \bar u), \\
    (\delta_t \sigma^n_h, \bar \sigma) + D_\sigma(\nabla \sigma^n_h, \nabla \bar \sigma) = (P_{er}S_v(\beta - \sigma_h^n), \bar \sigma)-\displaystyle(\frac{A_{ox}u^n_h\sigma^n_h}{k_{ox}+\sigma_h^{n-1}}, \bar \sigma) + ((1-\ControlS)S_cu^n_h, \bar \sigma),
\end{cases}
\]
for all $[\bar u, \bar \sigma] \in [X_h]^2$. Here $\delta_t z_h^n := (z_h^n-z_h^{n-1})/\Delta t$.\\ 

\noindent
{\bf Approximation for the adjoint  system  (\ref{eq:sistemaAdjunto}):}
\begin{itemize}
\item \textbf{Initialization:} Let $[p^{N+1}_{1,h}, p_{2,h}^{N+1}] = [\mathcal{Q}^h p_1(T), \mathcal{Q}^hp_2(T)] \in X_h \times X_h$.
\item \textbf{Time step $n$:} Given the vector $[p_{1,h}^{n+1}, p_{2,h}^{n+1}] \in X_h \times X_h$, calculate $[p_{1,h}^n, p_{2,h}^n] \in X_h \times X_h$ such that
\end{itemize}
\[
\begin{cases}
    -(\tilde\delta_t p_{1,h}^n, \bar p) + D_u(\nabla p_{1,h}^n, \nabla \bar p) = (\rho(\sigma^*)(\alpha - 2u^*)p_{1,h}^n, \bar p) -\kappa(\sigma^*\ControlC^*p_{1,h}^n, \bar p) + (\nabla \sigma^* \cdot \nabla p_{1,h}^n,\bar p) \\ \qquad \qquad - \displaystyle(\frac{A_{ox}\sigma^*}{k_{ox}+\sigma^*}p_{2,h}^{n+1}, \bar p) + ((1-\ControlS^*)p_{2,h}^{n+1}, \bar p) + (k_1 u^*, \bar p), \\
    -(\tilde\delta_t p^n_{2,h}, \bar q) + D_\sigma(\nabla p^n_{2,h}, \nabla \bar q) = (\rho'(\sigma^*)u^*(\alpha -u^*)p_{1,h}^n, \bar q) - \kappa(u^*\ControlC^* p_{1,h}^n, \bar q) + (u^*\nabla p_{1,h}^n, \nabla  \bar q) \\ \qquad \qquad- \displaystyle(\frac{A_{ox}k_{ox}u^*}{(k_{ox}+\sigma^*)^2}p_{2,h}^n, \bar q)- (P_{er}S_v p_{2,h}^n, \bar q) + (k_2(\sigma^*-\sigma_Q), \bar q),
\end{cases}
\]
for all $[\bar p, \bar q] \in [X_h]^2.$ Here $\tilde\delta_t z_h^n :=(z_h^{n+1}-z_h^{n})/\Delta t$.

\begin{comment}
\begin{remark}
    In order to preserve the biological interpretation of the variables, after each resolution, a point truncation is applied to ensure non-negativity. In particular, the following is imposed:
    \[
    u_h^n \leftarrow \text{max}\{u_h^n, 0\}, \qquad \sigma_h^n \leftarrow \text{max}\{\sigma_h^n, 0\}.
    \]
\end{remark}
\end{comment}

\subsection{Optimal Control Algorithm}
To approximate the solutions of the optimal control problem  we propose a gradient descent type method, namely, the Adam method \cite{Adam-cite}. By simplicity, in all simulations, the controls $\ControlC$ and $\ControlS$ depend only on time. The optimal controls are obtained by building sequences $\{\ControlC_k\}_{k \geq 1}$ and $\{\ControlS_k\}_{k\geq 1}$, starting from given initial controls $\ControlC_0$ and $\ControlS_0$. These sequence are computed using the following iterative scheme:\\

\noindent
\textbf{Step 1.} Compute $[u_k, \sigma_k]$ by solving state system \eqref{eq:SistemaControlado} with 
$[\ControlC^*, \ControlS^*] = [\ControlC_k, \ControlS_k]$. \\ \\
\noindent
\textbf{Step 2.} Compute $[p_1^k, p_2^k]$ by solving adjoint system \eqref{eq:sistemaAdjunto}  with
 $[\ControlC^*, \ControlS^*] = [\ControlC_k, \ControlS_k]$ and $[u^*, \sigma^*] = [u_k ,\sigma_k]$. \\ \\
\noindent
\textbf{Step 3.} Compute the gradient of the cost functional $J$ in $[\ControlC_k, \ControlS_k]$ as
\[
d^k_{\ControlC} = \nabla_{\ControlC}J([u_k, \sigma_k, \ControlC_k, \ControlS_k])
 = k_3 - \int_\Omega\kappa \, \sigma_k u_k p_1^k, 
 \qquad d^k_{\ControlS}=\nabla_{\ControlS} J([u_k,\sigma_k,\ControlC_k,\ControlS_k]) = k_4 - \int_\Omega S_c u_k p_2^k.
\]
\noindent
\textbf{Step 4.} Compute the value of the cost functional at the current iterate: $J_k = J([u_k, \sigma_k, \ControlC_k, \ControlS_k])$.\\ \\
\noindent
\textbf{Step 5.} %Update the controls by using an Adam type descent strategy. For each time level, 
Compute the first and second moments of the discrete gradients as
\[
m^{k+1}=\beta_1 m^k + (1-\beta_1)d^k,
\qquad
v^{k+1}=\beta_2 v^k + (1-\beta_2)(d^k)^2,
\]
where $d^k$ denotes the corresponding discrete gradient component. After bias correction, obtain a descent direction as
\[
p^{k+1}=-\frac{\widehat m^{k+1}}{\sqrt{\widehat v^{k+1}}+\varepsilon}.
\]
Then, define the tentative controls by
\[
\ControlC_{k+1}
=
\ControlC_k+\alpha_k p_{\ControlC}^{k+1},
\qquad
\ControlS_{k+1}
=
\ControlS_k+\alpha_k p_{\ControlS}^{k+1},
\]
where $\alpha_k>0$ is the step size. \\ \\

\noindent
\textbf{Step 6.} Compute a projection of the tentative controls onto the admissible set. For the pointwise bounds, truncate the controls as 
\[
0\leq \ControlC_{k+1}(t)\leq 1,
\qquad
0\leq \ControlS_{k+1}(t)\leq 1.
\]
In the case of $\ControlC_{k+1}$, we must to apply the additional integral constraint
\[
\int_0^T \ControlC_{k+1}(t)\,dt \leq c_{max}.
\]
If the truncated control already satisfies this condition, no further correction is required. Otherwise, following Proposition \ref{prop:proyeccion}, there exists a constant $\lambda>0$ such that 
\[
    \ControlC_{k+1} = \min\{\max\{\ControlC-\lambda,0\},1\} \qquad \text{and} \qquad \int_0^T \min\{\max\{\ControlC-\lambda,0\},1\}dt = c_{max}.
\]
For this case, we compute the scalar $\lambda$ using the bisection method.
\\ \\

\noindent
\textbf{Step 7. (Stopping criterion).} 
Let $J_k$ denote the value of the objective functional at iteration $k$ and define
\[
\delta_k =
|J_k - J_{k-1}|.
\]
Given a tolerance ${\bf tol}>0$ and an integer
$\mathcal{N}\in\mathbb{N}$, the algorithm stops if
\[
\delta_{k-j} < {\mbox{\bf tol}},
\qquad
j=0,\dots,\mathcal{N}-1,
\]
that is, if the value of the functional remains stable during $\mathcal{N}$ consecutive iterations.
Otherwise, set $k\leftarrow k+1$ and repeat Steps 1--6.

\subsection{Computational Setup}

The numerical simulations are performed on a two-dimensional normalized domain discretized using finite elements with $|\Omega| = 1$, representing the axial section of a brain (see Figure \ref{fig:dominio}). For the temporal discretization, we consider a final time $T = n_{iter}\Delta t$ with $n_{iter} = 250$ and time step $\Delta t = 0.008$; therefore, the total simulated time was $T=2$. \\

The computational domain (see Figure \ref{fig:dominio}) was constructed  from a radiographic image by identifying the relevant geometric structures. Initially, the external boundary of the domain and its interior were highlighted in black, producing a binary image that defines the principal region. Subsequently, an analogous procedure was applied to detect and delineate the internal boundaries of the domain, yielding additional binary masks corresponding to voids (regions where the mesh is not defined). These images, stored in PGM format with binary encoding (black representing the domain and white the exterior), were imported into FreeFEM++, where each one was interpreted as a discrete scalar field. From these fields, level curves (isolines) were extracted to reconstruct the associated boundaries. Finite element meshes were then generated from these curves. Finally, the global domain was obtained by subtracting the meshes of the internal regions from the main mesh, resulting in an irregular computational geometry that approximates the structural characteristics of the original radiographic image.

\begin{figure}[H]
     \centering
     \begin{subfigure}[b]{0.3\textwidth}
         \centering
         \includegraphics[width=\textwidth]{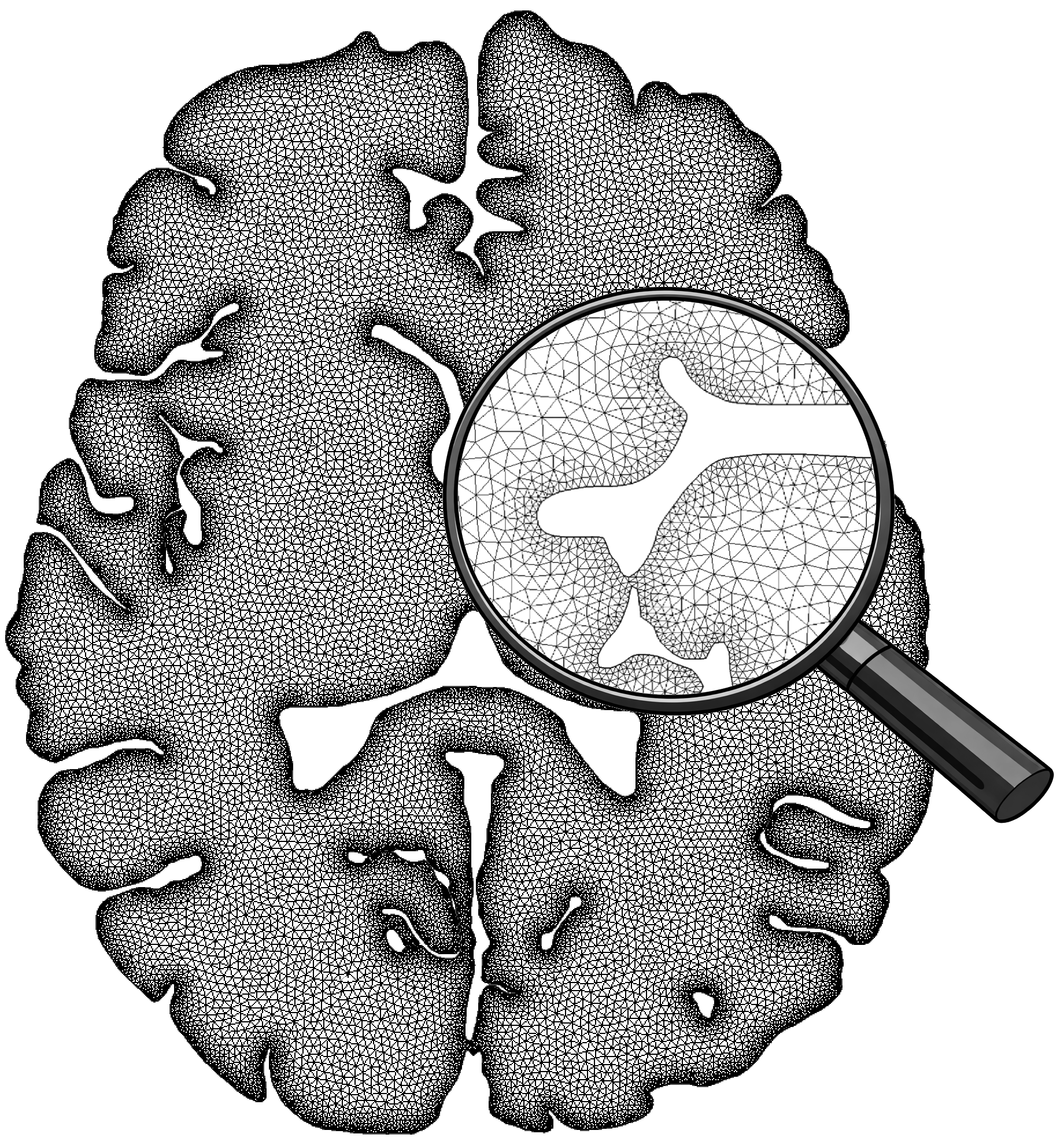}
         \caption{Domain mesh.}
         \label{fig:dominio}
     \end{subfigure}
     \hfill
     \begin{subfigure}[b]{0.3\textwidth}
         \centering
         \includegraphics[width=\textwidth]{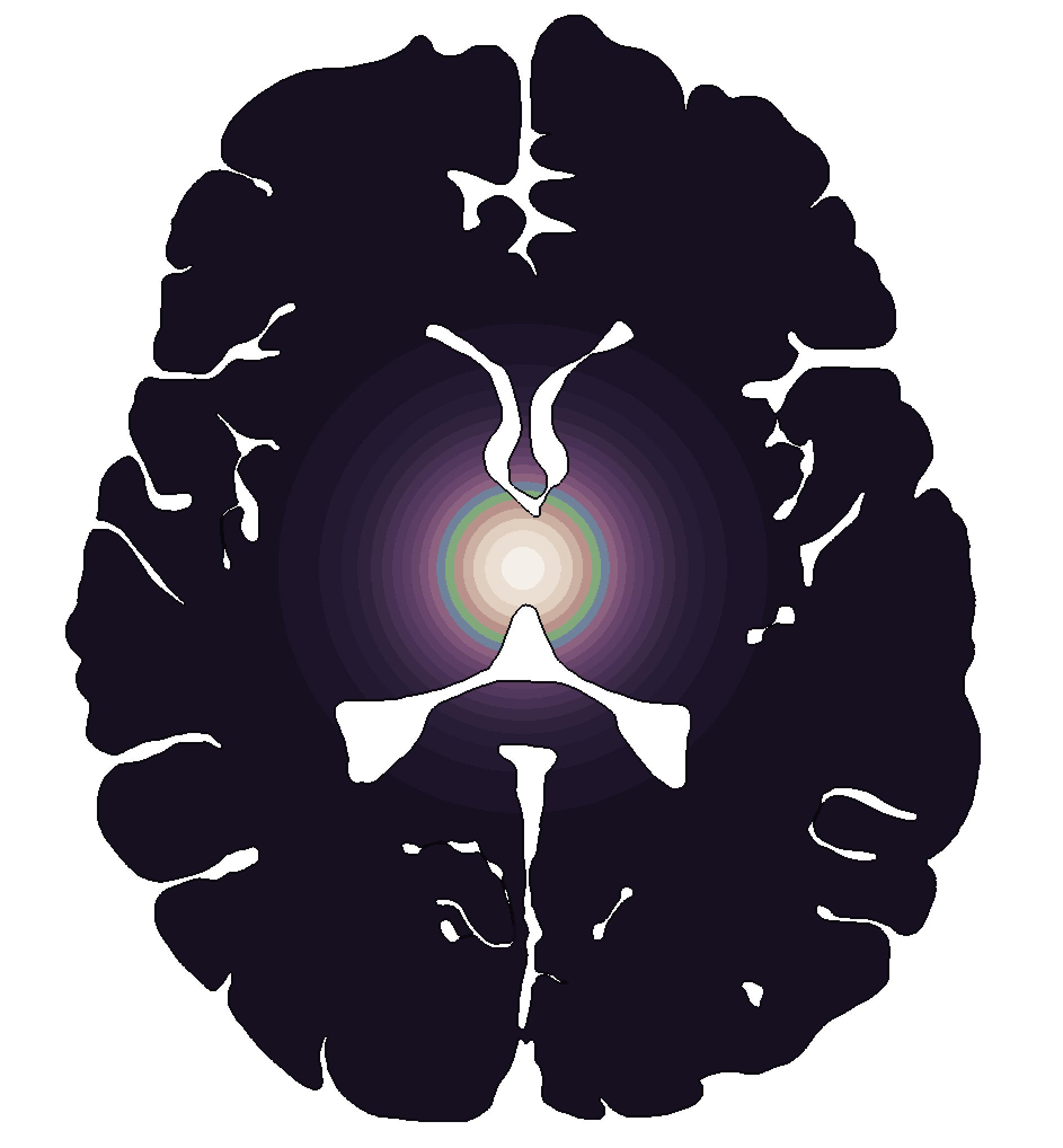}
         \caption{Initial tumor density.}
         \label{fig:u0}
     \end{subfigure}
     \hfill
     \begin{subfigure}[b]{0.3\textwidth}
         \centering
         \includegraphics[width=\textwidth]{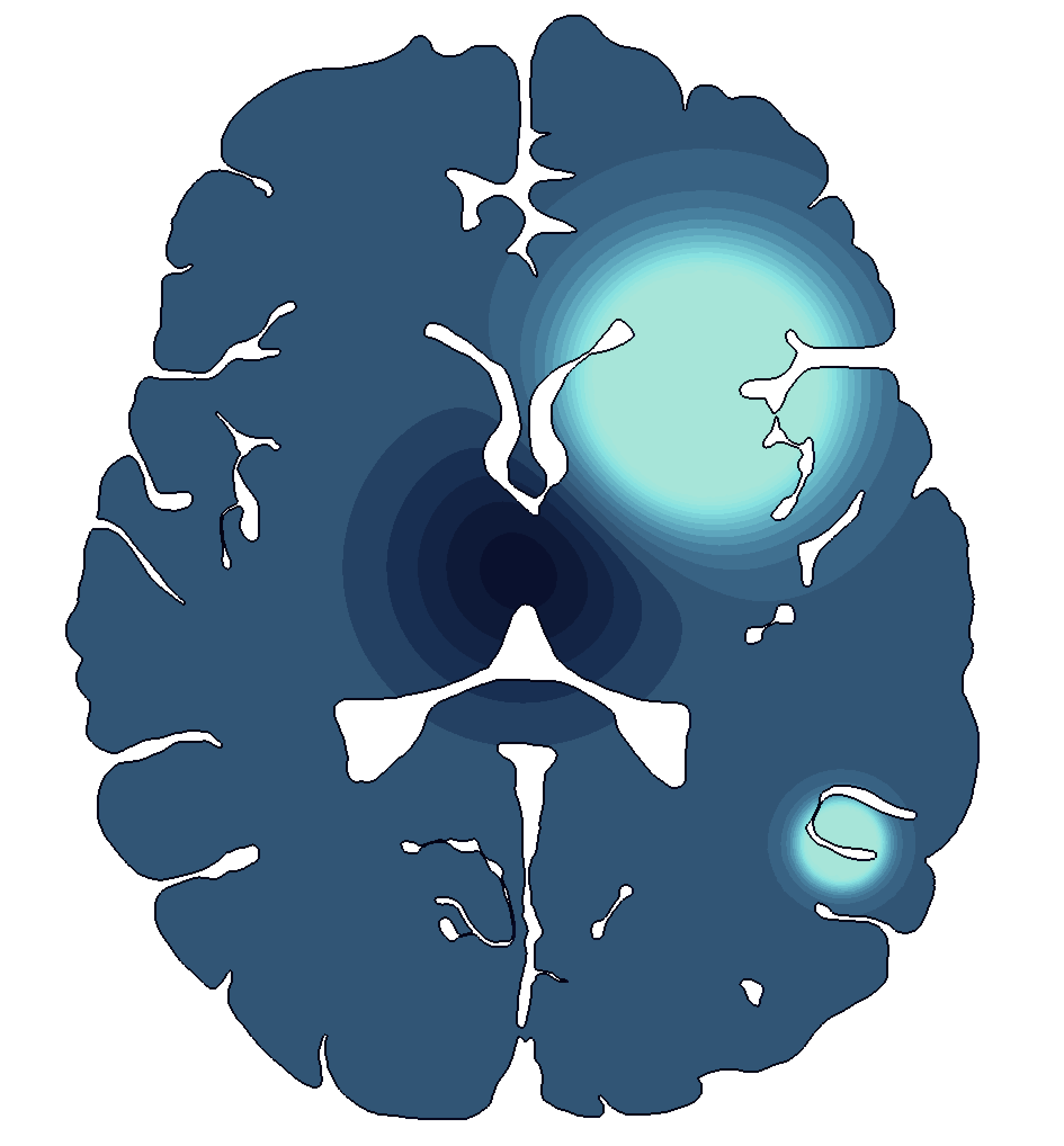}
         \caption{Initial oxygen concentration.}
         \label{fig:s0}
     \end{subfigure}
        \caption{Initial settings for the simulations.}
        \label{fig:confInitial}
\end{figure}

\noindent
We define the initial condition for tumor density $u$ as a Gaussian distribution localized at the center of the domain (see Figure \ref{fig:u0}). Denoting by $L_x$ and $L_y$ the horizontal and vertical lengths of the domain, respectively, and by $(x_c, y_c)$ the center point of the domain, with a width based on the size of the domain, we consider
\[
u_0(x,y) = 0.6 \exp \Big(-\frac{(x-x_c)^2+(y-y_c)^2}{2(0.1 \min\{L_x, L_y\})^2}\Big).
\]
The initial condition for the oxygen $\sigma$ was constructed using $u_0$ and two nearby Gaussians in order to generate nearby gradients (see Figure \ref{fig:s0}). More precisely, we consider
\[
\sigma_0(x,y)= 0.7 \beta (1-0.3 u_0(x,y)) + \beta \exp\left(-\frac{(x-x_1)^2 + (y-y_1)^2}{2(0.08\min\{L_x, L_y\})^2}\right) + \beta \exp\left(-\frac{(x-x_2)^2 + (y-y_2)^2}{2(0.025\min\{L_x, L_y\})^2}\right).
\]
The parameters used for the state equations are shown in the Table \ref{tab:Parametros}, and the parameters used for the Adam algorithm are presented in the Table \ref{tab:ParametrosAdam}.
\begin{table}[H]
\centering

\begin{minipage}{0.6\textwidth}
\centering
\begin{tabular}{cc|cc}
\hline
\multicolumn{2}{c|}{\textbf{Tumor equation}} & \multicolumn{2}{c}{\textbf{Oxygen equation}} \\ \hline
\textbf{Parameter}      & \textbf{Value}     & \textbf{Parameter}      & \textbf{Value}     \\ \hline
$D_u$                   & $2$                & $D_\sigma$              & $0.02$             \\
$\alpha$                & $4$                & $A_{ox}$                & $0.6$              \\
$\hat \rho$            & $3.5$              & $k_{ox}$                & $0.1$              \\
$b$                     & $1$                & $\beta$                & $1$                \\
$\chi$                  & $10$               & $P_{er}S_v$                   & $1$                \\
$\kappa$                & $1$                & $S_c$                   & $0.4$              \\ \hline
\end{tabular}
\caption{Parameters for the state equations.}
\label{tab:Parametros}
\end{minipage}
\hspace{0.05\textwidth}
\begin{minipage}{0.3\textwidth}
\centering
\begin{tabular}{cc}
\hline
\textbf{Parameter} & \textbf{Value} \\ \hline
$\beta_1$          & $0.9$          \\
$\beta_2$          & $0.999$        \\
$\varepsilon$      & $10^{-8}$      \\
$\alpha_0$         & $0.1$          \\ \hline
\end{tabular}
\caption{Parameters for the Adam algorithm.}
\label{tab:ParametrosAdam}
\end{minipage}

\end{table}
\noindent
Finally, for the cost functional, we consider the desired oxygen as the constant capacity   $\beta=\sigma_Q = \sigma_\Omega $.

\subsection{Experiment 1 - Uncontrolled Evolution}
This section analyzes the behavior of model (\ref{eq:SistemaControlado}) in the absence of control, that is, by setting $\ControlC = 0$ and $\ControlS = 0$, in order to observe the natural growth of the tumor. Figure \ref{fig:taxis} shows the spatio-temporal evolution of the cell density $u$ and the oxygen concentration $\sigma$ during the early times of the simulation. The first row (violet-type color) corresponds to the evolution of the tumor, while the second row (blue color) corresponds to the evolution of oxygen. During this period, the dynamic of the system is primarily governed by the oxytaxis, characterized by a rapid migration toward regions with higher oxygen concentrations. Figure \ref{fig:Evol_1anio} illustrates the evolution of the system over a longer time horizon up to $T=1$, where random diffusion becomes significant in conjunction with oxytaxis effects.
\begin{figure}[H]
     \centering
     \begin{subfigure}[b]{0.24\textwidth}
         \centering
         \includegraphics[width=\textwidth]{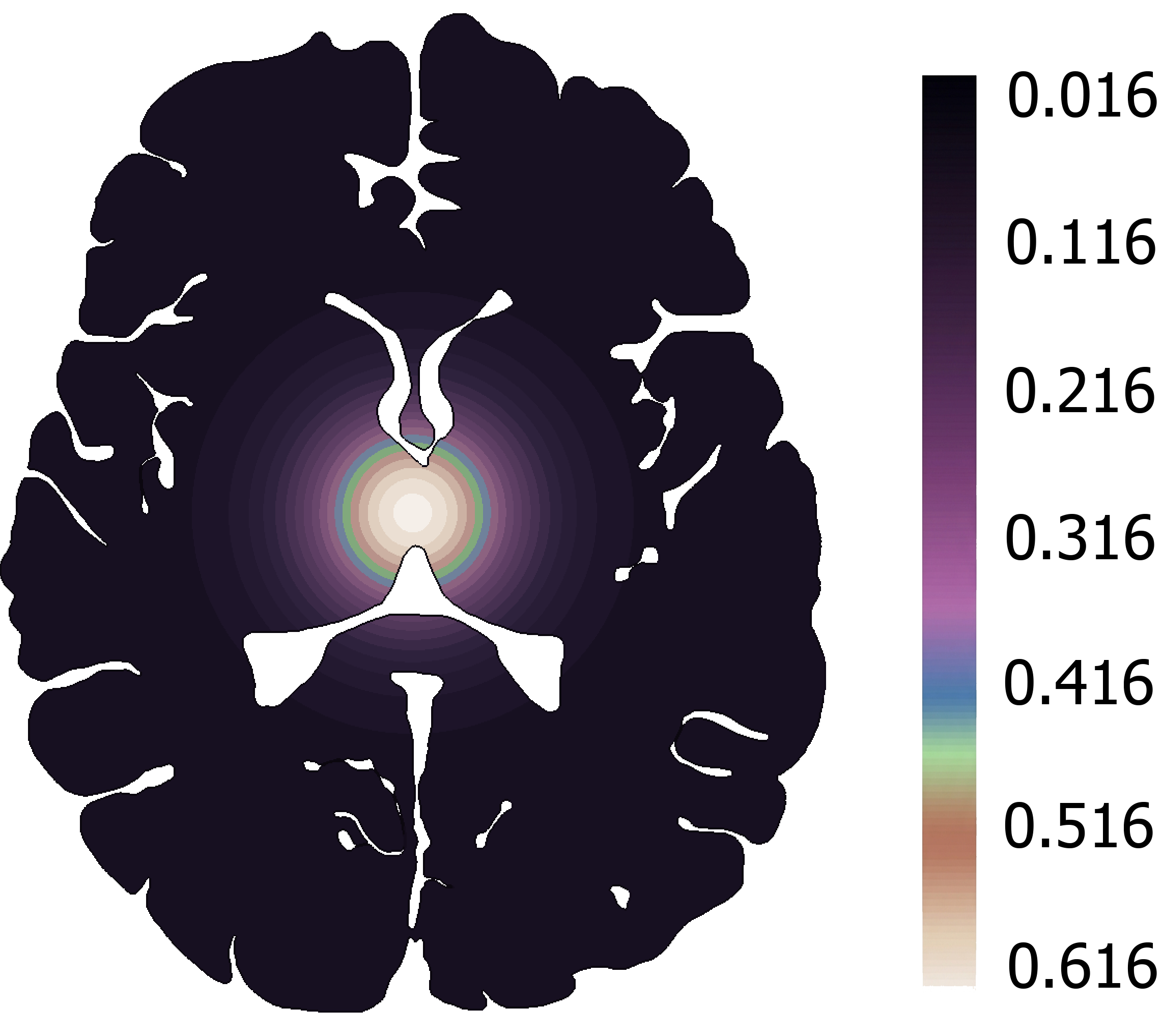}
         \caption{Initial}
         \label{fig:u0_inicial_quimioatraccion}
     \end{subfigure}
     \hfill
     \begin{subfigure}[b]{0.24\textwidth}
         \centering
         \includegraphics[width=\textwidth]{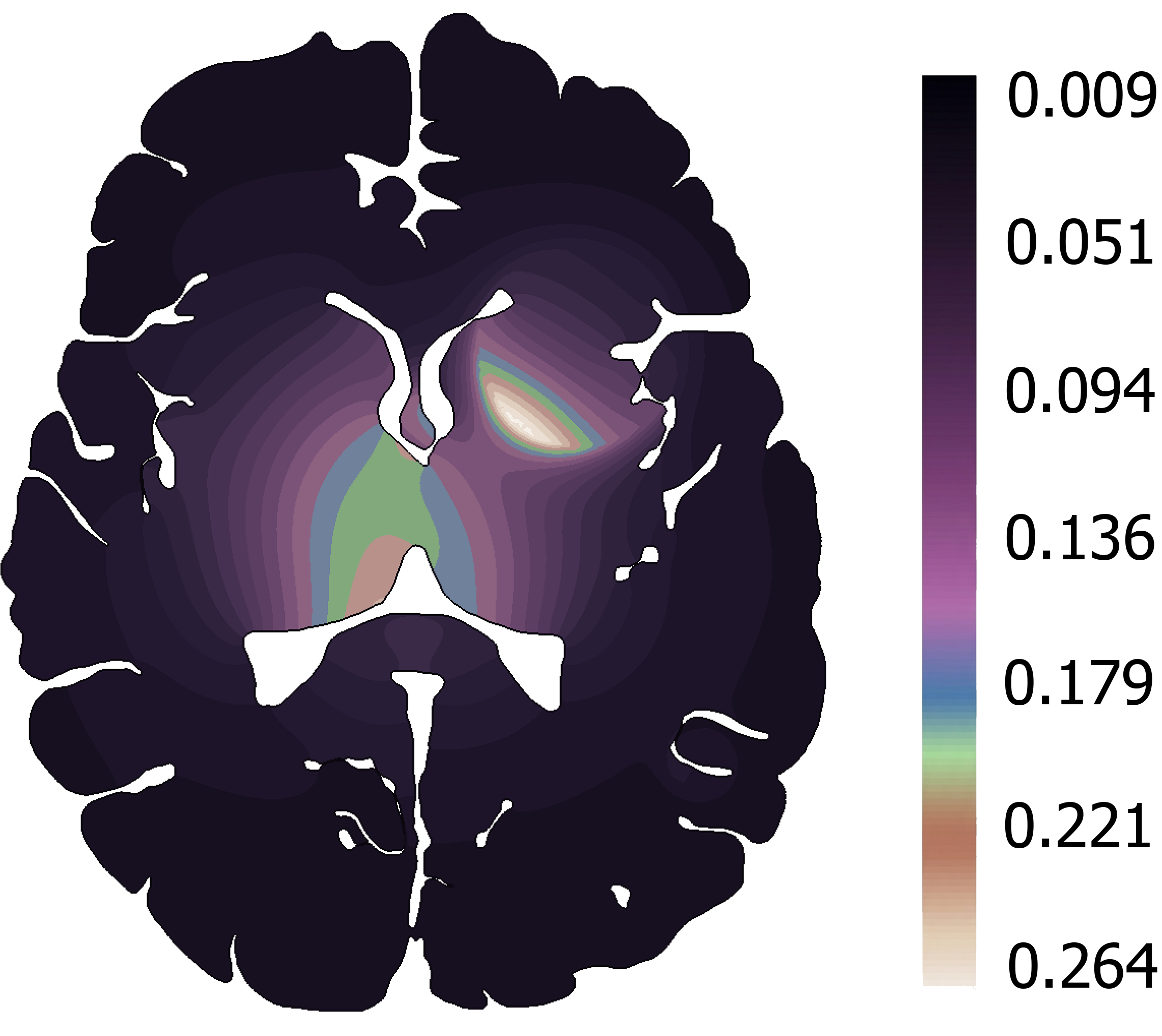}
         \caption{$t = 0.008$}
         
         \label{fig:u0_0_08}
     \end{subfigure}
     \hfill
     \begin{subfigure}[b]{0.24\textwidth}
         \centering
         \includegraphics[width=\textwidth]{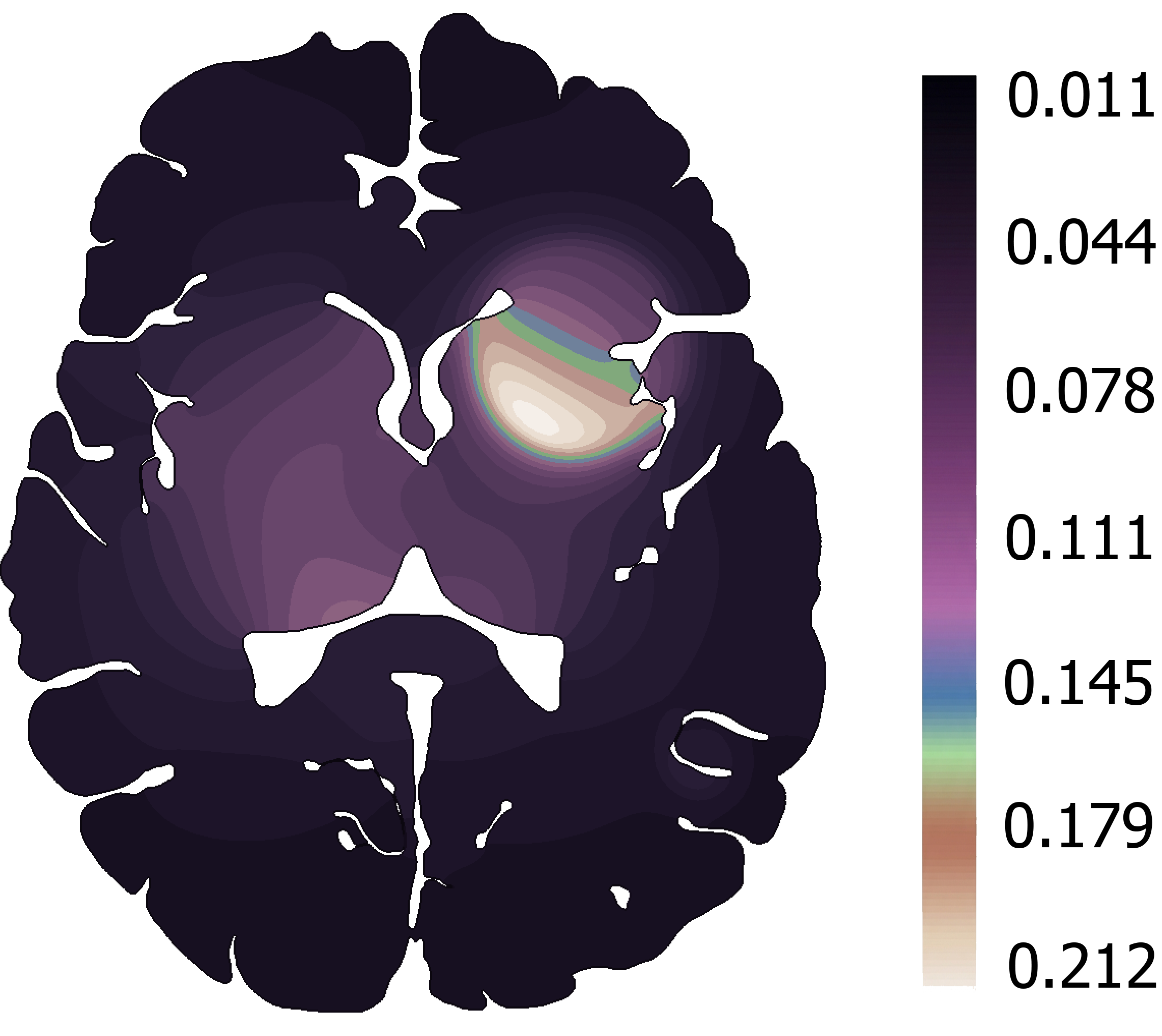}
         \caption{$t = 0.016$}
         \label{fig:u0_0_016}
     \end{subfigure}
     \hfill
     \begin{subfigure}[b]{0.24\textwidth}
         \centering
         \includegraphics[width=\textwidth]{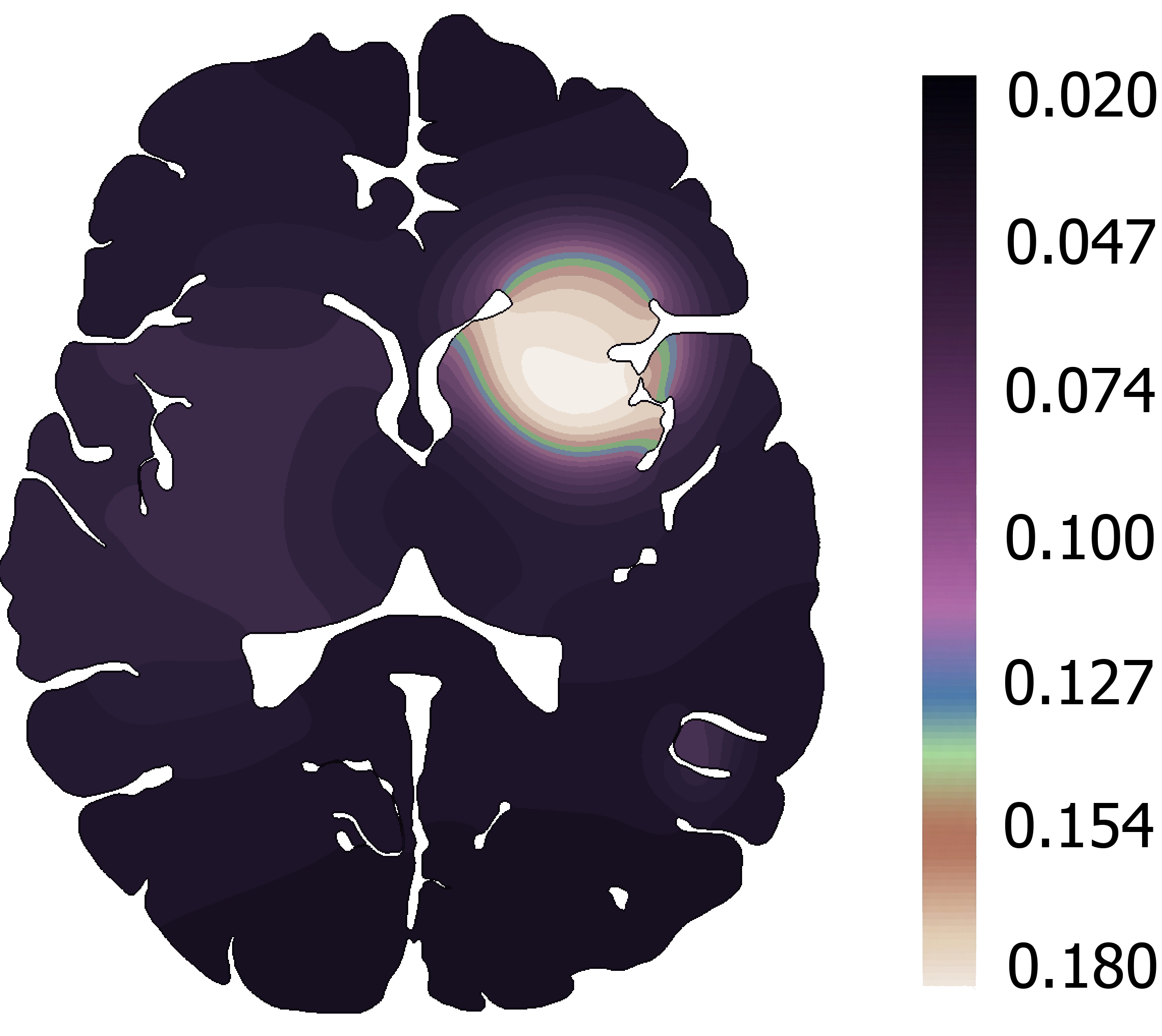}
         \caption{$t = 0.032$}
         \label{fig:u0_0_04}
     \end{subfigure}
     \begin{subfigure}[b]{0.24\textwidth}
         \centering
         \includegraphics[width=\textwidth]{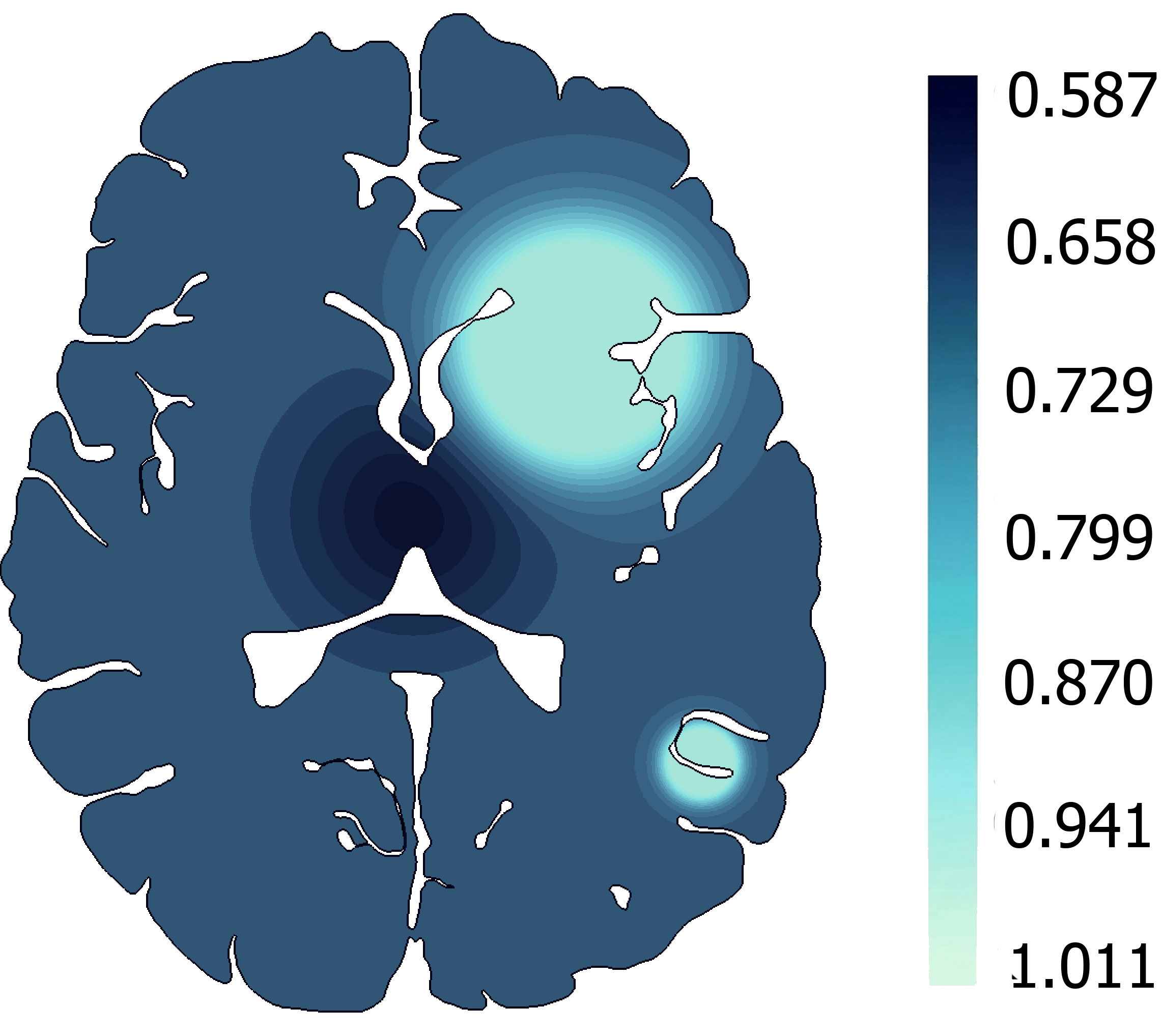}
         \caption{Initial}
         \label{fig:s0_inicial_quimioatraccion}
     \end{subfigure}
     \hfill
     \begin{subfigure}[b]{0.24\textwidth}
         \centering
         \includegraphics[width=\textwidth]{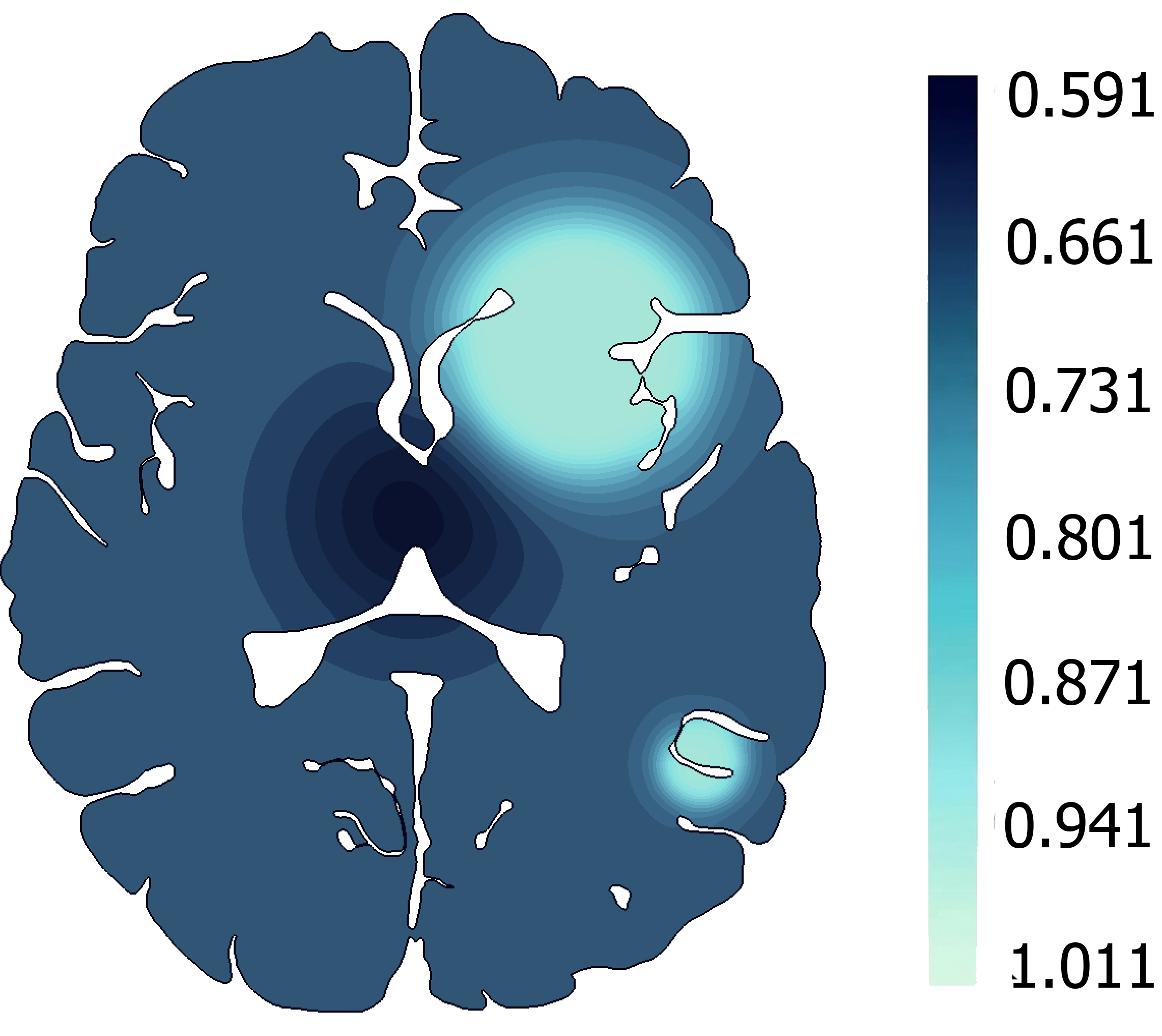}
         \caption{$t = 0.08$}
         \label{fig:s0_0_08}
     \end{subfigure}
     \hfill
     \begin{subfigure}[b]{0.24\textwidth}
         \centering
         \includegraphics[width=\textwidth]{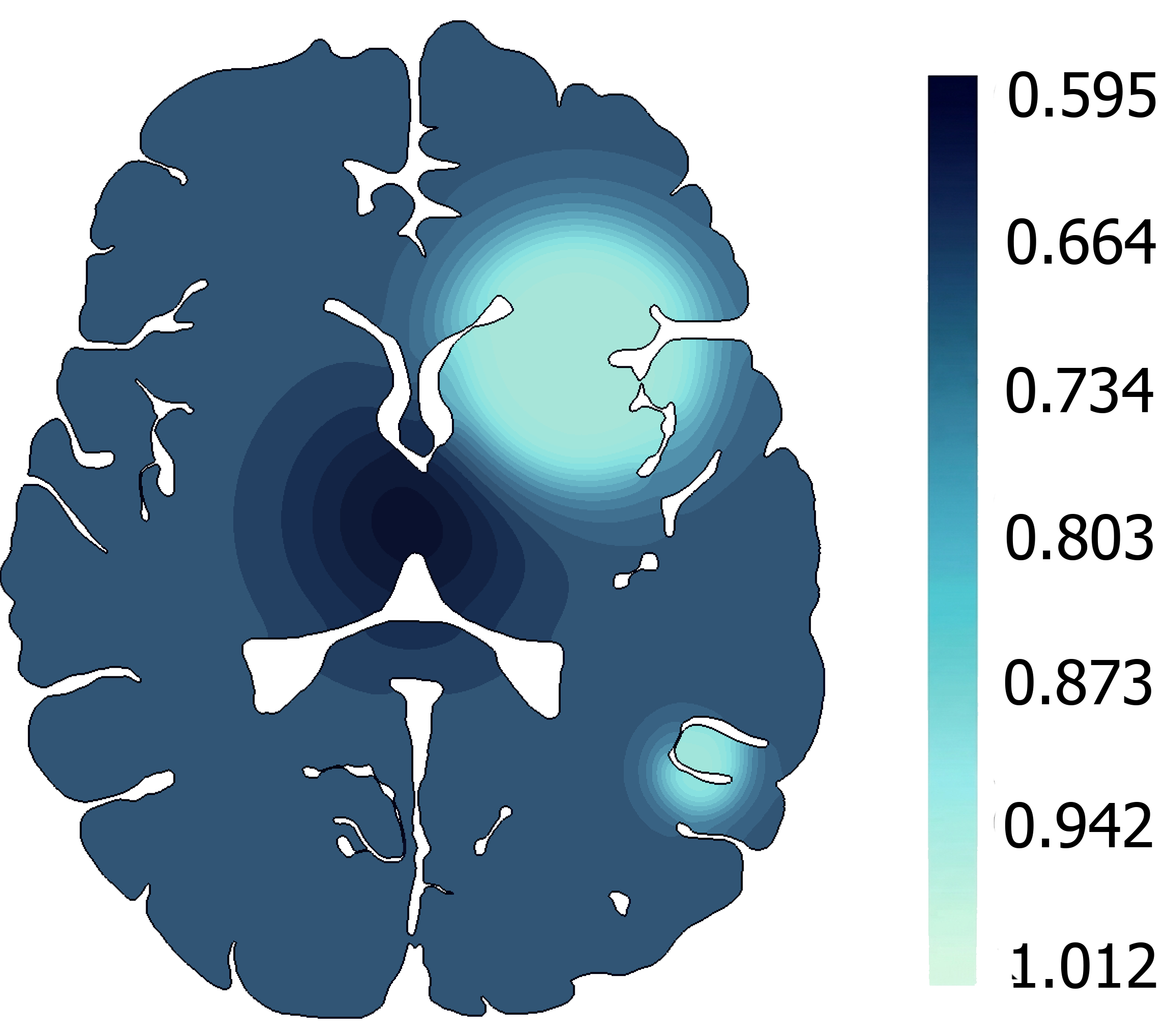}
         \caption{$t = 0.016$}
         \label{fig:s_0_016}
     \end{subfigure}
     \hfill
     \begin{subfigure}[b]{0.24\textwidth}
         \centering
         \includegraphics[width=\textwidth]{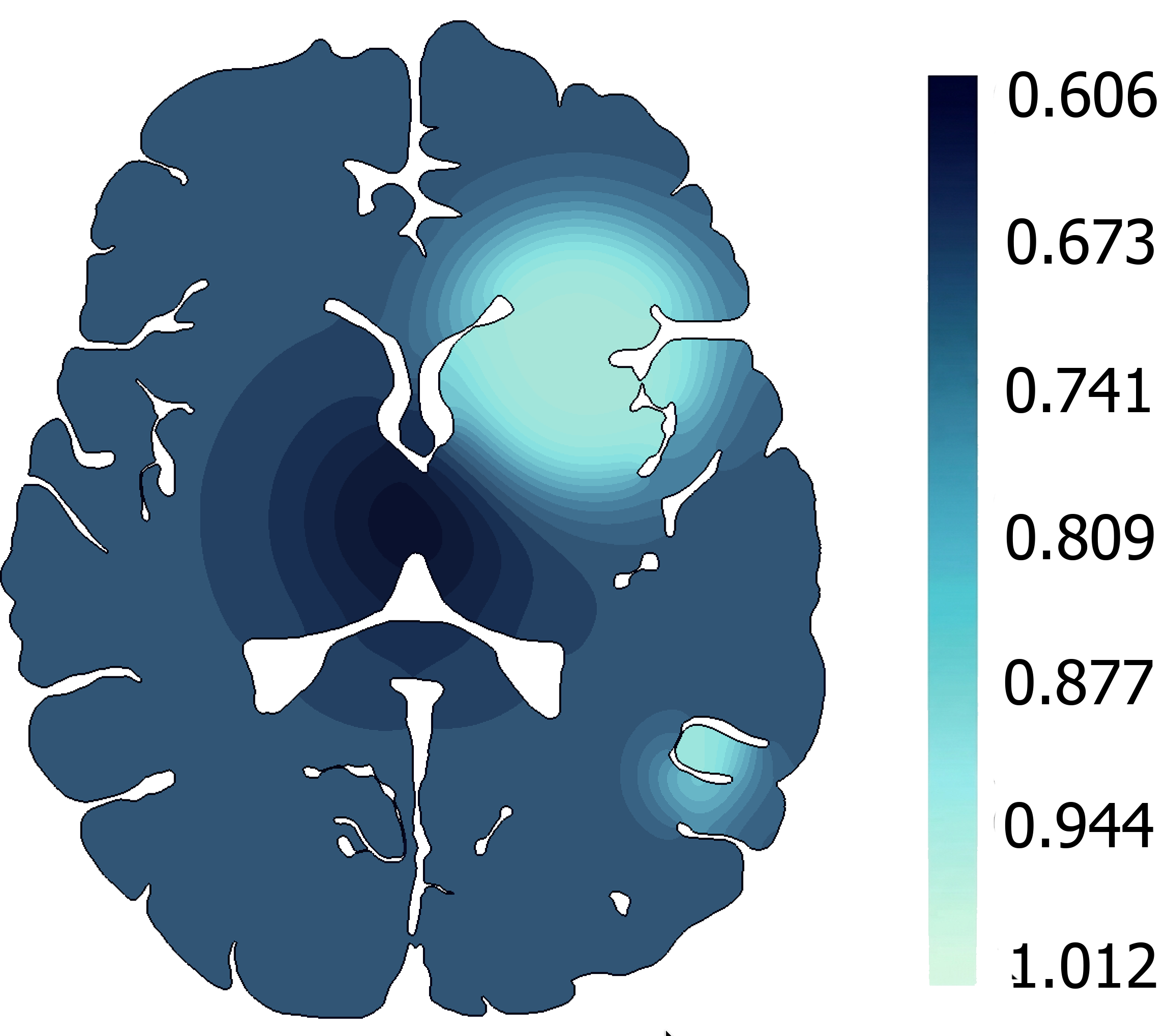}
         \caption{$t = 0.032$}
         \label{fig:s0_0_04}
     \end{subfigure}
        \caption{First times of evolution.}
        \label{fig:taxis}
\end{figure}
\begin{figure}[H]
     \centering
     \begin{subfigure}[b]{0.30\textwidth}
         \centering
         \includegraphics[width=\textwidth]{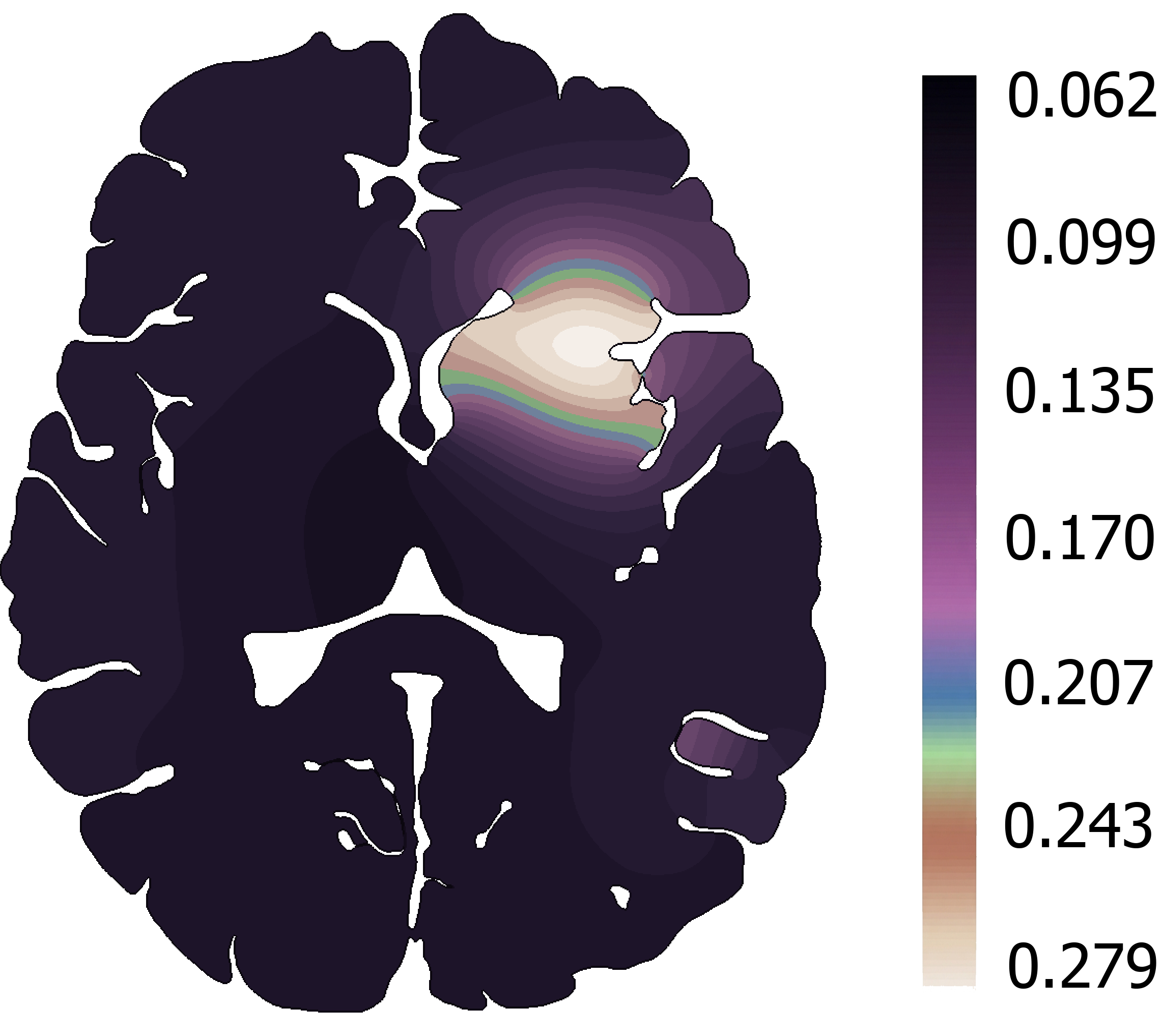}
         \caption{$t = 0.25$}
         \label{fig:u_3meses}
     \end{subfigure}
     \hfill
     \begin{subfigure}[b]{0.30\textwidth}
         \centering
         \includegraphics[width=\textwidth]{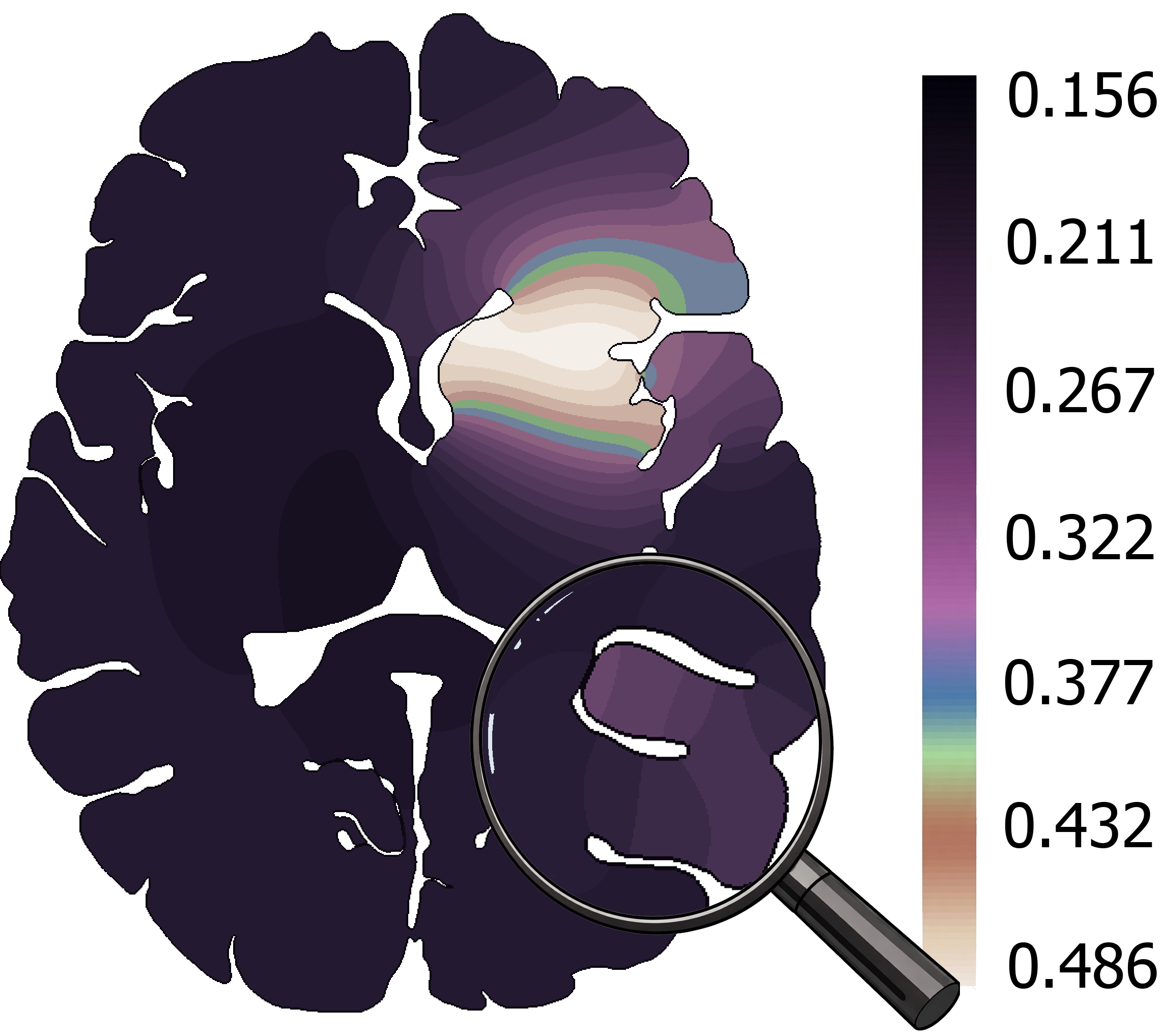}
         \caption{$t = 0.5$}
         \label{fig:u_6meses}
     \end{subfigure}
     \hfill
     \begin{subfigure}[b]{0.30\textwidth}
         \centering
         \includegraphics[width=\textwidth]{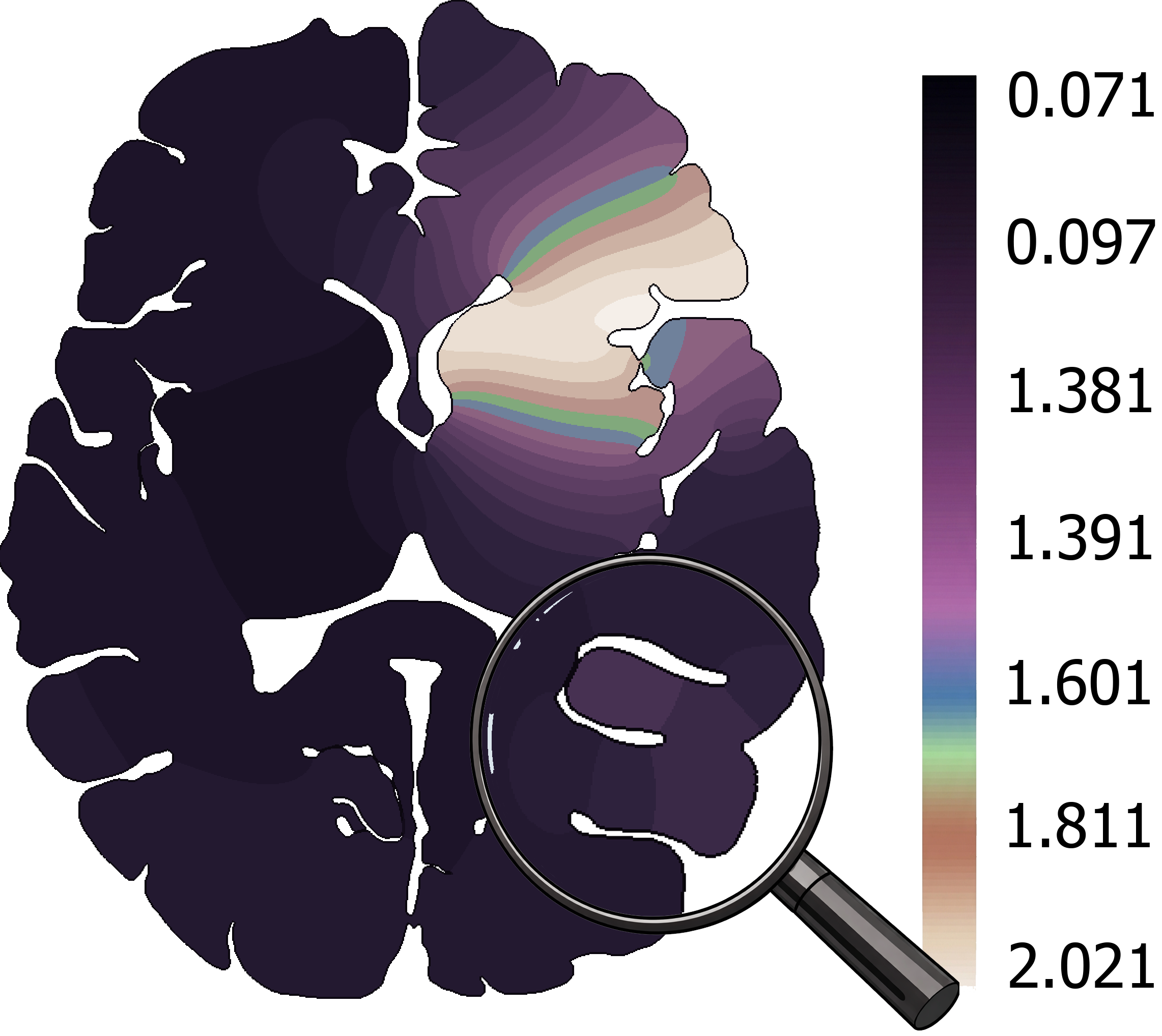}
         \caption{$t = 1$}
         \label{fig:u_1anio}
     \end{subfigure}

     \centering
     \begin{subfigure}[b]{0.30\textwidth}
         \centering
         \includegraphics[width=\textwidth]{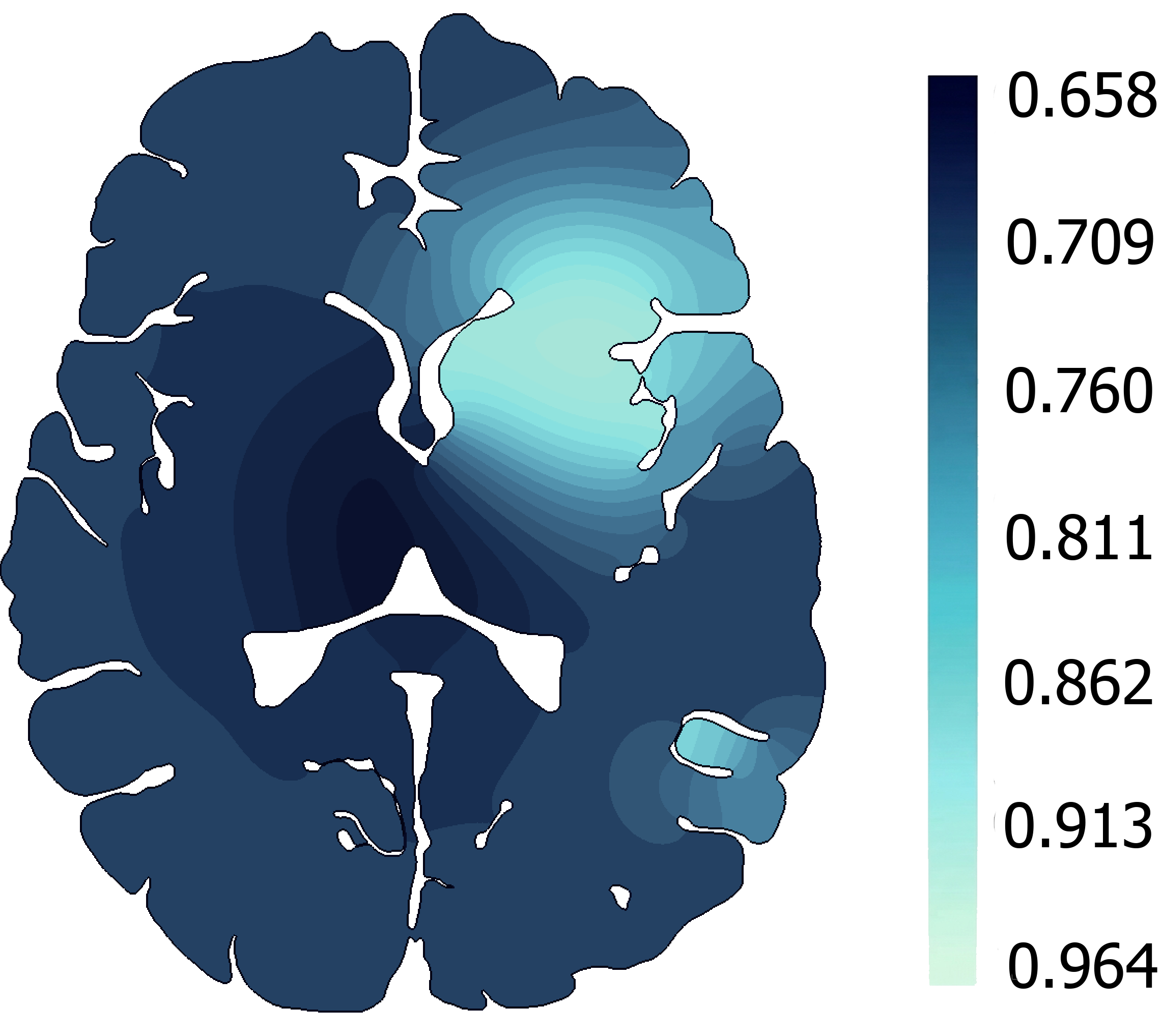}
         \caption{$t = 0.25$}
         \label{fig:s_3meses}
     \end{subfigure}
     \hfill
     \begin{subfigure}[b]{0.30\textwidth}
         \centering
         \includegraphics[width=\textwidth]{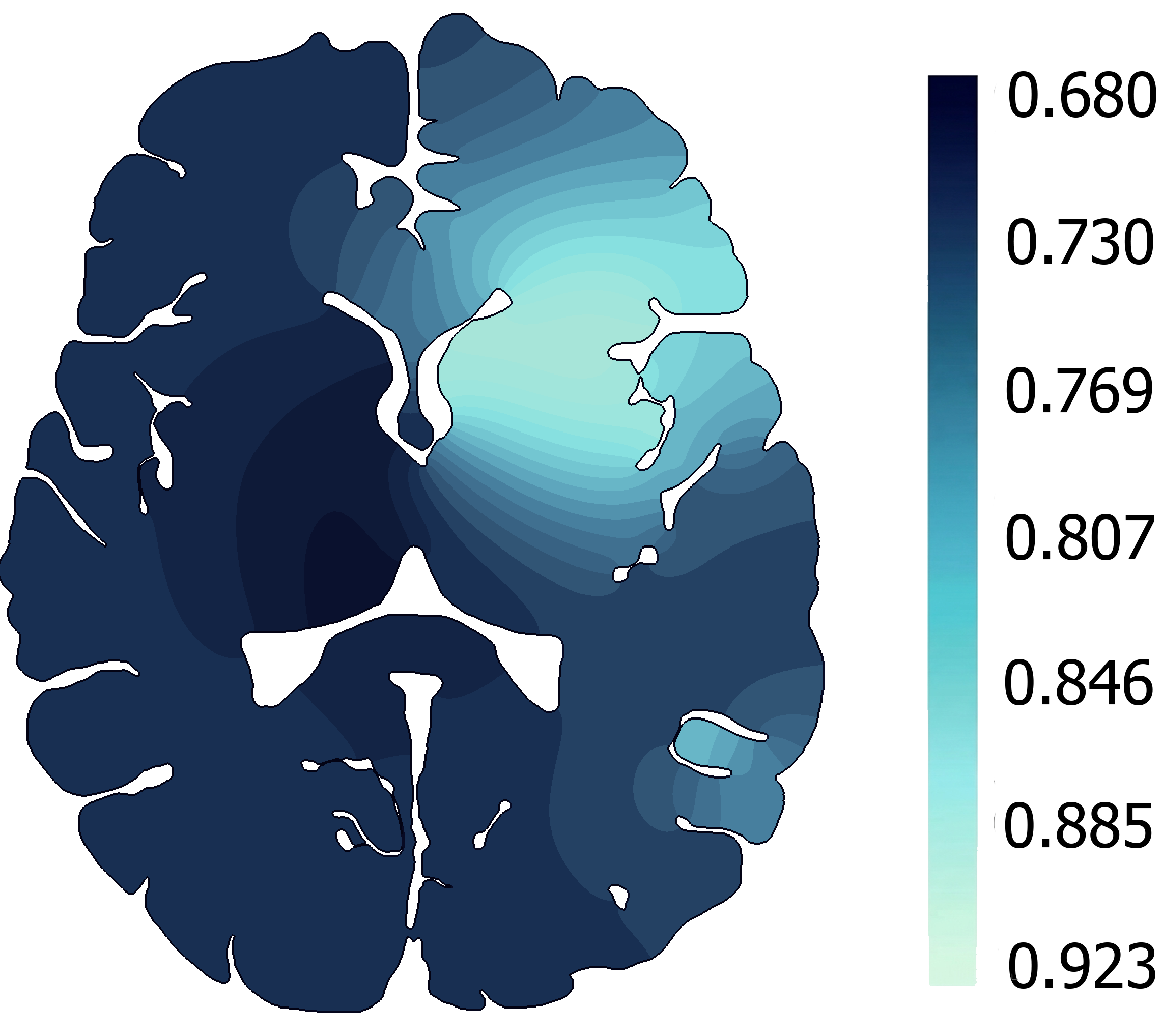}
         \caption{$t = 0.5$}
         \label{fig:s_6meses}
     \end{subfigure}
     \hfill
     \begin{subfigure}[b]{0.30\textwidth}
         \centering
         \includegraphics[width=\textwidth]{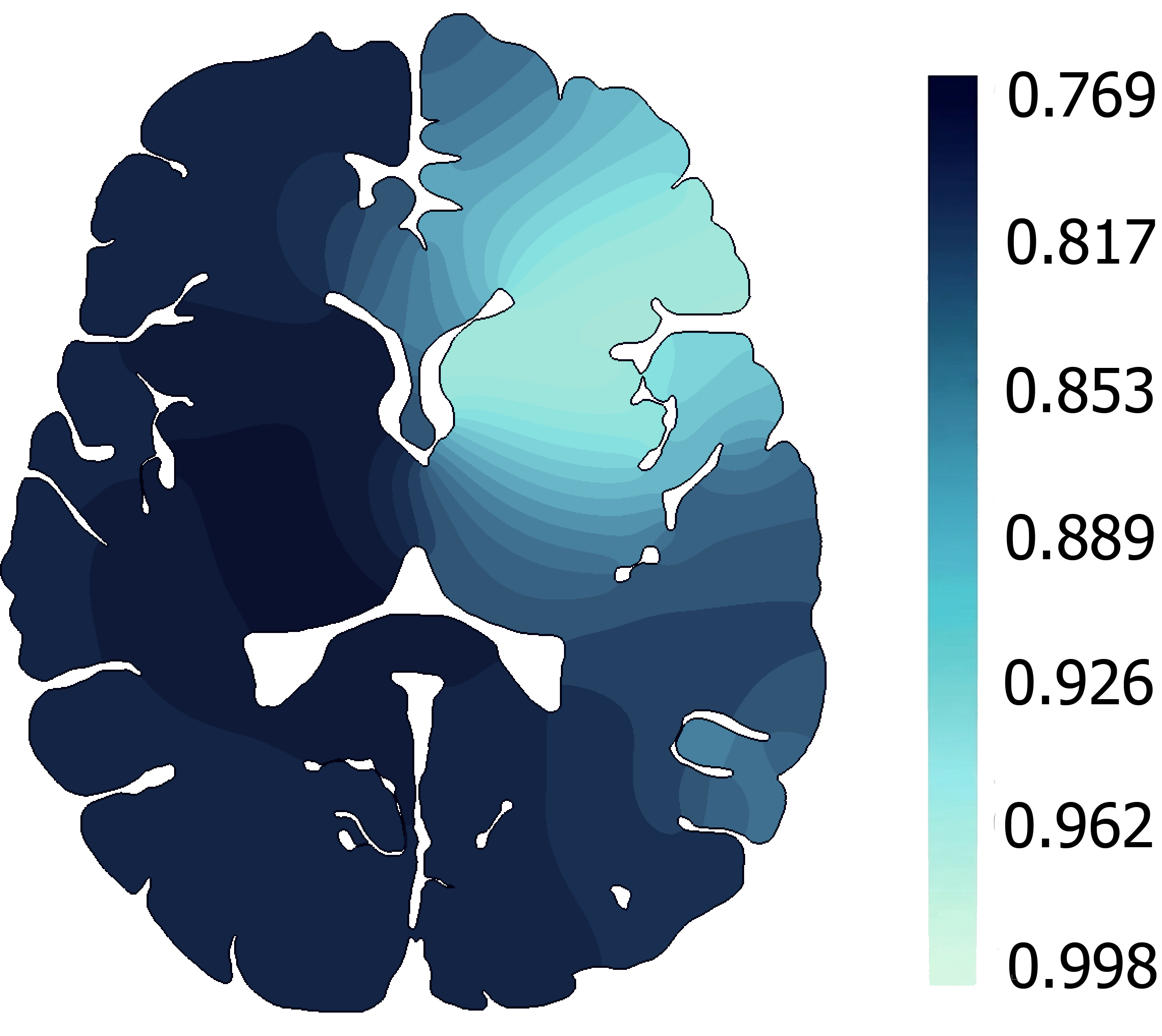}
         \caption{$t = 1$}
         \label{fig:s_1anio}
     \end{subfigure}
        \caption{Evolution up to $T=1.$}
        \label{fig:Evol_1anio}
\end{figure}
Figure \ref{fig:Evol_2anio} shows the formation of a cell aggregation point; that is, cells tend to concentrate in regions where the gradient of $\sigma$ is favorable. This aggregation phenomenon is directly related to the oxytaxis term and its sensitivity parameter $\chi$. In particular, the tumor reaches its maximum value of aggregation at $t = 1.664$ (see Figure \ref{fig:u_18meses}).
\begin{figure}[H]
     \centering
     \begin{subfigure}[b]{0.24\textwidth}
         \centering
         \includegraphics[width=\textwidth]{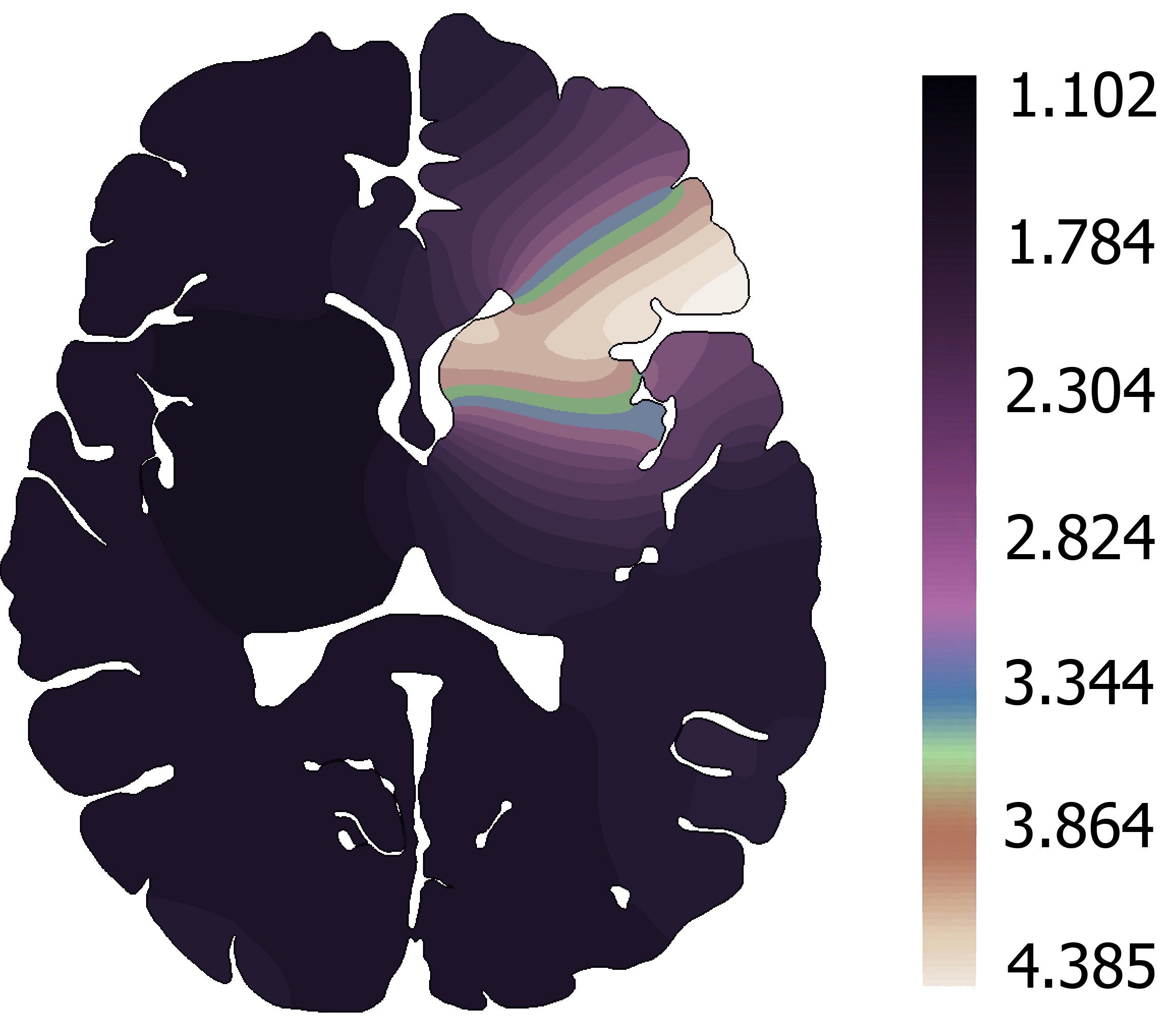}
         \caption{$t = 1.25$}
         \label{fig:u_15meses}
     \end{subfigure}
     \hfill
     \begin{subfigure}[b]{0.24\textwidth}
         \centering
         \includegraphics[width=\textwidth]{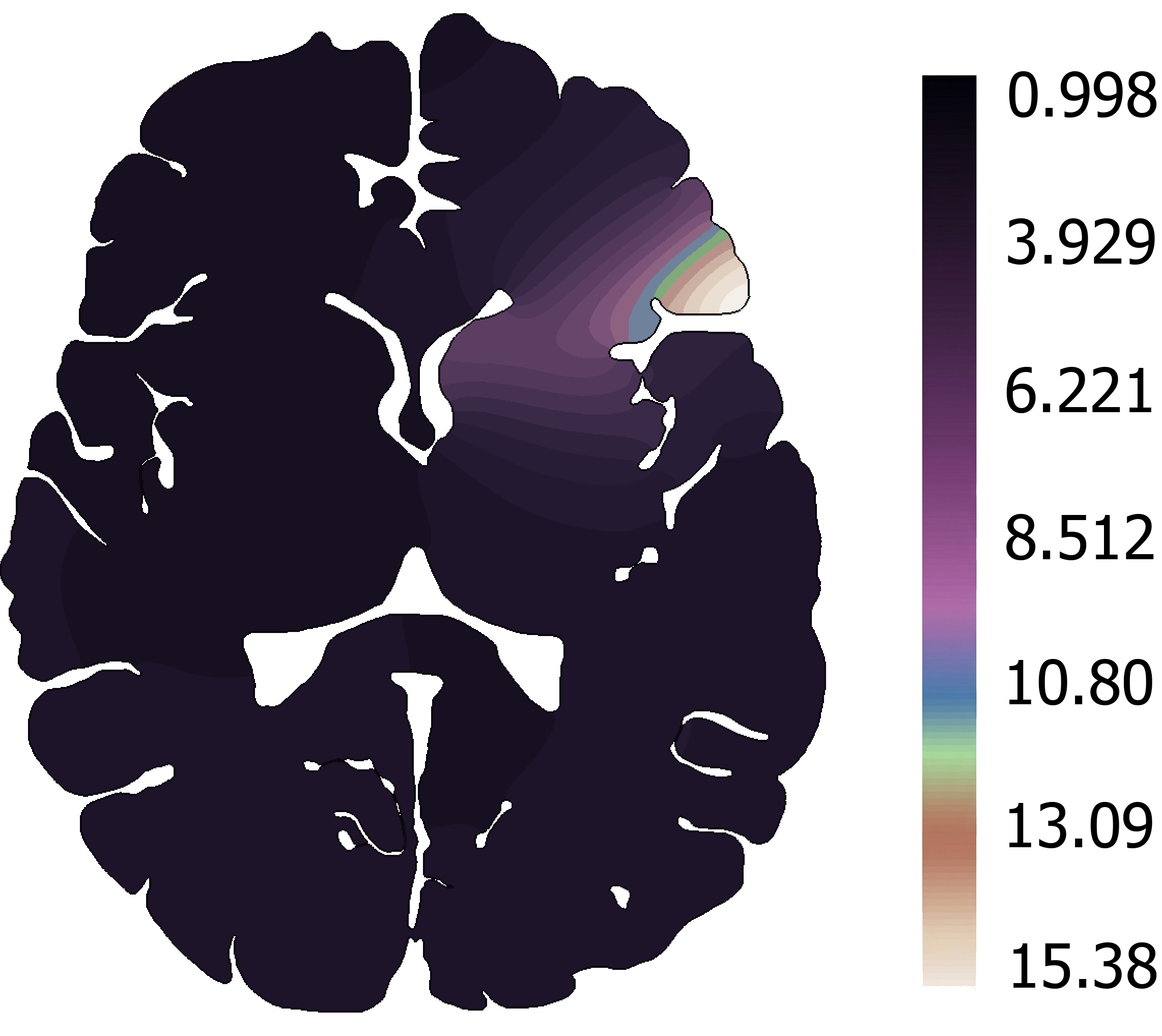}
         \caption{$t = 1.5$}
         \label{fig:u_182meses}
     \end{subfigure}
     \hfill
     \begin{subfigure}[b]{0.24\textwidth}
         \centering
         \includegraphics[width=\textwidth]{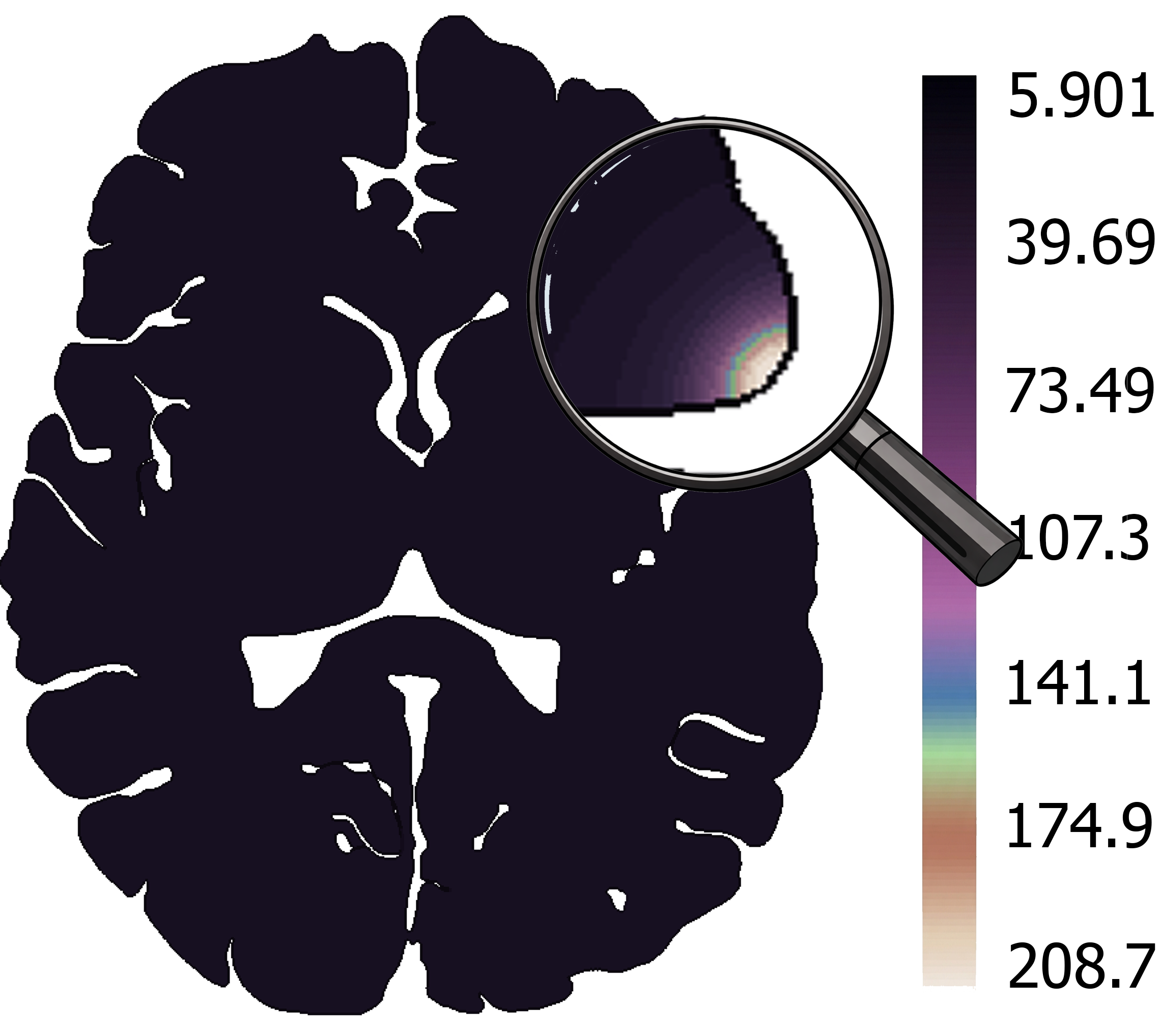}
         \caption{t = 1.664}
         \label{fig:u_18meses}
     \end{subfigure}
     \hfill
     \begin{subfigure}[b]{0.24\textwidth}
         \centering
         \includegraphics[width=\textwidth]{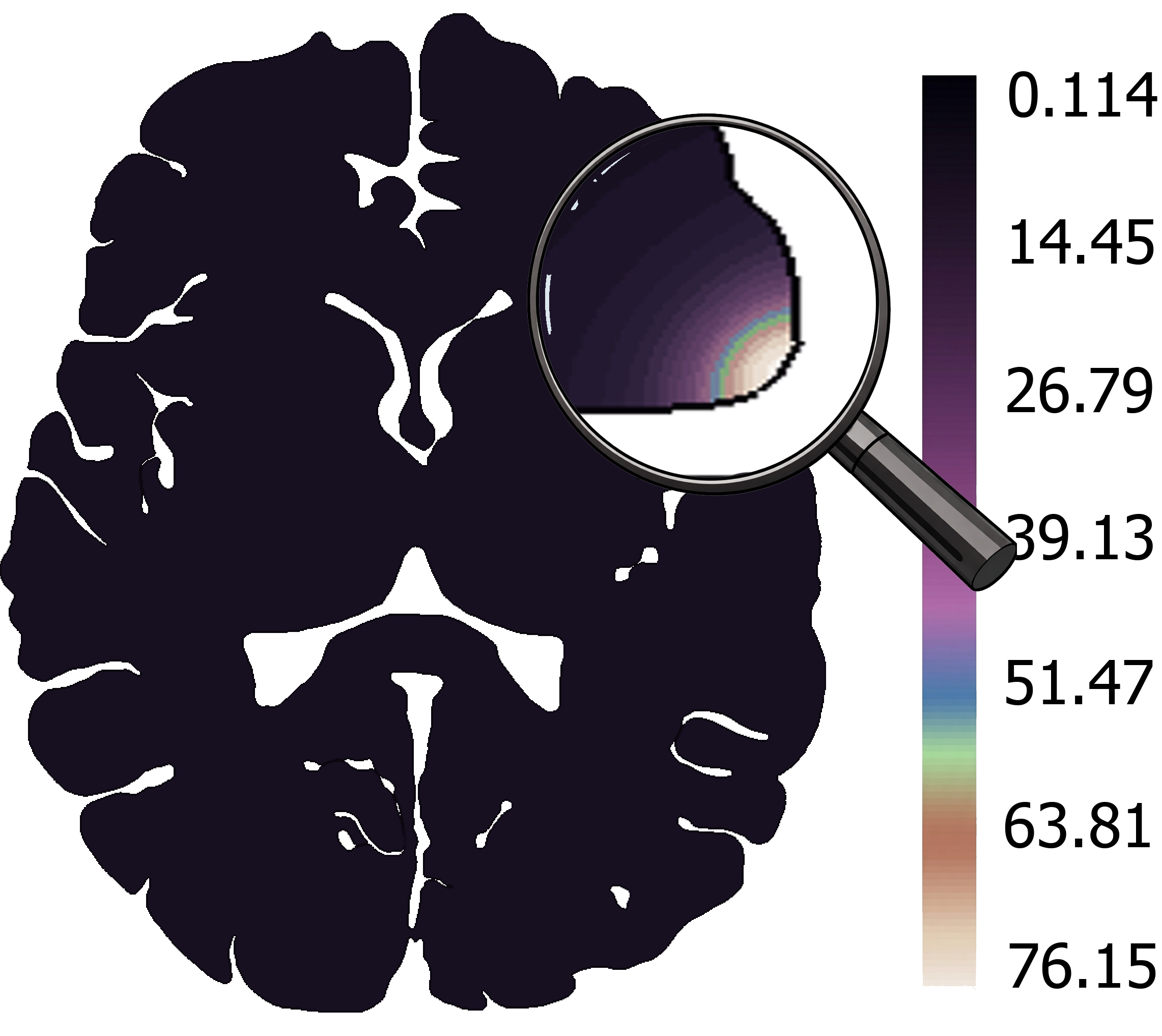}
         \caption{$t = 2$}
         \label{fig:u_2anio}
     \end{subfigure}

     \centering
     \begin{subfigure}[b]{0.24\textwidth}
         \centering
         \includegraphics[width=\textwidth]{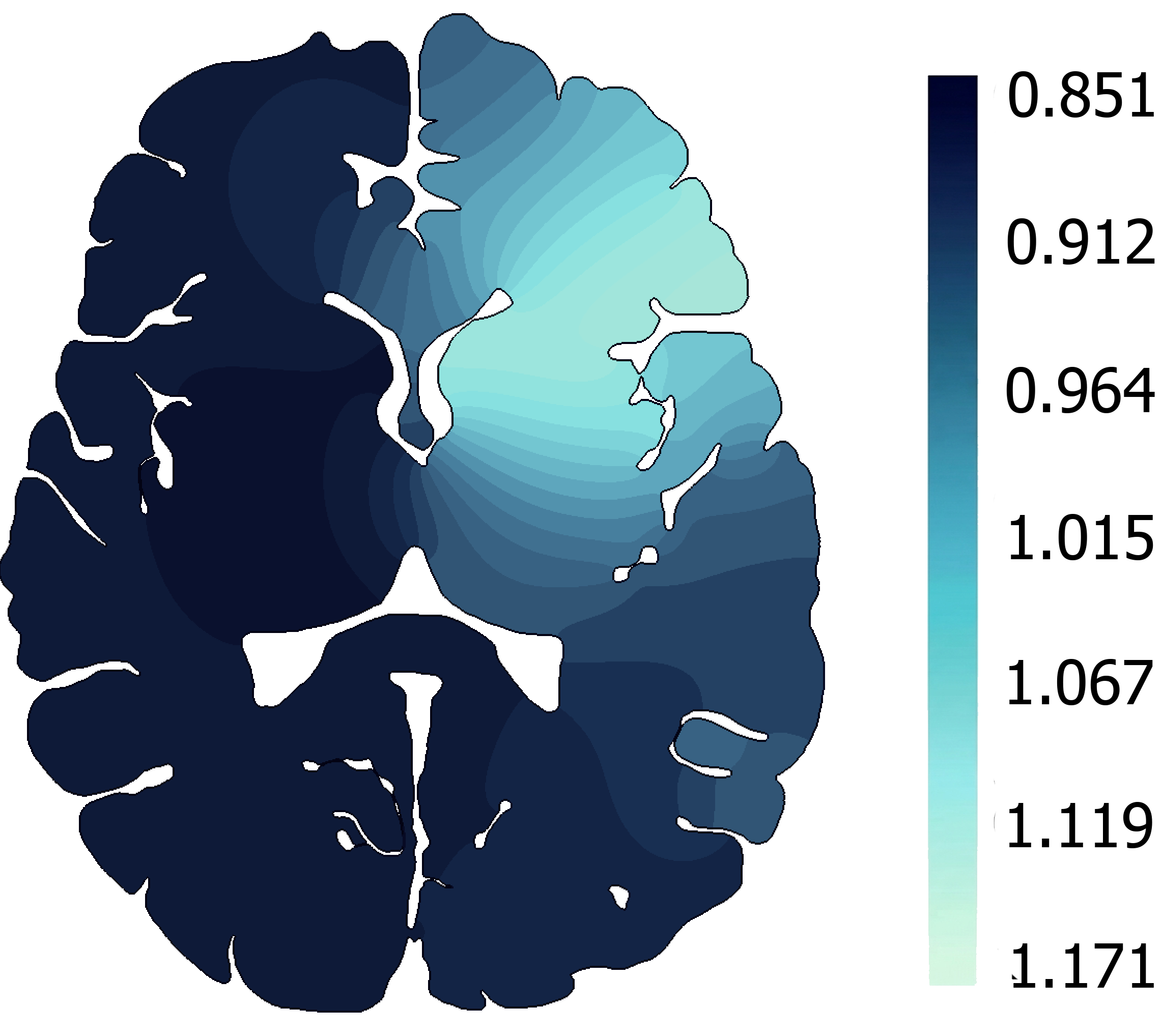}
         \caption{$t = 1.25$}
         \label{fig:s_15meses}
     \end{subfigure}
     \hfill
     \begin{subfigure}[b]{0.24\textwidth}
         \centering
         \includegraphics[width=\textwidth]{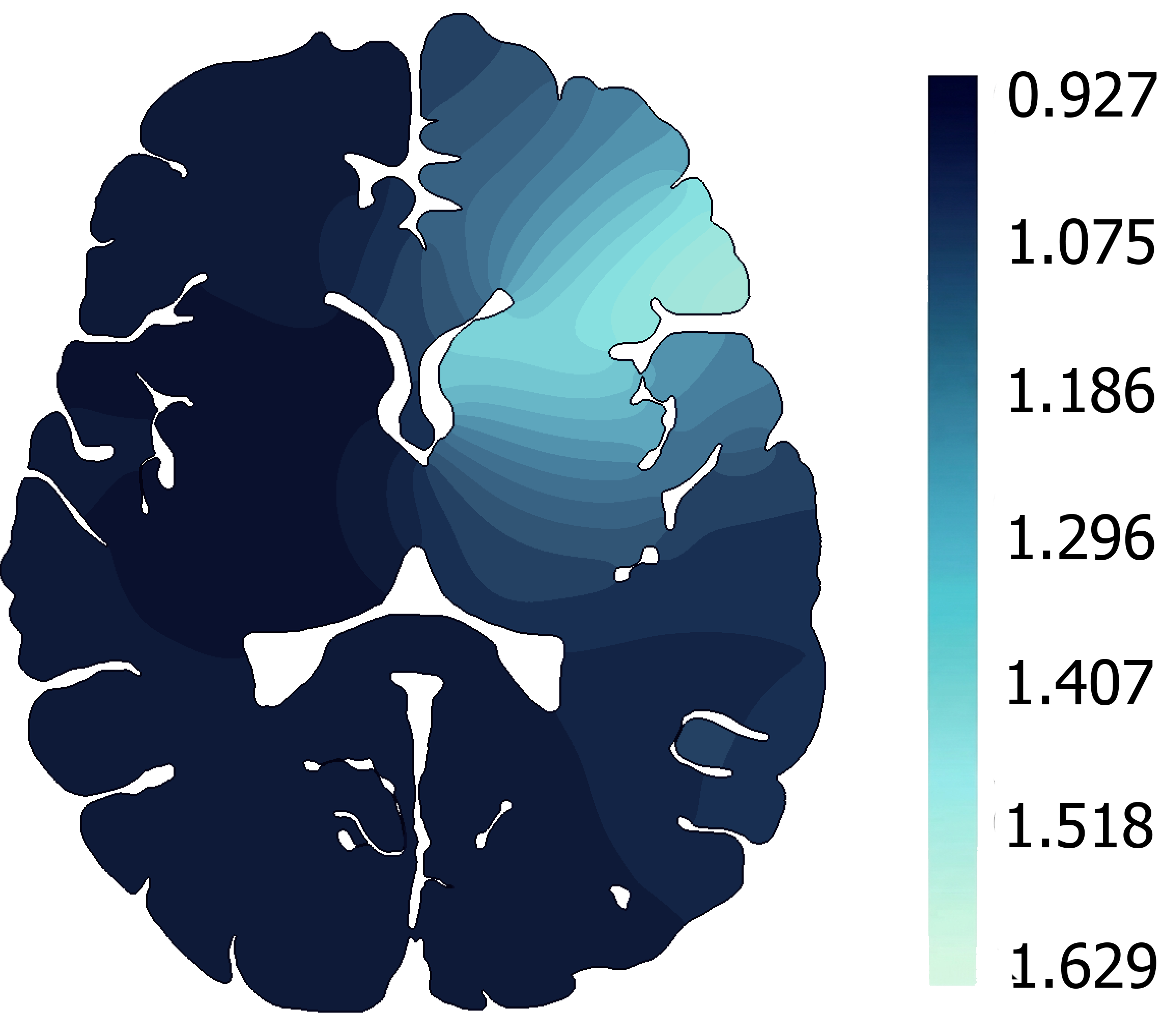}
         \caption{$t = 1.5$}
         \label{fig:s_18meses}
     \end{subfigure}
     \hfill
     \begin{subfigure}[b]{0.24\textwidth}
         \centering
         \includegraphics[width=\textwidth]{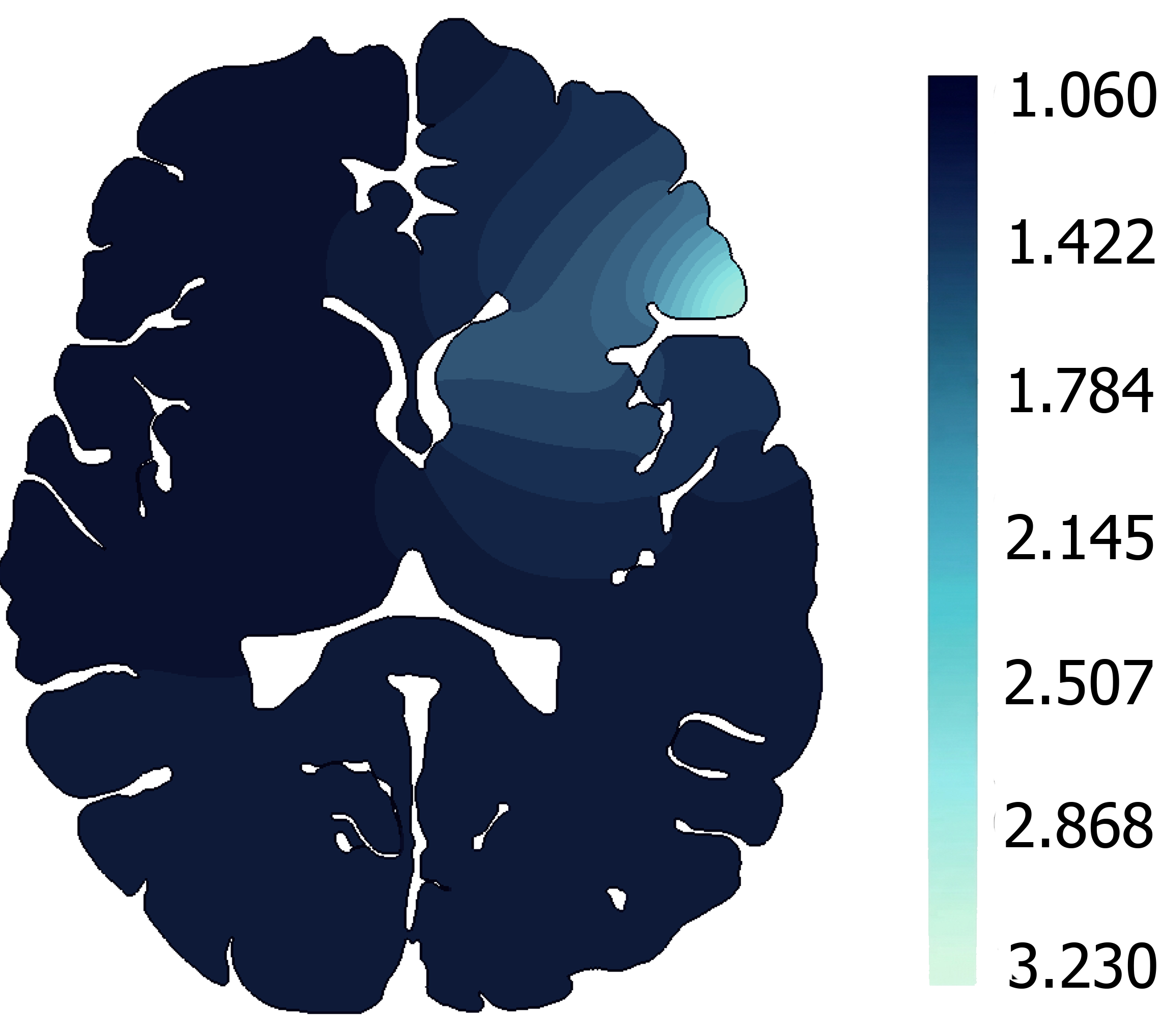}
         \caption{$t = 1.664$}
         \label{fig:s_199meses}
     \end{subfigure}
     \hfill
     \begin{subfigure}[b]{0.24\textwidth}
         \centering
         \includegraphics[width=\textwidth]{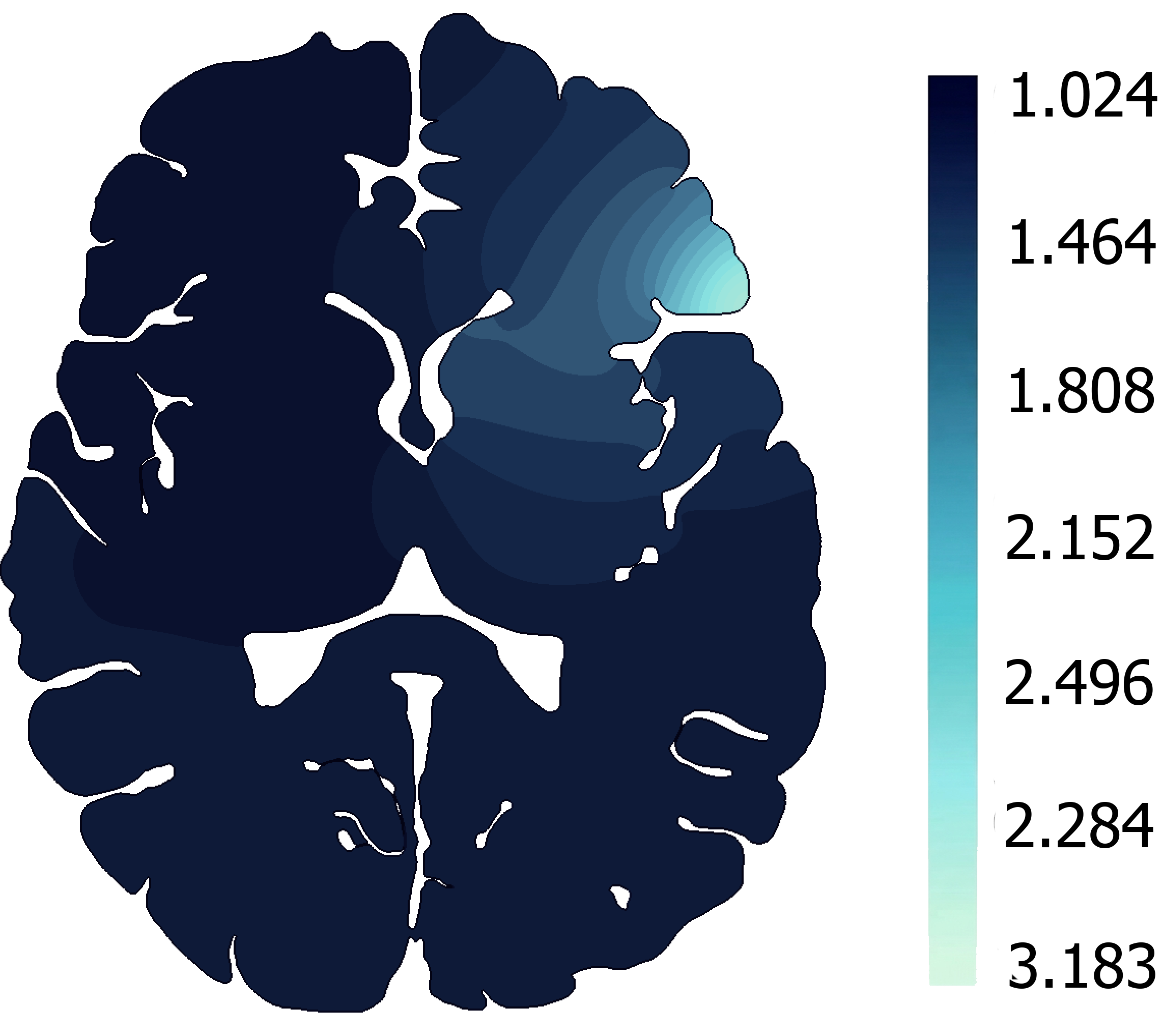}
         \caption{$t = 2$}
         \label{fig:s_2anio}
     \end{subfigure}
        \caption{Evolution up to $T=2$.}
        \label{fig:Evol_2anio}
\end{figure}
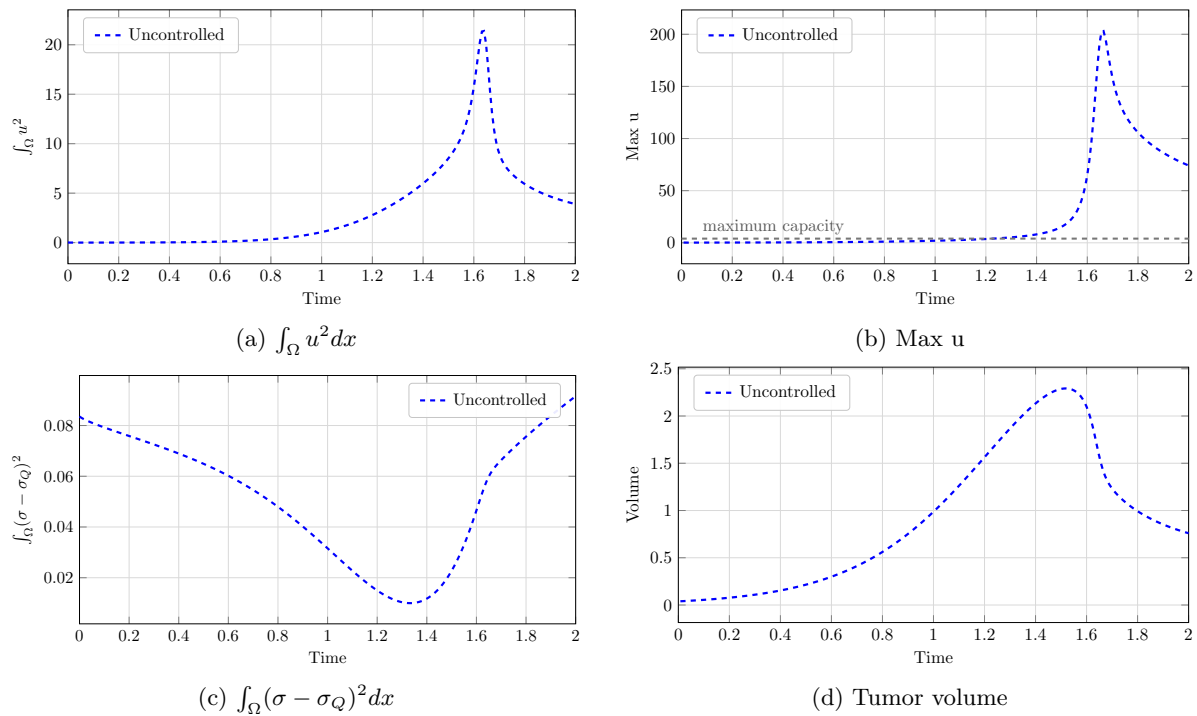
\begin{figure}[H]
     \centering
     \begin{subfigure}[b]{0.49\textwidth}
         \centering
         \resizebox{\linewidth}{!}{
            \begin{tikzpicture}
\begin{axis}[width=12cm, height=6.8cm, xmin=0, xmax=2, xlabel={Time}, ylabel={$\int_\Omega u^2$}, grid=both, ymajorgrids, grid style={gray!30}, legend cell align=left, legend style={font=\normalsize, at={(0.03,0.97)}, anchor=north west, fill=white, fill opacity=0.92, text opacity=1, draw=gray!45, rounded corners=2pt, inner xsep=6pt, inner ysep=4pt}, /pgf/number format/use comma=false, /pgf/number format/set decimal separator={.}, /pgf/number format/1000 sep={}, scaled x ticks=false, scaled y ticks=false, tick label style={font=\normalsize}, xticklabel style={font=\normalsize, /pgf/number format/fixed, /pgf/number format/precision=3, /pgf/number format/use comma=false, /pgf/number format/set decimal separator={.}, /pgf/number format/1000 sep={}}, yticklabel style={font=\normalsize, /pgf/number format/fixed, /pgf/number format/precision=6, /pgf/number format/use comma=false, /pgf/number format/set decimal separator={.}, /pgf/number format/1000 sep={}}, label style={font=\normalsize}, title style={font=\normalsize}, ]
\addplot+[color={rgb,255:red,0;green,0;blue,255}, dashed, line width=1.3pt, mark=none] coordinates {(0,0.0116615) (0.008,0.00424311) (0.016,0.00317803) (0.024,0.00292957) (0.032,0.00289467) (0.04,0.00294018) (0.048,0.00302428) (0.056,0.00313147) (0.064,0.00325527) (0.072,0.00339272) (0.08,0.00354233) (0.088,0.0037034) (0.096,0.00387562) (0.104,0.00405899) (0.112,0.00425369) (0.12,0.00446005) (0.128,0.00467847) (0.136,0.00490949) (0.144,0.00515368) (0.152,0.00541169) (0.16,0.00568423) (0.168,0.00597204) (0.176,0.00627594) (0.184,0.00659678) (0.192,0.00693547) (0.2,0.00729297) (0.208,0.00767031) (0.216,0.00806855) (0.224,0.00848883) (0.232,0.00893234) (0.24,0.00940035) (0.248,0.00989417) (0.256,0.0104152) (0.264,0.010965) (0.272,0.0115449) (0.28,0.0121568) (0.288,0.0128023) (0.296,0.0134832) (0.304,0.0142014) (0.312,0.014959) (0.32,0.015758) (0.328,0.0166007) (0.336,0.0174894) (0.344,0.0184265) (0.352,0.0194148) (0.36,0.0204568) (0.368,0.0215555) (0.376,0.0227138) (0.384,0.023935) (0.392,0.0252222) (0.4,0.0265791) (0.408,0.0280093) (0.416,0.0295165) (0.424,0.031105) (0.432,0.0327788) (0.44,0.0345425) (0.448,0.0364006) (0.456,0.0383581) (0.464,0.0404202) (0.472,0.0425921) (0.48,0.0448796) (0.488,0.0472885) (0.496,0.049825) (0.504,0.0524955) (0.512,0.0553069) (0.52,0.0582663) (0.528,0.061381) (0.536,0.0646588) (0.544,0.0681079) (0.552,0.0717367) (0.56,0.075554) (0.568,0.0795692) (0.576,0.0837919) (0.584,0.0882322) (0.592,0.0929006) (0.6,0.0978081) (0.608,0.102966) (0.616,0.108387) (0.624,0.114082) (0.632,0.120065) (0.64,0.12635) (0.648,0.13295) (0.656,0.139879) (0.664,0.147154) (0.672,0.154789) (0.68,0.162802) (0.688,0.171208) (0.696,0.180026) (0.704,0.189274) (0.712,0.19897) (0.72,0.209135) (0.728,0.219788) (0.736,0.23095) (0.744,0.242643) (0.752,0.25489) (0.76,0.267712) (0.768,0.281135) (0.776,0.295182) (0.784,0.309879) (0.792,0.325252) (0.8,0.341326) (0.808,0.358131) (0.816,0.375693) (0.824,0.394043) (0.832,0.413208) (0.84,0.43322) (0.848,0.45411) (0.856,0.475909) (0.864,0.49865) (0.872,0.522366) (0.88,0.54709) (0.888,0.572856) (0.896,0.5997) (0.904,0.627657) (0.912,0.656762) (0.92,0.687052) (0.928,0.718564) (0.936,0.751335) (0.944,0.785403) (0.952,0.820804) (0.96,0.857579) (0.968,0.895764) (0.976,0.935398) (0.984,0.97652) (0.992,1.01917) (1,1.06338) (1.008,1.1092) (1.016,1.15666) (1.024,1.20579) (1.032,1.25665) (1.04,1.30926) (1.048,1.36366) (1.056,1.41989) (1.064,1.47799) (1.072,1.53798) (1.08,1.59991) (1.088,1.6638) (1.096,1.7297) (1.104,1.79763) (1.112,1.86762) (1.12,1.93971) (1.128,2.01392) (1.136,2.09028) (1.144,2.16882) (1.152,2.24957) (1.16,2.33255) (1.168,2.41778) (1.176,2.5053) (1.184,2.59512) (1.192,2.68725) (1.2,2.78174) (1.208,2.87859) (1.216,2.97782) (1.224,3.07946) (1.232,3.18352) (1.24,3.29003) (1.248,3.399) (1.256,3.51046) (1.264,3.62442) (1.272,3.74091) (1.28,3.85996) (1.288,3.9816) (1.296,4.10585) (1.304,4.23275) (1.312,4.36233) (1.32,4.49465) (1.328,4.62975) (1.336,4.7677) (1.344,4.90855) (1.352,5.05238) (1.36,5.1993) (1.368,5.34939) (1.376,5.50278) (1.384,5.65961) (1.392,5.82005) (1.4,5.98428) (1.408,6.15252) (1.416,6.32504) (1.424,6.50214) (1.432,6.68419) (1.44,6.87163) (1.448,7.06496) (1.456,7.26481) (1.464,7.47192) (1.472,7.68722) (1.48,7.91183) (1.488,8.14715) (1.496,8.39492) (1.504,8.65736) (1.512,8.93729) (1.52,9.23838) (1.528,9.56537) (1.536,9.92455) (1.544,10.3242) (1.552,10.7754) (1.56,11.2928) (1.568,11.8959) (1.576,12.6095) (1.584,13.4647) (1.592,14.4956) (1.6,15.7317) (1.608,17.1765) (1.616,18.7634) (1.624,20.2864) (1.632,21.3396) (1.64,21.3924) (1.648,20.1318) (1.656,17.845) (1.664,15.2774) (1.672,13.0505) (1.68,11.3731) (1.688,10.1812) (1.696,9.33827) (1.704,8.72622) (1.712,8.26311) (1.72,7.89639) (1.728,7.59333) (1.736,7.33361) (1.744,7.10456) (1.752,6.89815) (1.76,6.70921) (1.768,6.53429) (1.776,6.37106) (1.784,6.21784) (1.792,6.07339) (1.8,5.93678) (1.808,5.80723) (1.816,5.68411) (1.824,5.5669) (1.832,5.45513) (1.84,5.3484) (1.848,5.24635) (1.856,5.14865) (1.864,5.05502) (1.872,4.9652) (1.88,4.87895) (1.888,4.79604) (1.896,4.7163) (1.904,4.63953) (1.912,4.56557) (1.92,4.49427) (1.928,4.4255) (1.936,4.35911) (1.944,4.295) (1.952,4.23305) (1.96,4.17316) (1.968,4.11524) (1.976,4.05919) (1.984,4.00494) (1.992,3.95241) (2,3.90151)};
\addlegendentry{Uncontrolled}
\end{axis}
\end{tikzpicture}
         }
         \caption{$\int_\Omega u^2dx$}
         \label{fig:int_u}
     \end{subfigure}
     \hfill
     \begin{subfigure}[b]{0.49\textwidth}
         \centering
         \resizebox{\linewidth}{!}{
            \begin{tikzpicture}
\begin{axis}[width=12cm, height=6.8cm, xmin=0, xmax=2, xlabel={Time}, ylabel={Max u}, grid=both, ymajorgrids, grid style={gray!30}, legend cell align=left, legend style={font=\normalsize, at={(0.03,0.97)}, anchor=north west, fill=white, fill opacity=0.92, text opacity=1, draw=gray!45, rounded corners=2pt, inner xsep=6pt, inner ysep=4pt}, tick label style={font=\normalsize}, label style={font=\normalsize}, title style={font=\small}, ]
\addplot+[color={rgb,255:red,0;green,0;blue,255}, opacity=1, dashed, line width=1.3pt, mark=none] coordinates {(0.008,0.256873) (0.016,0.206981) (0.024,0.187258) (0.032,0.178822) (0.04,0.176312) (0.048,0.177225) (0.056,0.180147) (0.064,0.184085) (0.072,0.188386) (0.08,0.192753) (0.088,0.197023) (0.096,0.20124) (0.104,0.205328) (0.112,0.209294) (0.12,0.213152) (0.128,0.216918) (0.136,0.22061) (0.144,0.224245) (0.152,0.227841) (0.16,0.231415) (0.168,0.234982) (0.176,0.238558) (0.184,0.242156) (0.192,0.245789) (0.2,0.249467) (0.208,0.253201) (0.216,0.257001) (0.224,0.260876) (0.232,0.264833) (0.24,0.268881) (0.248,0.273025) (0.256,0.277273) (0.264,0.28163) (0.272,0.286102) (0.28,0.290695) (0.288,0.295414) (0.296,0.300263) (0.304,0.305248) (0.312,0.310373) (0.32,0.315643) (0.328,0.321063) (0.336,0.326636) (0.344,0.332368) (0.352,0.338262) (0.36,0.344323) (0.368,0.350556) (0.376,0.356964) (0.384,0.363553) (0.392,0.370332) (0.4,0.377301) (0.408,0.384463) (0.416,0.391823) (0.424,0.399385) (0.432,0.407164) (0.44,0.41516) (0.448,0.423373) (0.456,0.431807) (0.464,0.440469) (0.472,0.449363) (0.48,0.458495) (0.488,0.46787) (0.496,0.477494) (0.504,0.487373) (0.512,0.497512) (0.52,0.507918) (0.528,0.518596) (0.536,0.529552) (0.544,0.540794) (0.552,0.552553) (0.56,0.56479) (0.568,0.577323) (0.576,0.590159) (0.584,0.603307) (0.592,0.616772) (0.6,0.630563) (0.608,0.644691) (0.616,0.659162) (0.624,0.673981) (0.632,0.689156) (0.64,0.704697) (0.648,0.720611) (0.656,0.736907) (0.664,0.753593) (0.672,0.770678) (0.68,0.788171) (0.688,0.806082) (0.696,0.824419) (0.704,0.843192) (0.712,0.862411) (0.72,0.882085) (0.728,0.902225) (0.736,0.92284) (0.744,0.943941) (0.752,0.965538) (0.76,0.987642) (0.768,1.01026) (0.776,1.03341) (0.784,1.05711) (0.792,1.08135) (0.8,1.10615) (0.808,1.13153) (0.816,1.1575) (0.824,1.18473) (0.832,1.21276) (0.84,1.24148) (0.848,1.2709) (0.856,1.30102) (0.864,1.33189) (0.872,1.3635) (0.88,1.39587) (0.888,1.42902) (0.896,1.46298) (0.904,1.49776) (0.912,1.53337) (0.92,1.56984) (0.928,1.60718) (0.936,1.64544) (0.944,1.68461) (0.952,1.72472) (0.96,1.76578) (0.968,1.80784) (0.976,1.8509) (0.984,1.895) (0.992,1.94019) (1,1.98646) (1.008,2.03386) (1.016,2.0824) (1.024,2.13212) (1.032,2.18307) (1.04,2.23528) (1.048,2.28878) (1.056,2.3436) (1.064,2.39978) (1.072,2.45736) (1.08,2.51641) (1.088,2.57697) (1.096,2.63906) (1.104,2.70275) (1.112,2.76808) (1.12,2.83511) (1.128,2.90391) (1.136,2.97457) (1.144,3.04712) (1.152,3.12164) (1.16,3.19822) (1.168,3.27694) (1.176,3.36014) (1.184,3.4507) (1.192,3.54427) (1.2,3.64098) (1.208,3.74103) (1.216,3.84459) (1.224,3.95187) (1.232,4.06309) (1.24,4.1785) (1.248,4.29835) (1.256,4.42294) (1.264,4.55258) (1.272,4.68762) (1.28,4.82843) (1.288,4.97545) (1.296,5.12911) (1.304,5.28995) (1.312,5.45852) (1.32,5.63545) (1.328,5.82144) (1.336,6.01727) (1.344,6.22381) (1.352,6.44204) (1.36,6.67306) (1.368,6.91813) (1.376,7.17865) (1.384,7.45626) (1.392,7.75277) (1.4,8.07028) (1.408,8.41121) (1.416,8.77836) (1.424,9.17497) (1.432,9.60484) (1.44,10.0724) (1.448,10.583) (1.456,11.1429) (1.464,11.7595) (1.472,12.4418) (1.48,13.2009) (1.488,14.0502) (1.496,15.0061) (1.504,16.0893) (1.512,17.3258) (1.52,18.7485) (1.528,20.3998) (1.536,22.3346) (1.544,24.6258) (1.552,27.3693) (1.56,30.6944) (1.568,34.777) (1.576,39.8574) (1.584,46.2656) (1.592,54.45) (1.6,65.0044) (1.608,78.6611) (1.616,96.1789) (1.624,117.973) (1.632,143.313) (1.64,169.278) (1.648,190.658) (1.656,202.373) (1.664,203.049) (1.672,195.597) (1.68,184.357) (1.688,172.722) (1.696,162.154) (1.704,153.087) (1.712,145.478) (1.72,139.101) (1.728,133.708) (1.736,129.08) (1.744,125.047) (1.752,121.477) (1.76,118.274) (1.768,115.365) (1.776,112.696) (1.784,110.228) (1.792,107.928) (1.8,105.774) (1.808,103.747) (1.816,101.831) (1.824,100.014) (1.832,98.2875) (1.84,96.6416) (1.848,95.0698) (1.856,93.5657) (1.864,92.124) (1.872,90.7402) (1.88,89.41) (1.888,88.1299) (1.896,86.8966) (1.904,85.7071) (1.912,84.5589) (1.92,83.4495) (1.928,82.3768) (1.936,81.3388) (1.944,80.3336) (1.952,79.3597) (1.96,78.4156) (1.968,77.4997) (1.976,76.6108) (1.984,75.7478) (1.992,74.9094) (2,74.0946)};
\addlegendentry{Uncontrolled}
\addplot+[color={rgb,255:red,119;green,119;blue,119},mark= none, dashed, line width=1.2pt, forget plot] coordinates {(0,4) (2,4)};
\node[anchor=west, mark= none, text={rgb,255:red,85;green,85;blue,85}] at (axis cs:0.06,15.02) {maximum capacity};
\end{axis}
\end{tikzpicture}
        }
         \caption{Max u}
         \label{fig:maxU_sinControl}
     \end{subfigure}
     \hfill
     \begin{subfigure}[b]{0.49\textwidth}
         \centering
         \resizebox{\linewidth}{!}{
         \begin{tikzpicture}
\begin{axis}[width=12cm, height=6.8cm, xmin=0, xmax=2, xlabel={Time}, ylabel={$\int_\Omega (\sigma - \sigma_Q)^2$}, grid=both, ymajorgrids, grid style={gray!30}, legend cell align=left, legend style={font=\normalsize, at={(0.97,0.97)}, anchor=north east, fill=white, fill opacity=0.92, text opacity=1, draw=gray!45, rounded corners=2pt, inner xsep=6pt, inner ysep=4pt}, /pgf/number format/use comma=false, /pgf/number format/set decimal separator={.}, /pgf/number format/1000 sep={}, scaled x ticks=false, scaled y ticks=false, tick label style={font=\normalsize}, xticklabel style={font=\normalsize, /pgf/number format/fixed, /pgf/number format/precision=3, /pgf/number format/use comma=false, /pgf/number format/set decimal separator={.}, /pgf/number format/1000 sep={}}, yticklabel style={font=\normalsize, /pgf/number format/fixed, /pgf/number format/precision=6, /pgf/number format/use comma=false, /pgf/number format/set decimal separator={.}, /pgf/number format/1000 sep={}}, label style={font=\normalsize}, title style={font=\normalsize}, ]
\addplot+[color={rgb,255:red,0;green,0;blue,255}, dashed, line width=1.3pt, mark=none] coordinates {(0,0.0835513) (0.008,0.0830233) (0.016,0.0825902) (0.024,0.0822019) (0.032,0.0818413) (0.04,0.0814996) (0.048,0.0811721) (0.056,0.0808555) (0.064,0.0805477) (0.072,0.0802471) (0.08,0.0799526) (0.088,0.0796631) (0.096,0.079378) (0.104,0.0790966) (0.112,0.0788183) (0.12,0.0785428) (0.128,0.0782697) (0.136,0.0779985) (0.144,0.0777291) (0.152,0.077461) (0.16,0.0771942) (0.168,0.0769283) (0.176,0.0766632) (0.184,0.0763987) (0.192,0.0761345) (0.2,0.0758706) (0.208,0.0756067) (0.216,0.0753427) (0.224,0.0750785) (0.232,0.0748139) (0.24,0.0745488) (0.248,0.074283) (0.256,0.0740165) (0.264,0.0737491) (0.272,0.0734806) (0.28,0.073211) (0.288,0.0729401) (0.296,0.0726678) (0.304,0.072394) (0.312,0.0721186) (0.32,0.0718415) (0.328,0.0715626) (0.336,0.0712817) (0.344,0.0709987) (0.352,0.0707136) (0.36,0.0704262) (0.368,0.0701364) (0.376,0.0698441) (0.384,0.0695491) (0.392,0.0692515) (0.4,0.068951) (0.408,0.0686476) (0.416,0.0683412) (0.424,0.0680316) (0.432,0.0677187) (0.44,0.0674024) (0.448,0.0670827) (0.456,0.0667593) (0.464,0.0664322) (0.472,0.0661014) (0.48,0.0657665) (0.488,0.0654277) (0.496,0.0650847) (0.504,0.0647374) (0.512,0.0643858) (0.52,0.0640296) (0.528,0.0636689) (0.536,0.0633035) (0.544,0.0629332) (0.552,0.062558) (0.56,0.0621778) (0.568,0.0617924) (0.576,0.0614018) (0.584,0.0610058) (0.592,0.0606044) (0.6,0.0601974) (0.608,0.0597847) (0.616,0.0593662) (0.624,0.0589419) (0.632,0.0585115) (0.64,0.0580751) (0.648,0.0576326) (0.656,0.0571837) (0.664,0.0567285) (0.672,0.0562669) (0.68,0.0557988) (0.688,0.0553241) (0.696,0.0548427) (0.704,0.0543545) (0.712,0.0538596) (0.72,0.0533577) (0.728,0.052849) (0.736,0.0523332) (0.744,0.0518104) (0.752,0.0512806) (0.76,0.0507437) (0.768,0.0501996) (0.776,0.0496484) (0.784,0.0490901) (0.792,0.0485246) (0.8,0.0479519) (0.808,0.0473722) (0.816,0.0467853) (0.824,0.0461914) (0.832,0.0455905) (0.84,0.0449826) (0.848,0.0443679) (0.856,0.0437464) (0.864,0.0431182) (0.872,0.0424834) (0.88,0.0418422) (0.888,0.0411947) (0.896,0.0405411) (0.904,0.0398815) (0.912,0.0392161) (0.92,0.0385453) (0.928,0.0378691) (0.936,0.0371878) (0.944,0.0365018) (0.952,0.0358112) (0.96,0.0351165) (0.968,0.034418) (0.976,0.033716) (0.984,0.0330109) (0.992,0.0323031) (1,0.031593) (1.008,0.0308812) (1.016,0.0301681) (1.024,0.0294542) (1.032,0.02874) (1.04,0.0280262) (1.048,0.0273133) (1.056,0.0266019) (1.064,0.0258928) (1.072,0.0251866) (1.08,0.024484) (1.088,0.0237857) (1.096,0.0230927) (1.104,0.0224056) (1.112,0.0217254) (1.12,0.0210529) (1.128,0.020389) (1.136,0.0197347) (1.144,0.019091) (1.152,0.0184589) (1.16,0.0178395) (1.168,0.0172338) (1.176,0.0166429) (1.184,0.0160681) (1.192,0.0155106) (1.2,0.0149715) (1.208,0.0144521) (1.216,0.0139537) (1.224,0.0134777) (1.232,0.0130254) (1.24,0.0125983) (1.248,0.0121977) (1.256,0.0118251) (1.264,0.0114821) (1.272,0.0111701) (1.28,0.0108907) (1.288,0.0106455) (1.296,0.0104362) (1.304,0.0102643) (1.312,0.0101316) (1.32,0.0100398) (1.328,0.0099906) (1.336,0.00998575) (1.344,0.010027) (1.352,0.0101163) (1.36,0.0102553) (1.368,0.010446) (1.376,0.0106901) (1.384,0.0109897) (1.392,0.0113465) (1.4,0.0117625) (1.408,0.0122397) (1.416,0.0127798) (1.424,0.013385) (1.432,0.0140571) (1.44,0.0147981) (1.448,0.0156099) (1.456,0.0164944) (1.464,0.0174536) (1.472,0.0184892) (1.48,0.0196033) (1.488,0.0207976) (1.496,0.0220738) (1.504,0.0234337) (1.512,0.0248788) (1.52,0.0264106) (1.528,0.0280305) (1.536,0.0297394) (1.544,0.0315383) (1.552,0.0334273) (1.56,0.0354061) (1.568,0.0374734) (1.576,0.0396262) (1.584,0.0418592) (1.592,0.0441632) (1.6,0.0465231) (1.608,0.0489144) (1.616,0.0512993) (1.624,0.0536235) (1.632,0.0558166) (1.64,0.0578051) (1.648,0.0595364) (1.656,0.0610025) (1.664,0.0622406) (1.672,0.0633093) (1.68,0.0642641) (1.688,0.0651464) (1.696,0.0659829) (1.704,0.0667898) (1.712,0.0675769) (1.72,0.0683494) (1.728,0.0691106) (1.736,0.0698623) (1.744,0.0706056) (1.752,0.0713411) (1.76,0.0720694) (1.768,0.0727908) (1.776,0.0735055) (1.784,0.0742137) (1.792,0.0749158) (1.8,0.075612) (1.808,0.0763023) (1.816,0.0769871) (1.824,0.0776664) (1.832,0.0783406) (1.84,0.0790098) (1.848,0.0796741) (1.856,0.0803336) (1.864,0.0809887) (1.872,0.0816394) (1.88,0.0822858) (1.888,0.0829281) (1.896,0.0835664) (1.904,0.0842009) (1.912,0.0848316) (1.92,0.0854588) (1.928,0.0860825) (1.936,0.0867028) (1.944,0.0873198) (1.952,0.0879337) (1.96,0.0885445) (1.968,0.0891524) (1.976,0.0897575) (1.984,0.0903598) (1.992,0.0909595) (2,0.0915566)};
\addlegendentry{Uncontrolled}
\end{axis}
\end{tikzpicture}
         }
         \caption{$\int_\Omega (\sigma - \sigma_Q)^2dx$}
         \label{fig:int_sigma}
     \end{subfigure}
     \hfill
     \begin{subfigure}[b]{0.49\textwidth}
         \centering
         \resizebox{\linewidth}{!}{
            \begin{tikzpicture}
                \begin{axis}[width=12cm, height=6.8cm, xmin=0, xmax=2, xlabel={Time}, ylabel={Volume}, grid=both, ymajorgrids, grid style={gray!30}, legend cell align=left, legend style={font=\normalsize, at={(0.03,0.97)}, anchor=north west, fill=white, fill opacity=0.92, text opacity=1, draw=gray!45, rounded corners=2pt, inner xsep=6pt, inner ysep=4pt}, tick label style={font=\normalsize}, label style={font=\normalsize}, title style={font=\normalsize}, ]
                \addplot+[color={rgb,255:red,0;green,0;blue,255}, dashed, line width=1.3pt, mark=none] coordinates {(0.008,0.0396039) (0.016,0.04072) (0.024,0.0418715) (0.032,0.0430569) (0.04,0.0442763) (0.048,0.0455303) (0.056,0.0468197) (0.064,0.0481454) (0.072,0.0495084) (0.08,0.0509096) (0.088,0.0523501) (0.096,0.0538308) (0.104,0.0553529) (0.112,0.0569175) (0.12,0.0585258) (0.128,0.0601788) (0.136,0.0618778) (0.144,0.0636241) (0.152,0.0654189) (0.16,0.0672634) (0.168,0.0691591) (0.176,0.0711073) (0.184,0.0731093) (0.192,0.0751666) (0.2,0.0772807) (0.208,0.079453) (0.216,0.0816851) (0.224,0.0839786) (0.232,0.086335) (0.24,0.088756) (0.248,0.0912433) (0.256,0.0937986) (0.264,0.0964237) (0.272,0.0991203) (0.28,0.10189) (0.288,0.104736) (0.296,0.107658) (0.304,0.11066) (0.312,0.113742) (0.32,0.116908) (0.328,0.12016) (0.336,0.123498) (0.344,0.126927) (0.352,0.130447) (0.36,0.134061) (0.368,0.137772) (0.376,0.141581) (0.384,0.145491) (0.392,0.149506) (0.4,0.153626) (0.408,0.157855) (0.416,0.162195) (0.424,0.166649) (0.432,0.17122) (0.44,0.17591) (0.448,0.180722) (0.456,0.185659) (0.464,0.190723) (0.472,0.195918) (0.48,0.201247) (0.488,0.206713) (0.496,0.212318) (0.504,0.218066) (0.512,0.223959) (0.52,0.230002) (0.528,0.236197) (0.536,0.242547) (0.544,0.249056) (0.552,0.255727) (0.56,0.262564) (0.568,0.269569) (0.576,0.276746) (0.584,0.284099) (0.592,0.291631) (0.6,0.299346) (0.608,0.307246) (0.616,0.315336) (0.624,0.323619) (0.632,0.332098) (0.64,0.340778) (0.648,0.349661) (0.656,0.358752) (0.664,0.368053) (0.672,0.377568) (0.68,0.387302) (0.688,0.397256) (0.696,0.407435) (0.704,0.417843) (0.712,0.428482) (0.72,0.439356) (0.728,0.450468) (0.736,0.461822) (0.744,0.473421) (0.752,0.485268) (0.76,0.497365) (0.768,0.509717) (0.776,0.522326) (0.784,0.535194) (0.792,0.548325) (0.8,0.561721) (0.808,0.575384) (0.816,0.589317) (0.824,0.603523) (0.832,0.618002) (0.84,0.632758) (0.848,0.647791) (0.856,0.663104) (0.864,0.678697) (0.872,0.694573) (0.88,0.710731) (0.888,0.727173) (0.896,0.743899) (0.904,0.76091) (0.912,0.778206) (0.92,0.795785) (0.928,0.813649) (0.936,0.831796) (0.944,0.850225) (0.952,0.868935) (0.96,0.887924) (0.968,0.90719) (0.976,0.926731) (0.984,0.946544) (0.992,0.966626) (1,0.986974) (1.008,1.00758) (1.016,1.02845) (1.024,1.04957) (1.032,1.07094) (1.04,1.09256) (1.048,1.11441) (1.056,1.13649) (1.064,1.15879) (1.072,1.18132) (1.08,1.20405) (1.088,1.22698) (1.096,1.25011) (1.104,1.27342) (1.112,1.29691) (1.12,1.32056) (1.128,1.34437) (1.136,1.36832) (1.144,1.39241) (1.152,1.41661) (1.16,1.44092) (1.168,1.46533) (1.176,1.48981) (1.184,1.51437) (1.192,1.53897) (1.2,1.56362) (1.208,1.58828) (1.216,1.61295) (1.224,1.6376) (1.232,1.66223) (1.24,1.6868) (1.248,1.71131) (1.256,1.73573) (1.264,1.76005) (1.272,1.78423) (1.28,1.80826) (1.288,1.83212) (1.296,1.85578) (1.304,1.87922) (1.312,1.90241) (1.32,1.92532) (1.328,1.94793) (1.336,1.97021) (1.344,1.99213) (1.352,2.01365) (1.36,2.03475) (1.368,2.05538) (1.376,2.07552) (1.384,2.09512) (1.392,2.11415) (1.4,2.13255) (1.408,2.15029) (1.416,2.16731) (1.424,2.18356) (1.432,2.19899) (1.44,2.21353) (1.448,2.22712) (1.456,2.23969) (1.464,2.25116) (1.472,2.26143) (1.48,2.27042) (1.488,2.27802) (1.496,2.2841) (1.504,2.28852) (1.512,2.29112) (1.52,2.29171) (1.528,2.29006) (1.536,2.2859) (1.544,2.2789) (1.552,2.26866) (1.56,2.25466) (1.568,2.23626) (1.576,2.21263) (1.584,2.18275) (1.592,2.14531) (1.6,2.09876) (1.608,2.0414) (1.616,1.97168) (1.624,1.88897) (1.632,1.79476) (1.64,1.69388) (1.648,1.59431) (1.656,1.50443) (1.664,1.42917) (1.672,1.36874) (1.68,1.32046) (1.688,1.28101) (1.696,1.2477) (1.704,1.21865) (1.712,1.19261) (1.72,1.1688) (1.728,1.14671) (1.736,1.12601) (1.744,1.10647) (1.752,1.08793) (1.76,1.07028) (1.768,1.05343) (1.776,1.03731) (1.784,1.02187) (1.792,1.00705) (1.8,0.992817) (1.808,0.979132) (1.816,0.965963) (1.824,0.953279) (1.832,0.941054) (1.84,0.929264) (1.848,0.917886) (1.856,0.906898) (1.864,0.896284) (1.872,0.886023) (1.88,0.876102) (1.888,0.866503) (1.896,0.857213) (1.904,0.848219) (1.912,0.839509) (1.92,0.831071) (1.928,0.822894) (1.936,0.814969) (1.944,0.807284) (1.952,0.799832) (1.96,0.792605) (1.968,0.785593) (1.976,0.778789) (1.984,0.772187) (1.992,0.765779) (2,0.759558)};
                \addlegendentry{Uncontrolled}
                \end{axis}
                \end{tikzpicture}
         }
         \caption{Tumor volume}
         \label{fig:volumen_sinControl}
     \end{subfigure}
    \caption{Evolution over time.}
        \label{fig:u_sinControl_datos}
        \end{figure}
Figure  \ref{fig:int_u} and Figure \ref{fig:int_sigma}  show the evolution of the amount of the square of cells and oxygen with respect to a desired stated, namely, absence of tumor and $\sigma_Q$  threshold of oxygen, respectively.   Figure~\ref{fig:maxU_sinControl} shows the maximum value of  $u$ as a function of time over the time horizon, where the influence of the previously identified aggregation point becomes more evident. In addition, it shows the influence of the consumption effect of the logistic term $-\rho(\sigma)u^2$ to control that aggregation point, but at the same time it causes a reduction in tumor volume (see Figure \ref{fig:volumen_sinControl}).     
    
\subsection{Experiment 2 (assuming \texorpdfstring{$c_{max} = 0.25$}{cmax = 0.25})}
In this experiment, we analyze the behavior of the system with control assuming a small amount of maximum  therapy $\ControlC$; specifically, for this experiment, we consider the limit for the global therapeutic quantity $c_{max} = 0.25$.  Figure \ref{fig:robustes_controles_1} shows a comparison of the initial control versus the optimal control obtained for two cases: the first one includes the cost functional with the terms involving the final time $T$ (i.e., $l_j = 1$ for $j=1,2$), and the second one excluding those terms (i.e., $l_j = 0$ for $j = 1,2$). Essentially, when the cost functional is not considered in the final time $T$, the therapy acts earlier.   
\noindent
\begin{figure}[H]
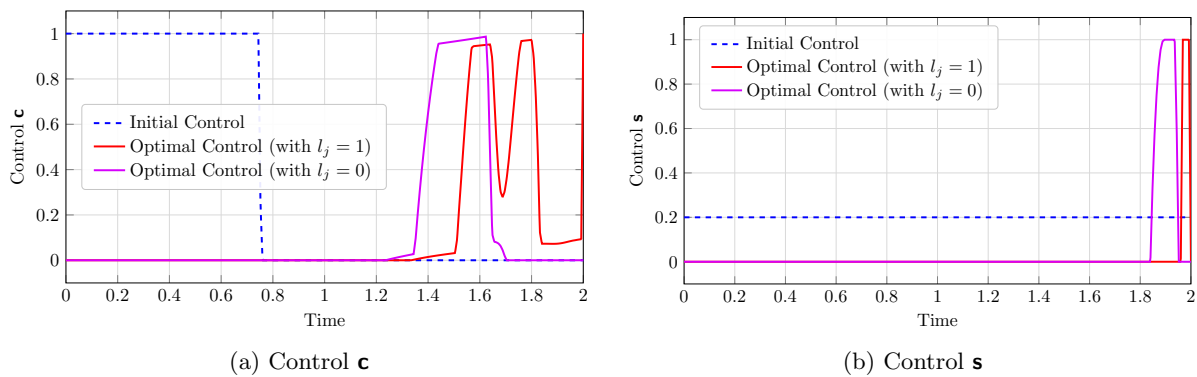

     \centering
     \begin{subfigure}[b]{0.49\textwidth}
         \centering
         \resizebox{\linewidth}{!}{
}
         \caption{Control \ControlC}
         \label{fig:robustes_controlC_1}
     \end{subfigure}
     \hfill
     \begin{subfigure}[b]{0.49\textwidth}
         \centering
         \resizebox{\linewidth}{!}{
         %
         }
         \caption{Control \ControlS}
         \label{fig:robustes_controlS_1}
     \end{subfigure}
    
    \caption{Optimal control compared with initial control.}
    \label{fig:robustes_controles_1}
\end{figure}
Figures \ref{fig:robustes_u_1} shows a comparison of the two cases of the optimal control versus the uncontrolled case, where low effectiveness is observed due to the limited therapy available for control $\ControlC$. In particular, we see that the case where $l_j = 0$ (i.e., excluding the final time terms $T$ in the cost functional) is more effective, giving higher emphasis on controlling the tumor and oxygen levels throughout the entire time interval.\\

Figures \ref{fig:func_exp1_lj1}-\ref{fig:func_exp1_lj0} show the evolution of the cost functional during the simulation. In particular, we observe a stabilization is in the final iterations, indicating a good convergence of the Adam algorithm. In the Figures \ref{fig:func_exp1_lj1}-\ref{fig:func_exp1_lj0}, green color indicates a decrease in the functional, red indicates an increase in the functional, and yellow indicates that the functional has stabilized. Additionally, Figures \ref{fig:controlC_exp1_pert}-\ref{fig:controlS_exp1_pert} illustrate the nature of a local minimum of the cost functional, perturbing the optimal values $\ControlC$ and $\ControlS$ in a specific direction.
\begin{figure}[H]
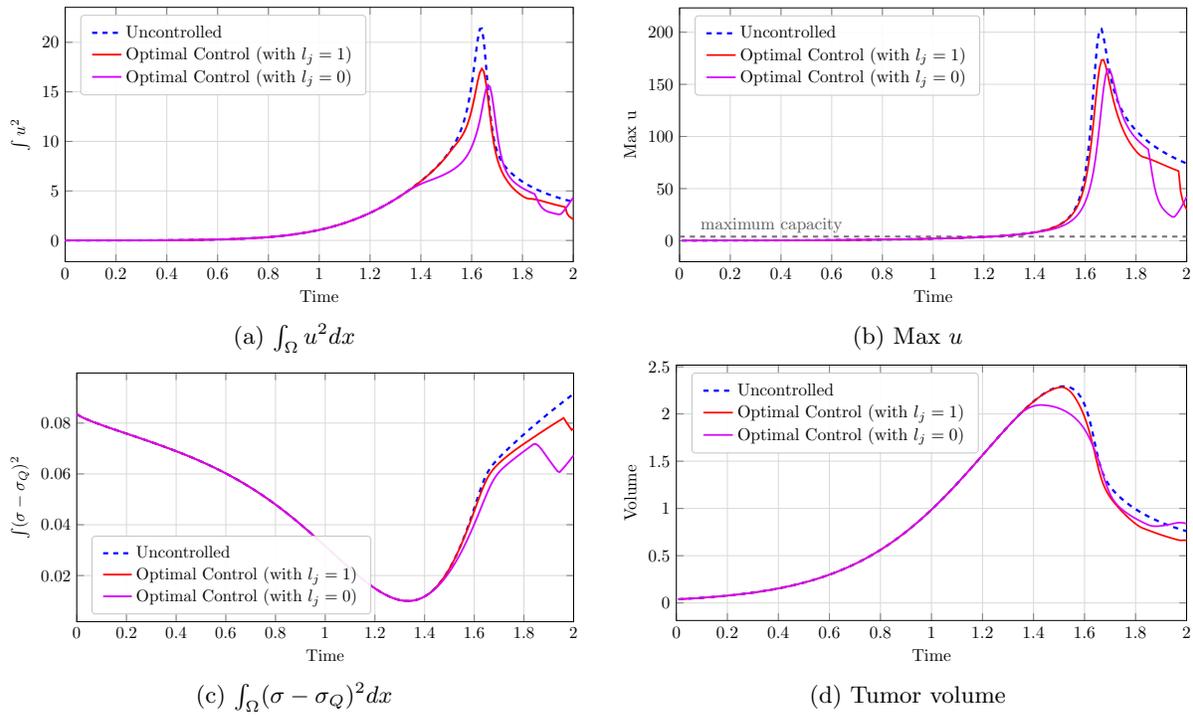

     \centering
     \begin{subfigure}[b]{0.49\textwidth}
         \centering
         \resizebox{\linewidth}{!}{

         }
         \caption{$\int_\Omega u^2 dx$}
         \label{fig:robustes_min_u_1}
     \end{subfigure}
     \hfill
     \begin{subfigure}[b]{0.49\textwidth}
         \centering
         \resizebox{\linewidth}{!}{
         %
         }
         \caption{Max $u$}
         \label{fig:robustes_max_u_1}
     \end{subfigure}
     \hfill
     \begin{subfigure}[b]{0.49\textwidth}
         \centering
         \resizebox{\linewidth}{!}{
         %
         }
         \caption{$\int_\Omega (\sigma-\sigma_Q)^2dx$}
         \label{fig:robustes_vol_u_2}
     \end{subfigure}
     \hfill
     \begin{subfigure}[b]{0.49\textwidth}
         \centering
         \resizebox{\linewidth}{!}{
         %
         }
         \caption{Tumor volume}
         \label{fig:robustes_vol_u_1}
     \end{subfigure}
    \caption{Evolution over time.}
    \label{fig:robustes_u_1}
\end{figure}
\begin{figure}[H]
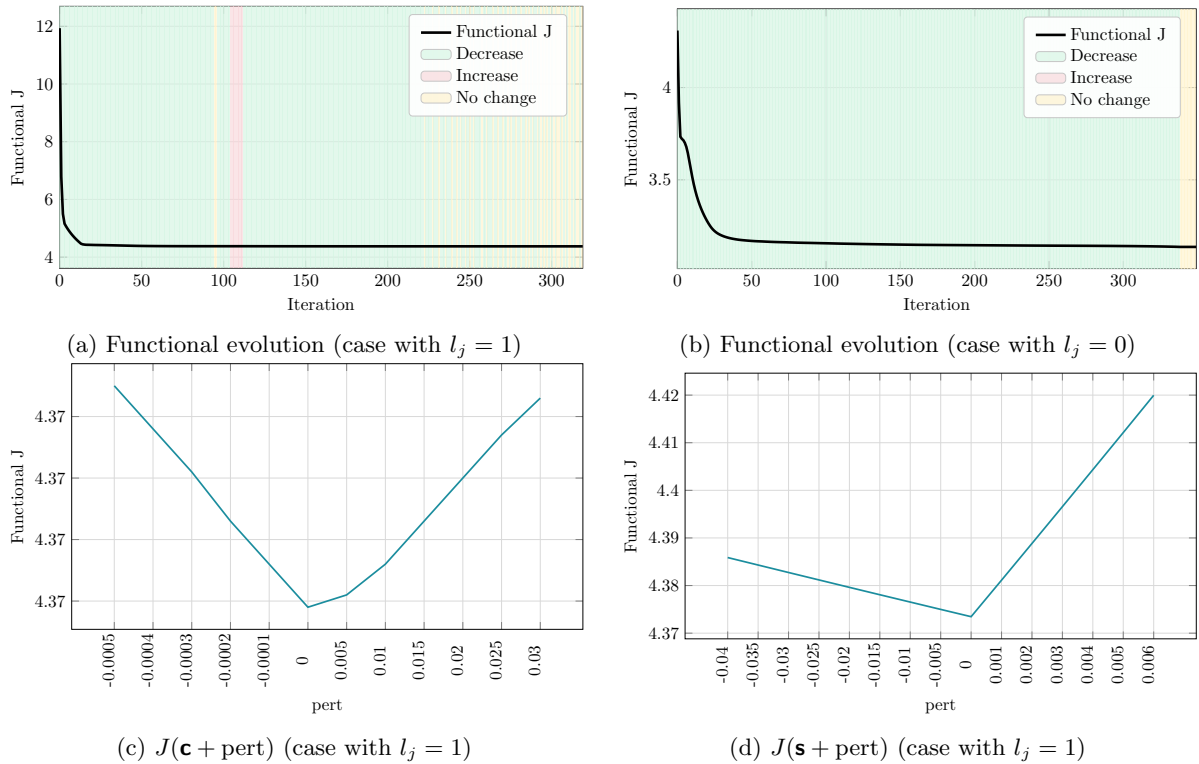

\begin{subfigure}[b]{0.49\textwidth}
    \centering
    \resizebox{\linewidth}{!}{
%
    }
    \caption{Functional evolution (case with $l_j=1$)}
    \label{fig:func_exp1_lj1}
    \end{subfigure}
    \begin{subfigure}[b]{0.49\linewidth}
        \resizebox{\linewidth}{!}{
        %
        }
        \caption{Functional evolution (case with $l_j=0$)}
        \label{fig:func_exp1_lj0}
    \end{subfigure}
            \hfill
     \begin{subfigure}[b]{0.49\textwidth}
         \centering
         \resizebox{\linewidth}{!}{
         %
         }
         \caption{$J(\ControlC+\text{pert})$ (case with $l_j=1$)}
         \label{fig:controlC_exp1_pert}
     \end{subfigure}
     \hfill
     \begin{subfigure}[b]{0.49\textwidth}
         \centering
         \resizebox{\linewidth}{!}{
%
         }
         \caption{$J(\ControlS+\text{pert})$ (case with $l_j = 1$)}
         \label{fig:controlS_exp1_pert}
     \end{subfigure}
     \caption{Stability of functional and minimum local.}
\end{figure}

\subsection{Experiment 3 (assuming $c_{max} = 0.75$)}
In this experiment, we analyze the evolution of the optimal control problem under the assumption of a greater maximum control capacity $\ControlC$, specifically, we consider the limit for the global therapeutic quantity $c_{max} = 0.75$. Figure \ref{fig:robustes_controles_2} displays the initial control together with the corresponding optimal control obtained. In this case, the optimal control $\ControlC$ uses all the quantity allowed at the final times. 
\begin{figure}[H]
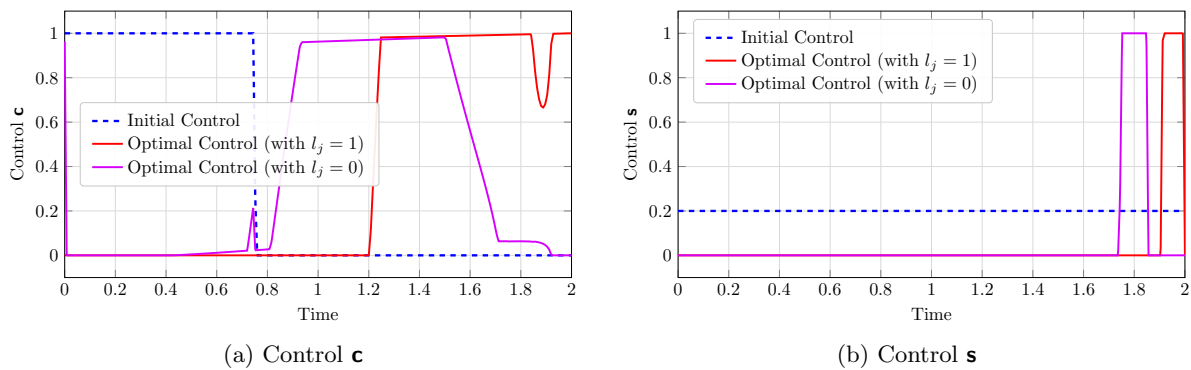

     \centering
     \begin{subfigure}[b]{0.49\textwidth}
         \centering
         \resizebox{\linewidth}{!}{

         }
         \caption{Control \ControlC}
         \label{fig:robustes_controlC_2}
     \end{subfigure}
     \hfill
     \begin{subfigure}[b]{0.49\textwidth}
         \centering
         \resizebox{\linewidth}{!}{
         %
         }
         \caption{Control \ControlS}
         \label{fig:robustes_controlS_2}
     \end{subfigure}
    \caption{Optimal control compared with initial control.}
    \label{fig:robustes_controles_2}
\end{figure}
Figure \ref{fig:robustes_u_2} shows that  the increase in the amount of control  $\ControlC$, allows to obtain a greater regulation of both the volume (see Figure \ref{fig:vol_u_exp2}) and the $L^2$-norm of $u$ (see Figure \ref{fig:int_u_exp2}); this fact demonstrates a greater efficiency of the optimal control strategy. Furthermore, it facilitates enhanced control of the aggregation point (see Figure 10b).
\begin{figure}[H]
     \centering
     \begin{subfigure}[b]{0.49\textwidth}
         \centering
         \resizebox{\linewidth}{!}{
         \begin{tikzpicture}
\begin{axis}[width=12cm, height=6.8cm, xmin=0, xmax=2, xlabel={Time}, ylabel={$\int u^2$}, grid=both, ymajorgrids, grid style={gray!30}, legend cell align=left, legend style={font=\normalsize, at={(0.03,0.97)}, anchor=north west, fill=white, fill opacity=0.92, text opacity=1, draw=gray!45, rounded corners=2pt, inner xsep=6pt, inner ysep=4pt}, /pgf/number format/use comma=false, /pgf/number format/set decimal separator={.}, /pgf/number format/1000 sep={}, scaled x ticks=false, scaled y ticks=false, tick label style={font=\normalsize}, xticklabel style={font=\normalsize, /pgf/number format/fixed, /pgf/number format/precision=3, /pgf/number format/use comma=false, /pgf/number format/set decimal separator={.}, /pgf/number format/1000 sep={}}, yticklabel style={font=\normalsize, /pgf/number format/fixed, /pgf/number format/precision=6, /pgf/number format/use comma=false, /pgf/number format/set decimal separator={.}, /pgf/number format/1000 sep={}}, label style={font=\normalsize}, title style={font=\normalsize}, ]
\addplot+[color={rgb,255:red,0;green,0;blue,255}, dashed, line width=1.3pt, mark=none] coordinates {(0,0.0116615) (0.008,0.00424311) (0.016,0.00317803) (0.024,0.00292957) (0.032,0.00289467) (0.04,0.00294018) (0.048,0.00302428) (0.056,0.00313147) (0.064,0.00325527) (0.072,0.00339272) (0.08,0.00354233) (0.088,0.0037034) (0.096,0.00387562) (0.104,0.00405899) (0.112,0.00425369) (0.12,0.00446005) (0.128,0.00467847) (0.136,0.00490949) (0.144,0.00515368) (0.152,0.00541169) (0.16,0.00568423) (0.168,0.00597204) (0.176,0.00627594) (0.184,0.00659678) (0.192,0.00693547) (0.2,0.00729297) (0.208,0.00767031) (0.216,0.00806855) (0.224,0.00848883) (0.232,0.00893234) (0.24,0.00940035) (0.248,0.00989417) (0.256,0.0104152) (0.264,0.010965) (0.272,0.0115449) (0.28,0.0121568) (0.288,0.0128023) (0.296,0.0134832) (0.304,0.0142014) (0.312,0.014959) (0.32,0.015758) (0.328,0.0166007) (0.336,0.0174894) (0.344,0.0184265) (0.352,0.0194148) (0.36,0.0204568) (0.368,0.0215555) (0.376,0.0227138) (0.384,0.023935) (0.392,0.0252222) (0.4,0.0265791) (0.408,0.0280093) (0.416,0.0295165) (0.424,0.031105) (0.432,0.0327788) (0.44,0.0345425) (0.448,0.0364006) (0.456,0.0383581) (0.464,0.0404202) (0.472,0.0425921) (0.48,0.0448796) (0.488,0.0472885) (0.496,0.049825) (0.504,0.0524955) (0.512,0.0553069) (0.52,0.0582663) (0.528,0.061381) (0.536,0.0646588) (0.544,0.0681079) (0.552,0.0717367) (0.56,0.075554) (0.568,0.0795692) (0.576,0.0837919) (0.584,0.0882322) (0.592,0.0929006) (0.6,0.0978081) (0.608,0.102966) (0.616,0.108387) (0.624,0.114082) (0.632,0.120065) (0.64,0.12635) (0.648,0.13295) (0.656,0.139879) (0.664,0.147154) (0.672,0.154789) (0.68,0.162802) (0.688,0.171208) (0.696,0.180026) (0.704,0.189274) (0.712,0.19897) (0.72,0.209135) (0.728,0.219788) (0.736,0.23095) (0.744,0.242643) (0.752,0.25489) (0.76,0.267712) (0.768,0.281135) (0.776,0.295182) (0.784,0.309879) (0.792,0.325252) (0.8,0.341326) (0.808,0.358131) (0.816,0.375693) (0.824,0.394043) (0.832,0.413208) (0.84,0.43322) (0.848,0.45411) (0.856,0.475909) (0.864,0.49865) (0.872,0.522366) (0.88,0.54709) (0.888,0.572856) (0.896,0.5997) (0.904,0.627657) (0.912,0.656762) (0.92,0.687052) (0.928,0.718564) (0.936,0.751335) (0.944,0.785403) (0.952,0.820804) (0.96,0.857579) (0.968,0.895764) (0.976,0.935398) (0.984,0.97652) (0.992,1.01917) (1,1.06338) (1.008,1.1092) (1.016,1.15666) (1.024,1.20579) (1.032,1.25665) (1.04,1.30926) (1.048,1.36366) (1.056,1.41989) (1.064,1.47799) (1.072,1.53798) (1.08,1.59991) (1.088,1.6638) (1.096,1.7297) (1.104,1.79763) (1.112,1.86762) (1.12,1.93971) (1.128,2.01392) (1.136,2.09028) (1.144,2.16882) (1.152,2.24957) (1.16,2.33255) (1.168,2.41778) (1.176,2.5053) (1.184,2.59512) (1.192,2.68725) (1.2,2.78174) (1.208,2.87859) (1.216,2.97782) (1.224,3.07946) (1.232,3.18352) (1.24,3.29003) (1.248,3.399) (1.256,3.51046) (1.264,3.62442) (1.272,3.74091) (1.28,3.85996) (1.288,3.9816) (1.296,4.10585) (1.304,4.23275) (1.312,4.36233) (1.32,4.49465) (1.328,4.62975) (1.336,4.7677) (1.344,4.90855) (1.352,5.05238) (1.36,5.1993) (1.368,5.34939) (1.376,5.50278) (1.384,5.65961) (1.392,5.82005) (1.4,5.98428) (1.408,6.15252) (1.416,6.32504) (1.424,6.50214) (1.432,6.68419) (1.44,6.87163) (1.448,7.06496) (1.456,7.26481) (1.464,7.47192) (1.472,7.68722) (1.48,7.91183) (1.488,8.14715) (1.496,8.39492) (1.504,8.65736) (1.512,8.93729) (1.52,9.23838) (1.528,9.56537) (1.536,9.92455) (1.544,10.3242) (1.552,10.7754) (1.56,11.2928) (1.568,11.8959) (1.576,12.6095) (1.584,13.4647) (1.592,14.4956) (1.6,15.7317) (1.608,17.1765) (1.616,18.7634) (1.624,20.2864) (1.632,21.3396) (1.64,21.3924) (1.648,20.1318) (1.656,17.845) (1.664,15.2774) (1.672,13.0505) (1.68,11.3731) (1.688,10.1812) (1.696,9.33827) (1.704,8.72622) (1.712,8.26311) (1.72,7.89639) (1.728,7.59333) (1.736,7.33361) (1.744,7.10456) (1.752,6.89815) (1.76,6.70921) (1.768,6.53429) (1.776,6.37106) (1.784,6.21784) (1.792,6.07339) (1.8,5.93678) (1.808,5.80723) (1.816,5.68411) (1.824,5.5669) (1.832,5.45513) (1.84,5.3484) (1.848,5.24635) (1.856,5.14865) (1.864,5.05502) (1.872,4.9652) (1.88,4.87895) (1.888,4.79604) (1.896,4.7163) (1.904,4.63953) (1.912,4.56557) (1.92,4.49427) (1.928,4.4255) (1.936,4.35911) (1.944,4.295) (1.952,4.23305) (1.96,4.17316) (1.968,4.11524) (1.976,4.05919) (1.984,4.00494) (1.992,3.95241) (2,3.90151)};
\addlegendentry{Uncontrolled}
\addplot+[color={rgb,255:red,255;green,0;blue,0}, solid, line width=1.08pt, mark=none] coordinates {(0,0.0116615) (0.008,0.00424311) (0.016,0.00317803) (0.024,0.00292957) (0.032,0.00289467) (0.04,0.00294018) (0.048,0.00302428) (0.056,0.00313147) (0.064,0.00325527) (0.072,0.00339272) (0.08,0.00354233) (0.088,0.0037034) (0.096,0.00387562) (0.104,0.00405899) (0.112,0.00425369) (0.12,0.00446005) (0.128,0.00467847) (0.136,0.00490949) (0.144,0.00515368) (0.152,0.00541169) (0.16,0.00568423) (0.168,0.00597204) (0.176,0.00627594) (0.184,0.00659678) (0.192,0.00693547) (0.2,0.00729297) (0.208,0.00767031) (0.216,0.00806855) (0.224,0.00848883) (0.232,0.00893234) (0.24,0.00940035) (0.248,0.00989417) (0.256,0.0104152) (0.264,0.010965) (0.272,0.0115449) (0.28,0.0121568) (0.288,0.0128023) (0.296,0.0134832) (0.304,0.0142014) (0.312,0.014959) (0.32,0.015758) (0.328,0.0166007) (0.336,0.0174894) (0.344,0.0184265) (0.352,0.0194148) (0.36,0.0204568) (0.368,0.0215555) (0.376,0.0227138) (0.384,0.023935) (0.392,0.0252222) (0.4,0.0265791) (0.408,0.0280093) (0.416,0.0295165) (0.424,0.031105) (0.432,0.0327788) (0.44,0.0345425) (0.448,0.0364006) (0.456,0.0383581) (0.464,0.0404202) (0.472,0.0425921) (0.48,0.0448796) (0.488,0.0472885) (0.496,0.049825) (0.504,0.0524955) (0.512,0.0553069) (0.52,0.0582663) (0.528,0.061381) (0.536,0.0646588) (0.544,0.0681079) (0.552,0.0717367) (0.56,0.075554) (0.568,0.0795692) (0.576,0.0837919) (0.584,0.0882322) (0.592,0.0929006) (0.6,0.0978081) (0.608,0.102966) (0.616,0.108387) (0.624,0.114082) (0.632,0.120065) (0.64,0.12635) (0.648,0.13295) (0.656,0.139879) (0.664,0.147154) (0.672,0.154789) (0.68,0.162802) (0.688,0.171208) (0.696,0.180026) (0.704,0.189274) (0.712,0.19897) (0.72,0.209135) (0.728,0.219788) (0.736,0.23095) (0.744,0.242643) (0.752,0.25489) (0.76,0.267712) (0.768,0.281135) (0.776,0.295182) (0.784,0.309879) (0.792,0.325252) (0.8,0.341326) (0.808,0.358131) (0.816,0.375693) (0.824,0.394043) (0.832,0.413208) (0.84,0.43322) (0.848,0.45411) (0.856,0.475909) (0.864,0.49865) (0.872,0.522366) (0.88,0.54709) (0.888,0.572856) (0.896,0.5997) (0.904,0.627657) (0.912,0.656762) (0.92,0.687052) (0.928,0.718564) (0.936,0.751335) (0.944,0.785403) (0.952,0.820804) (0.96,0.857579) (0.968,0.895764) (0.976,0.935398) (0.984,0.97652) (0.992,1.01917) (1,1.06338) (1.008,1.1092) (1.016,1.15666) (1.024,1.20579) (1.032,1.25665) (1.04,1.30926) (1.048,1.36366) (1.056,1.41989) (1.064,1.47799) (1.072,1.53798) (1.08,1.59991) (1.088,1.6638) (1.096,1.7297) (1.104,1.79763) (1.112,1.86762) (1.12,1.93971) (1.128,2.01392) (1.136,2.09028) (1.144,2.16882) (1.152,2.24957) (1.16,2.33255) (1.168,2.41778) (1.176,2.5053) (1.184,2.59512) (1.192,2.68725) (1.2,2.78174) (1.208,2.87247) (1.216,2.95695) (1.224,3.03487) (1.232,3.10595) (1.24,3.17001) (1.248,3.22689) (1.256,3.28377) (1.264,3.34116) (1.272,3.39901) (1.28,3.4573) (1.288,3.51604) (1.296,3.57521) (1.304,3.6348) (1.312,3.69481) (1.32,3.75525) (1.328,3.81613) (1.336,3.87744) (1.344,3.93921) (1.352,4.00145) (1.36,4.06418) (1.368,4.12743) (1.376,4.19122) (1.384,4.25561) (1.392,4.32062) (1.4,4.38632) (1.408,4.45276) (1.416,4.52) (1.424,4.58814) (1.432,4.65726) (1.44,4.72747) (1.448,4.79889) (1.456,4.87168) (1.464,4.946) (1.472,5.02205) (1.48,5.10006) (1.488,5.18031) (1.496,5.26312) (1.504,5.34888) (1.512,5.43806) (1.52,5.53122) (1.528,5.62903) (1.536,5.73231) (1.544,5.84207) (1.552,5.95955) (1.56,6.08628) (1.568,6.22416) (1.576,6.37555) (1.584,6.54342) (1.592,6.73146) (1.6,6.94425) (1.608,7.18748) (1.616,7.4681) (1.624,7.79435) (1.632,8.17566) (1.64,8.62199) (1.648,9.14224) (1.656,9.74103) (1.664,10.4129) (1.672,11.1328) (1.68,11.8441) (1.688,12.4468) (1.696,12.8014) (1.704,12.7634) (1.712,12.2549) (1.72,11.33) (1.728,10.1661) (1.736,8.97362) (1.744,7.90323) (1.752,7.0175) (1.76,6.31539) (1.768,5.76733) (1.776,5.3382) (1.784,4.99734) (1.792,4.72094) (1.8,4.49157) (1.808,4.29675) (1.816,4.12765) (1.824,3.97802) (1.832,3.84341) (1.84,3.72064) (1.848,3.61765) (1.856,3.53225) (1.864,3.45841) (1.872,3.39142) (1.88,3.32778) (1.888,3.26481) (1.896,3.2003) (1.904,3.13206) (1.912,3.05546) (1.92,2.40989) (1.928,2.11611) (1.936,1.95296) (1.944,1.84645) (1.952,1.76817) (1.96,1.70531) (1.968,1.65145) (1.976,1.60321) (1.984,1.55874) (1.992,1.51704) (2,1.4776)};
\addlegendentry{Optimal Control (with $l_j = 1$)}
\addplot+[color={rgb,255:red,214;green,0;blue,252}, solid, line width=1.08pt, mark=none] coordinates {(0,0.0116615) (0.008,0.00424311) (0.016,0.00317803) (0.024,0.00292957) (0.032,0.00289467) (0.04,0.00294018) (0.048,0.00302428) (0.056,0.00313147) (0.064,0.00325527) (0.072,0.00339272) (0.08,0.00354233) (0.088,0.0037034) (0.096,0.00387562) (0.104,0.00405899) (0.112,0.00425369) (0.12,0.00446005) (0.128,0.00467847) (0.136,0.00490949) (0.144,0.00515368) (0.152,0.00541169) (0.16,0.00568423) (0.168,0.00597204) (0.176,0.00627594) (0.184,0.00659678) (0.192,0.00693547) (0.2,0.00729297) (0.208,0.00767031) (0.216,0.00806855) (0.224,0.00848883) (0.232,0.00893234) (0.24,0.00940035) (0.248,0.00989417) (0.256,0.0104152) (0.264,0.010965) (0.272,0.0115449) (0.28,0.0121568) (0.288,0.0128023) (0.296,0.0134832) (0.304,0.0142014) (0.312,0.014959) (0.32,0.015758) (0.328,0.0166007) (0.336,0.0174894) (0.344,0.0184265) (0.352,0.0194148) (0.36,0.0204568) (0.368,0.0215555) (0.376,0.0227138) (0.384,0.023935) (0.392,0.0252222) (0.4,0.0265791) (0.408,0.0280093) (0.416,0.0295165) (0.424,0.031105) (0.432,0.0327787) (0.44,0.034542) (0.448,0.0363996) (0.456,0.0383563) (0.464,0.0404171) (0.472,0.0425875) (0.48,0.0448729) (0.488,0.0472792) (0.496,0.0498125) (0.504,0.0524792) (0.512,0.0552861) (0.52,0.0582401) (0.528,0.0613484) (0.536,0.0646189) (0.544,0.0680595) (0.552,0.0716785) (0.56,0.0754846) (0.568,0.079487) (0.576,0.0836951) (0.584,0.0881189) (0.592,0.0927686) (0.6,0.097655) (0.608,0.102789) (0.616,0.108183) (0.624,0.113849) (0.632,0.119799) (0.64,0.126047) (0.648,0.132606) (0.656,0.13949) (0.664,0.146715) (0.672,0.154294) (0.68,0.162245) (0.688,0.170584) (0.696,0.179327) (0.704,0.188493) (0.712,0.198099) (0.72,0.208165) (0.728,0.218532) (0.736,0.229184) (0.744,0.240113) (0.752,0.252169) (0.76,0.264787) (0.768,0.277989) (0.776,0.291798) (0.784,0.306239) (0.792,0.321337) (0.8,0.337117) (0.808,0.353605) (0.816,0.370674) (0.824,0.388154) (0.832,0.406025) (0.84,0.424265) (0.848,0.442853) (0.856,0.461764) (0.864,0.480974) (0.872,0.500457) (0.88,0.520186) (0.888,0.540134) (0.896,0.560273) (0.904,0.580576) (0.912,0.601014) (0.92,0.621558) (0.928,0.642177) (0.936,0.663187) (0.944,0.684754) (0.952,0.706886) (0.96,0.729589) (0.968,0.752871) (0.976,0.776737) (0.984,0.801196) (0.992,0.826252) (1,0.851912) (1.008,0.878181) (1.016,0.905065) (1.024,0.932569) (1.032,0.960697) (1.04,0.989455) (1.048,1.01885) (1.056,1.04887) (1.064,1.07954) (1.072,1.11085) (1.08,1.1428) (1.088,1.17541) (1.096,1.20866) (1.104,1.24256) (1.112,1.27711) (1.12,1.31231) (1.128,1.34816) (1.136,1.38467) (1.144,1.42182) (1.152,1.45962) (1.16,1.49806) (1.168,1.53715) (1.176,1.57688) (1.184,1.61725) (1.192,1.65825) (1.2,1.69988) (1.208,1.74214) (1.216,1.78501) (1.224,1.82851) (1.232,1.87261) (1.24,1.91732) (1.248,1.96263) (1.256,2.00853) (1.264,2.05502) (1.272,2.10209) (1.28,2.14974) (1.288,2.19796) (1.296,2.24674) (1.304,2.29607) (1.312,2.34596) (1.32,2.39639) (1.328,2.44737) (1.336,2.49888) (1.344,2.55092) (1.352,2.60349) (1.36,2.65658) (1.368,2.7102) (1.376,2.76434) (1.384,2.819) (1.392,2.87418) (1.4,2.92989) (1.408,2.98613) (1.416,3.04291) (1.424,3.10023) (1.432,3.15811) (1.44,3.21656) (1.448,3.2756) (1.456,3.33526) (1.464,3.39554) (1.472,3.4565) (1.48,3.51815) (1.488,3.58055) (1.496,3.64373) (1.504,3.70817) (1.512,3.77607) (1.52,3.8476) (1.528,3.92295) (1.536,4.00233) (1.544,4.086) (1.552,4.17424) (1.56,4.26737) (1.568,4.36577) (1.576,4.46987) (1.584,4.58019) (1.592,4.69732) (1.6,4.82198) (1.608,4.955) (1.616,5.09743) (1.624,5.2505) (1.632,5.41573) (1.64,5.59503) (1.648,5.79071) (1.656,6.00573) (1.664,6.24376) (1.672,6.50952) (1.68,6.80901) (1.688,7.14993) (1.696,7.54212) (1.704,7.99815) (1.712,8.53064) (1.72,9.14881) (1.728,9.87027) (1.736,10.7108) (1.744,11.6778) (1.752,11.9829) (1.76,10.1429) (1.768,8.8111) (1.776,7.87074) (1.784,7.18335) (1.792,6.666) (1.8,6.26807) (1.808,5.95695) (1.816,5.7106) (1.824,5.51366) (1.832,5.35511) (1.84,5.2269) (1.848,5.12309) (1.856,5.03891) (1.864,5.13239) (1.872,5.30366) (1.88,5.52398) (1.888,5.78434) (1.896,6.0837) (1.904,6.42543) (1.912,6.81622) (1.92,7.26634) (1.928,7.78205) (1.936,8.37425) (1.944,9.0541) (1.952,9.82979) (1.96,10.7011) (1.968,11.6496) (1.976,12.6237) (1.984,13.5198) (1.992,14.1723) (2,14.379)};
\addlegendentry{Optimal Control (with $l_j = 0$)}
\end{axis}
\end{tikzpicture}
         }
         \caption{$\int_\Omega u^2dx$}
         \label{fig:int_u_exp2}
     \end{subfigure}
     \hfill
     \begin{subfigure}[b]{0.49\textwidth}
         \centering
         \resizebox{\linewidth}{!}{
         \begin{tikzpicture}
\begin{axis}[width=12cm, height=6.8cm, xmin=0, xmax=2, xlabel={Time}, ylabel={Max u}, grid=both, ymajorgrids, grid style={gray!30}, legend cell align=left, legend style={font=\normalsize, at={(0.03,0.97)}, anchor=north west, fill=white, fill opacity=0.92, text opacity=1, draw=gray!45, rounded corners=2pt, inner xsep=6pt, inner ysep=4pt}, tick label style={font=\normalsize}, label style={font=\normalsize}, title style={font=\normalsize}, ]
\addplot+[color={rgb,255:red,119;green,119;blue,119}, dashed, opacity=0.65, line width=1.2pt, mark=none, forget plot] coordinates {(0,4) (2,4)};
\node[anchor=west, text={rgb,255:red,85;green,85;blue,85}, text opacity=0.85] at (axis cs:0.06,15.02) {maximum capacity};
\addplot+[color={rgb,255:red,0;green,0;blue,255}, dashed, line width=1.3pt, mark=none] coordinates {(0.008,0.256873) (0.016,0.206981) (0.024,0.187258) (0.032,0.178822) (0.04,0.176312) (0.048,0.177225) (0.056,0.180147) (0.064,0.184085) (0.072,0.188386) (0.08,0.192753) (0.088,0.197023) (0.096,0.20124) (0.104,0.205328) (0.112,0.209294) (0.12,0.213152) (0.128,0.216918) (0.136,0.22061) (0.144,0.224245) (0.152,0.227841) (0.16,0.231415) (0.168,0.234982) (0.176,0.238558) (0.184,0.242156) (0.192,0.245789) (0.2,0.249467) (0.208,0.253201) (0.216,0.257001) (0.224,0.260876) (0.232,0.264833) (0.24,0.268881) (0.248,0.273025) (0.256,0.277273) (0.264,0.28163) (0.272,0.286102) (0.28,0.290695) (0.288,0.295414) (0.296,0.300263) (0.304,0.305248) (0.312,0.310373) (0.32,0.315643) (0.328,0.321063) (0.336,0.326636) (0.344,0.332368) (0.352,0.338262) (0.36,0.344323) (0.368,0.350556) (0.376,0.356964) (0.384,0.363553) (0.392,0.370332) (0.4,0.377301) (0.408,0.384463) (0.416,0.391823) (0.424,0.399385) (0.432,0.407164) (0.44,0.41516) (0.448,0.423373) (0.456,0.431807) (0.464,0.440469) (0.472,0.449363) (0.48,0.458495) (0.488,0.46787) (0.496,0.477494) (0.504,0.487373) (0.512,0.497512) (0.52,0.507918) (0.528,0.518596) (0.536,0.529552) (0.544,0.540794) (0.552,0.552553) (0.56,0.56479) (0.568,0.577323) (0.576,0.590159) (0.584,0.603307) (0.592,0.616772) (0.6,0.630563) (0.608,0.644691) (0.616,0.659162) (0.624,0.673981) (0.632,0.689156) (0.64,0.704697) (0.648,0.720611) (0.656,0.736907) (0.664,0.753593) (0.672,0.770678) (0.68,0.788171) (0.688,0.806082) (0.696,0.824419) (0.704,0.843192) (0.712,0.862411) (0.72,0.882085) (0.728,0.902225) (0.736,0.92284) (0.744,0.943941) (0.752,0.965538) (0.76,0.987642) (0.768,1.01026) (0.776,1.03341) (0.784,1.05711) (0.792,1.08135) (0.8,1.10615) (0.808,1.13153) (0.816,1.1575) (0.824,1.18473) (0.832,1.21276) (0.84,1.24148) (0.848,1.2709) (0.856,1.30102) (0.864,1.33189) (0.872,1.3635) (0.88,1.39587) (0.888,1.42902) (0.896,1.46298) (0.904,1.49776) (0.912,1.53337) (0.92,1.56984) (0.928,1.60718) (0.936,1.64544) (0.944,1.68461) (0.952,1.72472) (0.96,1.76578) (0.968,1.80784) (0.976,1.8509) (0.984,1.895) (0.992,1.94019) (1,1.98646) (1.008,2.03386) (1.016,2.0824) (1.024,2.13212) (1.032,2.18307) (1.04,2.23528) (1.048,2.28878) (1.056,2.3436) (1.064,2.39978) (1.072,2.45736) (1.08,2.51641) (1.088,2.57697) (1.096,2.63906) (1.104,2.70275) (1.112,2.76808) (1.12,2.83511) (1.128,2.90391) (1.136,2.97457) (1.144,3.04712) (1.152,3.12164) (1.16,3.19822) (1.168,3.27694) (1.176,3.36014) (1.184,3.4507) (1.192,3.54427) (1.2,3.64098) (1.208,3.74103) (1.216,3.84459) (1.224,3.95187) (1.232,4.06309) (1.24,4.1785) (1.248,4.29835) (1.256,4.42294) (1.264,4.55258) (1.272,4.68762) (1.28,4.82843) (1.288,4.97545) (1.296,5.12911) (1.304,5.28995) (1.312,5.45852) (1.32,5.63545) (1.328,5.82144) (1.336,6.01727) (1.344,6.22381) (1.352,6.44204) (1.36,6.67306) (1.368,6.91813) (1.376,7.17865) (1.384,7.45626) (1.392,7.75277) (1.4,8.07028) (1.408,8.41121) (1.416,8.77836) (1.424,9.17497) (1.432,9.60484) (1.44,10.0724) (1.448,10.583) (1.456,11.1429) (1.464,11.7595) (1.472,12.4418) (1.48,13.2009) (1.488,14.0502) (1.496,15.0061) (1.504,16.0893) (1.512,17.3258) (1.52,18.7485) (1.528,20.3998) (1.536,22.3346) (1.544,24.6258) (1.552,27.3693) (1.56,30.6944) (1.568,34.777) (1.576,39.8574) (1.584,46.2656) (1.592,54.45) (1.6,65.0044) (1.608,78.6611) (1.616,96.1789) (1.624,117.973) (1.632,143.313) (1.64,169.278) (1.648,190.658) (1.656,202.373) (1.664,203.049) (1.672,195.597) (1.68,184.357) (1.688,172.722) (1.696,162.154) (1.704,153.087) (1.712,145.478) (1.72,139.101) (1.728,133.708) (1.736,129.08) (1.744,125.047) (1.752,121.477) (1.76,118.274) (1.768,115.365) (1.776,112.696) (1.784,110.228) (1.792,107.928) (1.8,105.774) (1.808,103.747) (1.816,101.831) (1.824,100.014) (1.832,98.2875) (1.84,96.6416) (1.848,95.0698) (1.856,93.5657) (1.864,92.124) (1.872,90.7402) (1.88,89.41) (1.888,88.1299) (1.896,86.8966) (1.904,85.7071) (1.912,84.5589) (1.92,83.4495) (1.928,82.3768) (1.936,81.3388) (1.944,80.3336) (1.952,79.3597) (1.96,78.4156) (1.968,77.4997) (1.976,76.6108) (1.984,75.7478) (1.992,74.9094) (2,74.0946)};
\addlegendentry{Uncontrolled}
\addplot+[color={rgb,255:red,255;green,0;blue,0}, solid, line width=0.99pt, mark=none] coordinates {(0.008,0.256873) (0.016,0.206981) (0.024,0.187258) (0.032,0.178822) (0.04,0.176312) (0.048,0.177225) (0.056,0.180147) (0.064,0.184085) (0.072,0.188386) (0.08,0.192753) (0.088,0.197023) (0.096,0.20124) (0.104,0.205328) (0.112,0.209294) (0.12,0.213152) (0.128,0.216918) (0.136,0.22061) (0.144,0.224245) (0.152,0.227841) (0.16,0.231415) (0.168,0.234982) (0.176,0.238558) (0.184,0.242156) (0.192,0.245789) (0.2,0.249467) (0.208,0.253201) (0.216,0.257001) (0.224,0.260876) (0.232,0.264833) (0.24,0.268881) (0.248,0.273025) (0.256,0.277273) (0.264,0.28163) (0.272,0.286102) (0.28,0.290695) (0.288,0.295414) (0.296,0.300263) (0.304,0.305248) (0.312,0.310373) (0.32,0.315643) (0.328,0.321063) (0.336,0.326636) (0.344,0.332368) (0.352,0.338262) (0.36,0.344323) (0.368,0.350556) (0.376,0.356964) (0.384,0.363553) (0.392,0.370332) (0.4,0.377301) (0.408,0.384463) (0.416,0.391823) (0.424,0.399385) (0.432,0.407164) (0.44,0.41516) (0.448,0.423373) (0.456,0.431807) (0.464,0.440469) (0.472,0.449363) (0.48,0.458495) (0.488,0.46787) (0.496,0.477494) (0.504,0.487373) (0.512,0.497512) (0.52,0.507918) (0.528,0.518596) (0.536,0.529552) (0.544,0.540794) (0.552,0.552553) (0.56,0.56479) (0.568,0.577323) (0.576,0.590159) (0.584,0.603307) (0.592,0.616772) (0.6,0.630563) (0.608,0.644691) (0.616,0.659162) (0.624,0.673981) (0.632,0.689156) (0.64,0.704697) (0.648,0.720611) (0.656,0.736907) (0.664,0.753593) (0.672,0.770678) (0.68,0.788171) (0.688,0.806082) (0.696,0.824419) (0.704,0.843192) (0.712,0.862411) (0.72,0.882085) (0.728,0.902225) (0.736,0.92284) (0.744,0.943941) (0.752,0.965538) (0.76,0.987642) (0.768,1.01026) (0.776,1.03341) (0.784,1.05711) (0.792,1.08135) (0.8,1.10615) (0.808,1.13153) (0.816,1.1575) (0.824,1.18473) (0.832,1.21276) (0.84,1.24148) (0.848,1.2709) (0.856,1.30102) (0.864,1.33189) (0.872,1.3635) (0.88,1.39587) (0.888,1.42902) (0.896,1.46298) (0.904,1.49776) (0.912,1.53337) (0.92,1.56984) (0.928,1.60718) (0.936,1.64544) (0.944,1.68461) (0.952,1.72472) (0.96,1.76578) (0.968,1.80784) (0.976,1.8509) (0.984,1.895) (0.992,1.94019) (1,1.98646) (1.008,2.03386) (1.016,2.0824) (1.024,2.13212) (1.032,2.18307) (1.04,2.23528) (1.048,2.28878) (1.056,2.3436) (1.064,2.39978) (1.072,2.45736) (1.08,2.51641) (1.088,2.57697) (1.096,2.63906) (1.104,2.70275) (1.112,2.76808) (1.12,2.83511) (1.128,2.90391) (1.136,2.97457) (1.144,3.04712) (1.152,3.12164) (1.16,3.19822) (1.168,3.27694) (1.176,3.36014) (1.184,3.4507) (1.192,3.54427) (1.2,3.64098) (1.208,3.73361) (1.216,3.81673) (1.224,3.89051) (1.232,3.95508) (1.24,4.01058) (1.248,4.05838) (1.256,4.10907) (1.264,4.16239) (1.272,4.21821) (1.28,4.27648) (1.288,4.33719) (1.296,4.40037) (1.304,4.46609) (1.312,4.5344) (1.32,4.60541) (1.328,4.6792) (1.336,4.7559) (1.344,4.83565) (1.352,4.91858) (1.36,5.00486) (1.368,5.09467) (1.376,5.18824) (1.384,5.28575) (1.392,5.38744) (1.4,5.49357) (1.408,5.60442) (1.416,5.7203) (1.424,5.84154) (1.432,5.96851) (1.44,6.10162) (1.448,6.24132) (1.456,6.38807) (1.464,6.54243) (1.472,6.70498) (1.48,6.87637) (1.488,7.05733) (1.496,7.24863) (1.504,7.45118) (1.512,7.66599) (1.52,7.89417) (1.528,8.13689) (1.536,8.39555) (1.544,8.67166) (1.552,8.96697) (1.56,9.28343) (1.568,9.62325) (1.576,9.98892) (1.584,10.3833) (1.592,10.8096) (1.6,11.2716) (1.608,11.7734) (1.616,12.32) (1.624,12.917) (1.632,13.5708) (1.64,14.289) (1.648,15.0802) (1.656,15.9545) (1.664,16.9236) (1.672,18.0011) (1.68,19.2028) (1.688,20.547) (1.696,22.0548) (1.704,23.7507) (1.712,25.6614) (1.72,27.817) (1.728,30.2499) (1.736,32.993) (1.744,36.0776) (1.752,39.5282) (1.76,43.3555) (1.768,47.5462) (1.776,52.0495) (1.784,56.7622) (1.792,61.5168) (1.8,66.0782) (1.808,70.1596) (1.816,73.4605) (1.824,75.725) (1.832,76.8008) (1.84,76.6764) (1.848,75.6551) (1.856,74.0292) (1.864,71.9778) (1.872,69.6628) (1.88,67.2121) (1.888,64.7125) (1.896,62.2112) (1.904,59.7199) (1.912,57.1828) (1.92,40.6122) (1.928,32.1174) (1.936,27.1093) (1.944,23.7321) (1.952,21.2521) (1.96,19.3241) (1.968,17.7638) (1.976,16.4631) (1.984,15.3543) (1.992,14.3924) (2,13.5464)};
\addlegendentry{Optimal Control (with $l_j = 1$)}
\addplot+[color={rgb,255:red,214;green,0;blue,252}, solid, line width=1.08pt, mark=none] coordinates {(0.008,0.256873) (0.016,0.206981) (0.024,0.187258) (0.032,0.178822) (0.04,0.176312) (0.048,0.177225) (0.056,0.180147) (0.064,0.184085) (0.072,0.188386) (0.08,0.192753) (0.088,0.197023) (0.096,0.20124) (0.104,0.205328) (0.112,0.209294) (0.12,0.213152) (0.128,0.216918) (0.136,0.22061) (0.144,0.224245) (0.152,0.227841) (0.16,0.231415) (0.168,0.234982) (0.176,0.238558) (0.184,0.242156) (0.192,0.245789) (0.2,0.249467) (0.208,0.253201) (0.216,0.257001) (0.224,0.260876) (0.232,0.264833) (0.24,0.268881) (0.248,0.273025) (0.256,0.277273) (0.264,0.28163) (0.272,0.286102) (0.28,0.290695) (0.288,0.295414) (0.296,0.300263) (0.304,0.305248) (0.312,0.310373) (0.32,0.315643) (0.328,0.321063) (0.336,0.326636) (0.344,0.332368) (0.352,0.338262) (0.36,0.344323) (0.368,0.350556) (0.376,0.356964) (0.384,0.363553) (0.392,0.370332) (0.4,0.377301) (0.408,0.384463) (0.416,0.391823) (0.424,0.399385) (0.432,0.407163) (0.44,0.415154) (0.448,0.42336) (0.456,0.431785) (0.464,0.440433) (0.472,0.44931) (0.48,0.458422) (0.488,0.467772) (0.496,0.477367) (0.504,0.487213) (0.512,0.497313) (0.52,0.507676) (0.528,0.518305) (0.536,0.529207) (0.544,0.540388) (0.552,0.552082) (0.56,0.564244) (0.568,0.576696) (0.576,0.589443) (0.584,0.602493) (0.592,0.615852) (0.6,0.629528) (0.608,0.643531) (0.616,0.657866) (0.624,0.672539) (0.632,0.687557) (0.64,0.702927) (0.648,0.718659) (0.656,0.734758) (0.664,0.751234) (0.672,0.768093) (0.68,0.785345) (0.688,0.802998) (0.696,0.82106) (0.704,0.839539) (0.712,0.858445) (0.72,0.877787) (0.728,0.896776) (0.736,0.91535) (0.744,0.933474) (0.752,0.95466) (0.76,0.976287) (0.768,0.998385) (0.776,1.02097) (0.784,1.04406) (0.792,1.06767) (0.8,1.0918) (0.808,1.11647) (0.816,1.14117) (0.824,1.16593) (0.832,1.19031) (0.84,1.21409) (0.848,1.23723) (0.856,1.25969) (0.864,1.28143) (0.872,1.30242) (0.88,1.3226) (0.888,1.34196) (0.896,1.36047) (0.904,1.37812) (0.912,1.39486) (0.92,1.41069) (0.928,1.42558) (0.936,1.44036) (0.944,1.45538) (0.952,1.47063) (0.96,1.48609) (0.968,1.50175) (0.976,1.51762) (0.984,1.53368) (0.992,1.54995) (1,1.56642) (1.008,1.58307) (1.016,1.5999) (1.024,1.61692) (1.032,1.63412) (1.04,1.65153) (1.048,1.66911) (1.056,1.68687) (1.064,1.70481) (1.072,1.72293) (1.08,1.74126) (1.088,1.75976) (1.096,1.77844) (1.104,1.79731) (1.112,1.81635) (1.12,1.83558) (1.128,1.85502) (1.136,1.87464) (1.144,1.89445) (1.152,1.91447) (1.16,1.93493) (1.168,1.9592) (1.176,1.98378) (1.184,2.00869) (1.192,2.03392) (1.2,2.05949) (1.208,2.08541) (1.216,2.11168) (1.224,2.13831) (1.232,2.16532) (1.24,2.19271) (1.248,2.22049) (1.256,2.24867) (1.264,2.27727) (1.272,2.30629) (1.28,2.33576) (1.288,2.36568) (1.296,2.39606) (1.304,2.42693) (1.312,2.4583) (1.32,2.49017) (1.328,2.52258) (1.336,2.55554) (1.344,2.58907) (1.352,2.62318) (1.36,2.6579) (1.368,2.69325) (1.376,2.72925) (1.384,2.76592) (1.392,2.80329) (1.4,2.84138) (1.408,2.88022) (1.416,2.91985) (1.424,2.96028) (1.432,3.00156) (1.44,3.04371) (1.448,3.08678) (1.456,3.13079) (1.464,3.17579) (1.472,3.22184) (1.48,3.26895) (1.488,3.31719) (1.496,3.36659) (1.504,3.4176) (1.512,3.47234) (1.52,3.53088) (1.528,3.59335) (1.536,3.65992) (1.544,3.73076) (1.552,3.80609) (1.56,3.88613) (1.568,3.97116) (1.576,4.06146) (1.584,4.15735) (1.592,4.2592) (1.6,4.3674) (1.608,4.48242) (1.616,4.60474) (1.624,4.73494) (1.632,4.87365) (1.64,5.02161) (1.648,5.17963) (1.656,5.34867) (1.664,5.52981) (1.672,5.72429) (1.68,5.93358) (1.688,6.15941) (1.696,6.40387) (1.704,6.66945) (1.712,6.95693) (1.72,7.26326) (1.728,7.59092) (1.736,7.94255) (1.744,8.32115) (1.752,8.6409) (1.76,8.67554) (1.768,8.60317) (1.776,8.48367) (1.784,8.3431) (1.792,8.19422) (1.8,8.04363) (1.808,7.89486) (1.816,7.7498) (1.824,7.60943) (1.832,7.47423) (1.84,7.34441) (1.848,7.22006) (1.856,7.10081) (1.864,7.17029) (1.872,7.32764) (1.88,7.53799) (1.888,7.78761) (1.896,8.07142) (1.904,8.38858) (1.912,8.74102) (1.92,9.13343) (1.928,9.56497) (1.936,10.0397) (1.944,10.5633) (1.952,11.1425) (1.96,11.7854) (1.968,12.502) (1.976,13.3046) (1.984,14.2083) (1.992,15.2315) (2,16.3975)};
\addlegendentry{Optimal Control (with $l_j = 0$)}

\end{axis}
\end{tikzpicture}
         }
         \caption{Max $u$}
         \label{fig:max_u_exp2}
     \end{subfigure}
     \hfill
     \begin{subfigure}[b]{0.49\textwidth}
         \centering
         \resizebox{\linewidth}{!}{
         \begin{tikzpicture}
\begin{axis}[width=12cm, height=6.8cm, xmin=0, xmax=2, xlabel={Time}, ylabel={$\int (\sigma - \sigma_Q)^2$}, grid=both, ymajorgrids, grid style={gray!30}, legend cell align=left, legend style={font=\normalsize, at={(0.03,0.03)}, anchor=south west, fill=white, fill opacity=0.92, text opacity=1, draw=gray!45, rounded corners=2pt, inner xsep=6pt, inner ysep=4pt}, /pgf/number format/use comma=false, /pgf/number format/set decimal separator={.}, /pgf/number format/1000 sep={}, scaled x ticks=false, scaled y ticks=false, tick label style={font=\normalsize}, xticklabel style={font=\normalsize, /pgf/number format/fixed, /pgf/number format/precision=3, /pgf/number format/use comma=false, /pgf/number format/set decimal separator={.}, /pgf/number format/1000 sep={}}, yticklabel style={font=\normalsize, /pgf/number format/fixed, /pgf/number format/precision=6, /pgf/number format/use comma=false, /pgf/number format/set decimal separator={.}, /pgf/number format/1000 sep={}}, label style={font=\normalsize}, title style={font=\normalsize}, ]
\addplot+[color={rgb,255:red,0;green,0;blue,255}, dashed, line width=1.3pt, mark=none] coordinates {(0,0.0835513) (0.008,0.0830233) (0.016,0.0825902) (0.024,0.0822019) (0.032,0.0818413) (0.04,0.0814996) (0.048,0.0811721) (0.056,0.0808555) (0.064,0.0805477) (0.072,0.0802471) (0.08,0.0799526) (0.088,0.0796631) (0.096,0.079378) (0.104,0.0790966) (0.112,0.0788183) (0.12,0.0785428) (0.128,0.0782697) (0.136,0.0779985) (0.144,0.0777291) (0.152,0.077461) (0.16,0.0771942) (0.168,0.0769283) (0.176,0.0766632) (0.184,0.0763987) (0.192,0.0761345) (0.2,0.0758706) (0.208,0.0756067) (0.216,0.0753427) (0.224,0.0750785) (0.232,0.0748139) (0.24,0.0745488) (0.248,0.074283) (0.256,0.0740165) (0.264,0.0737491) (0.272,0.0734806) (0.28,0.073211) (0.288,0.0729401) (0.296,0.0726678) (0.304,0.072394) (0.312,0.0721186) (0.32,0.0718415) (0.328,0.0715626) (0.336,0.0712817) (0.344,0.0709987) (0.352,0.0707136) (0.36,0.0704262) (0.368,0.0701364) (0.376,0.0698441) (0.384,0.0695491) (0.392,0.0692515) (0.4,0.068951) (0.408,0.0686476) (0.416,0.0683412) (0.424,0.0680316) (0.432,0.0677187) (0.44,0.0674024) (0.448,0.0670827) (0.456,0.0667593) (0.464,0.0664322) (0.472,0.0661014) (0.48,0.0657665) (0.488,0.0654277) (0.496,0.0650847) (0.504,0.0647374) (0.512,0.0643858) (0.52,0.0640296) (0.528,0.0636689) (0.536,0.0633035) (0.544,0.0629332) (0.552,0.062558) (0.56,0.0621778) (0.568,0.0617924) (0.576,0.0614018) (0.584,0.0610058) (0.592,0.0606044) (0.6,0.0601974) (0.608,0.0597847) (0.616,0.0593662) (0.624,0.0589419) (0.632,0.0585115) (0.64,0.0580751) (0.648,0.0576326) (0.656,0.0571837) (0.664,0.0567285) (0.672,0.0562669) (0.68,0.0557988) (0.688,0.0553241) (0.696,0.0548427) (0.704,0.0543545) (0.712,0.0538596) (0.72,0.0533577) (0.728,0.052849) (0.736,0.0523332) (0.744,0.0518104) (0.752,0.0512806) (0.76,0.0507437) (0.768,0.0501996) (0.776,0.0496484) (0.784,0.0490901) (0.792,0.0485246) (0.8,0.0479519) (0.808,0.0473722) (0.816,0.0467853) (0.824,0.0461914) (0.832,0.0455905) (0.84,0.0449826) (0.848,0.0443679) (0.856,0.0437464) (0.864,0.0431182) (0.872,0.0424834) (0.88,0.0418422) (0.888,0.0411947) (0.896,0.0405411) (0.904,0.0398815) (0.912,0.0392161) (0.92,0.0385453) (0.928,0.0378691) (0.936,0.0371878) (0.944,0.0365018) (0.952,0.0358112) (0.96,0.0351165) (0.968,0.034418) (0.976,0.033716) (0.984,0.0330109) (0.992,0.0323031) (1,0.031593) (1.008,0.0308812) (1.016,0.0301681) (1.024,0.0294542) (1.032,0.02874) (1.04,0.0280262) (1.048,0.0273133) (1.056,0.0266019) (1.064,0.0258928) (1.072,0.0251866) (1.08,0.024484) (1.088,0.0237857) (1.096,0.0230927) (1.104,0.0224056) (1.112,0.0217254) (1.12,0.0210529) (1.128,0.020389) (1.136,0.0197347) (1.144,0.019091) (1.152,0.0184589) (1.16,0.0178395) (1.168,0.0172338) (1.176,0.0166429) (1.184,0.0160681) (1.192,0.0155106) (1.2,0.0149715) (1.208,0.0144521) (1.216,0.0139537) (1.224,0.0134777) (1.232,0.0130254) (1.24,0.0125983) (1.248,0.0121977) (1.256,0.0118251) (1.264,0.0114821) (1.272,0.0111701) (1.28,0.0108907) (1.288,0.0106455) (1.296,0.0104362) (1.304,0.0102643) (1.312,0.0101316) (1.32,0.0100398) (1.328,0.0099906) (1.336,0.00998575) (1.344,0.010027) (1.352,0.0101163) (1.36,0.0102553) (1.368,0.010446) (1.376,0.0106901) (1.384,0.0109897) (1.392,0.0113465) (1.4,0.0117625) (1.408,0.0122397) (1.416,0.0127798) (1.424,0.013385) (1.432,0.0140571) (1.44,0.0147981) (1.448,0.0156099) (1.456,0.0164944) (1.464,0.0174536) (1.472,0.0184892) (1.48,0.0196033) (1.488,0.0207976) (1.496,0.0220738) (1.504,0.0234337) (1.512,0.0248788) (1.52,0.0264106) (1.528,0.0280305) (1.536,0.0297394) (1.544,0.0315383) (1.552,0.0334273) (1.56,0.0354061) (1.568,0.0374734) (1.576,0.0396262) (1.584,0.0418592) (1.592,0.0441632) (1.6,0.0465231) (1.608,0.0489144) (1.616,0.0512993) (1.624,0.0536235) (1.632,0.0558166) (1.64,0.0578051) (1.648,0.0595364) (1.656,0.0610025) (1.664,0.0622406) (1.672,0.0633093) (1.68,0.0642641) (1.688,0.0651464) (1.696,0.0659829) (1.704,0.0667898) (1.712,0.0675769) (1.72,0.0683494) (1.728,0.0691106) (1.736,0.0698623) (1.744,0.0706056) (1.752,0.0713411) (1.76,0.0720694) (1.768,0.0727908) (1.776,0.0735055) (1.784,0.0742137) (1.792,0.0749158) (1.8,0.075612) (1.808,0.0763023) (1.816,0.0769871) (1.824,0.0776664) (1.832,0.0783406) (1.84,0.0790098) (1.848,0.0796741) (1.856,0.0803336) (1.864,0.0809887) (1.872,0.0816394) (1.88,0.0822858) (1.888,0.0829281) (1.896,0.0835664) (1.904,0.0842009) (1.912,0.0848316) (1.92,0.0854588) (1.928,0.0860825) (1.936,0.0867028) (1.944,0.0873198) (1.952,0.0879337) (1.96,0.0885445) (1.968,0.0891524) (1.976,0.0897575) (1.984,0.0903598) (1.992,0.0909595) (2,0.0915566)};
\addlegendentry{Uncontrolled}
\addplot+[color={rgb,255:red,255;green,0;blue,0}, solid, line width=1.08pt, mark=none] coordinates {(0,0.0835513) (0.008,0.0830233) (0.016,0.0825902) (0.024,0.0822019) (0.032,0.0818413) (0.04,0.0814996) (0.048,0.0811721) (0.056,0.0808555) (0.064,0.0805477) (0.072,0.0802471) (0.08,0.0799526) (0.088,0.0796631) (0.096,0.079378) (0.104,0.0790966) (0.112,0.0788183) (0.12,0.0785428) (0.128,0.0782697) (0.136,0.0779985) (0.144,0.0777291) (0.152,0.077461) (0.16,0.0771942) (0.168,0.0769283) (0.176,0.0766632) (0.184,0.0763987) (0.192,0.0761345) (0.2,0.0758706) (0.208,0.0756067) (0.216,0.0753427) (0.224,0.0750785) (0.232,0.0748139) (0.24,0.0745488) (0.248,0.074283) (0.256,0.0740165) (0.264,0.0737491) (0.272,0.0734806) (0.28,0.073211) (0.288,0.0729401) (0.296,0.0726678) (0.304,0.072394) (0.312,0.0721186) (0.32,0.0718415) (0.328,0.0715626) (0.336,0.0712817) (0.344,0.0709987) (0.352,0.0707136) (0.36,0.0704262) (0.368,0.0701364) (0.376,0.0698441) (0.384,0.0695491) (0.392,0.0692515) (0.4,0.068951) (0.408,0.0686476) (0.416,0.0683412) (0.424,0.0680316) (0.432,0.0677187) (0.44,0.0674024) (0.448,0.0670827) (0.456,0.0667593) (0.464,0.0664322) (0.472,0.0661014) (0.48,0.0657665) (0.488,0.0654277) (0.496,0.0650847) (0.504,0.0647374) (0.512,0.0643858) (0.52,0.0640296) (0.528,0.0636689) (0.536,0.0633035) (0.544,0.0629332) (0.552,0.062558) (0.56,0.0621778) (0.568,0.0617924) (0.576,0.0614018) (0.584,0.0610058) (0.592,0.0606044) (0.6,0.0601974) (0.608,0.0597847) (0.616,0.0593662) (0.624,0.0589419) (0.632,0.0585115) (0.64,0.0580751) (0.648,0.0576326) (0.656,0.0571837) (0.664,0.0567285) (0.672,0.0562669) (0.68,0.0557988) (0.688,0.0553241) (0.696,0.0548427) (0.704,0.0543545) (0.712,0.0538596) (0.72,0.0533577) (0.728,0.052849) (0.736,0.0523332) (0.744,0.0518104) (0.752,0.0512806) (0.76,0.0507437) (0.768,0.0501996) (0.776,0.0496484) (0.784,0.0490901) (0.792,0.0485246) (0.8,0.0479519) (0.808,0.0473722) (0.816,0.0467853) (0.824,0.0461914) (0.832,0.0455905) (0.84,0.0449826) (0.848,0.0443679) (0.856,0.0437464) (0.864,0.0431182) (0.872,0.0424834) (0.88,0.0418422) (0.888,0.0411947) (0.896,0.0405411) (0.904,0.0398815) (0.912,0.0392161) (0.92,0.0385453) (0.928,0.0378691) (0.936,0.0371878) (0.944,0.0365018) (0.952,0.0358112) (0.96,0.0351165) (0.968,0.034418) (0.976,0.033716) (0.984,0.0330109) (0.992,0.0323031) (1,0.031593) (1.008,0.0308812) (1.016,0.0301681) (1.024,0.0294542) (1.032,0.02874) (1.04,0.0280262) (1.048,0.0273133) (1.056,0.0266019) (1.064,0.0258928) (1.072,0.0251866) (1.08,0.024484) (1.088,0.0237857) (1.096,0.0230927) (1.104,0.0224056) (1.112,0.0217254) (1.12,0.0210529) (1.128,0.020389) (1.136,0.0197347) (1.144,0.019091) (1.152,0.0184589) (1.16,0.0178395) (1.168,0.0172338) (1.176,0.0166429) (1.184,0.0160681) (1.192,0.0155106) (1.2,0.0149715) (1.208,0.0144525) (1.216,0.0139554) (1.224,0.0134819) (1.232,0.0130333) (1.24,0.0126111) (1.248,0.0122165) (1.256,0.0118503) (1.264,0.0115132) (1.272,0.0112061) (1.28,0.0109298) (1.288,0.0106849) (1.296,0.0104723) (1.304,0.0102928) (1.312,0.010147) (1.32,0.0100357) (1.328,0.00995978) (1.336,0.00991985) (1.344,0.00991668) (1.352,0.009951) (1.36,0.0100235) (1.368,0.010135) (1.376,0.010286) (1.384,0.0104774) (1.392,0.0107098) (1.4,0.0109839) (1.408,0.0113003) (1.416,0.0116597) (1.424,0.0120628) (1.432,0.0125101) (1.44,0.0130023) (1.448,0.0135399) (1.456,0.0141236) (1.464,0.0147539) (1.472,0.0154313) (1.48,0.0161564) (1.488,0.0169297) (1.496,0.0177516) (1.504,0.0186226) (1.512,0.0195431) (1.52,0.0205135) (1.528,0.021534) (1.536,0.0226051) (1.544,0.023727) (1.552,0.0248999) (1.56,0.0261238) (1.568,0.0273988) (1.576,0.0287249) (1.584,0.0301019) (1.592,0.0315293) (1.6,0.0330065) (1.608,0.0345326) (1.616,0.0361062) (1.624,0.037725) (1.632,0.0393862) (1.64,0.0410854) (1.648,0.0428163) (1.656,0.0445698) (1.664,0.0463333) (1.672,0.0480892) (1.68,0.0498142) (1.688,0.0514792) (1.696,0.053051) (1.704,0.0544976) (1.712,0.0557948) (1.72,0.0569336) (1.728,0.0579215) (1.736,0.0587787) (1.744,0.0595311) (1.752,0.0602036) (1.76,0.0608169) (1.768,0.0613869) (1.776,0.0619249) (1.784,0.0624388) (1.792,0.062934) (1.8,0.0634141) (1.808,0.0638817) (1.816,0.0643384) (1.824,0.0647854) (1.832,0.0652236) (1.84,0.0656534) (1.848,0.0660766) (1.856,0.0664947) (1.864,0.0669085) (1.872,0.0673183) (1.88,0.0677241) (1.888,0.0681255) (1.896,0.068522) (1.904,0.0689127) (1.912,0.0677973) (1.92,0.0667475) (1.928,0.0657427) (1.936,0.0647723) (1.944,0.0638304) (1.952,0.062913) (1.96,0.0620177) (1.968,0.0611425) (1.976,0.0602859) (1.984,0.0594467) (1.992,0.0586237) (2,0.0591479)};
\addlegendentry{Optimal Control (with $l_j = 1$)}
\addplot+[color={rgb,255:red,214;green,0;blue,252}, solid, line width=1.08pt, mark=none] coordinates {(0,0.0835513) (0.008,0.0830233) (0.016,0.0825902) (0.024,0.0822019) (0.032,0.0818413) (0.04,0.0814996) (0.048,0.0811721) (0.056,0.0808555) (0.064,0.0805477) (0.072,0.0802471) (0.08,0.0799526) (0.088,0.0796631) (0.096,0.079378) (0.104,0.0790966) (0.112,0.0788183) (0.12,0.0785428) (0.128,0.0782697) (0.136,0.0779985) (0.144,0.0777291) (0.152,0.077461) (0.16,0.0771942) (0.168,0.0769283) (0.176,0.0766632) (0.184,0.0763987) (0.192,0.0761345) (0.2,0.0758706) (0.208,0.0756067) (0.216,0.0753427) (0.224,0.0750785) (0.232,0.0748139) (0.24,0.0745488) (0.248,0.074283) (0.256,0.0740165) (0.264,0.0737491) (0.272,0.0734806) (0.28,0.073211) (0.288,0.0729401) (0.296,0.0726678) (0.304,0.072394) (0.312,0.0721186) (0.32,0.0718415) (0.328,0.0715626) (0.336,0.0712817) (0.344,0.0709987) (0.352,0.0707136) (0.36,0.0704262) (0.368,0.0701364) (0.376,0.0698441) (0.384,0.0695491) (0.392,0.0692515) (0.4,0.068951) (0.408,0.0686476) (0.416,0.0683412) (0.424,0.0680316) (0.432,0.0677187) (0.44,0.0674024) (0.448,0.0670827) (0.456,0.0667593) (0.464,0.0664323) (0.472,0.0661014) (0.48,0.0657666) (0.488,0.0654277) (0.496,0.0650848) (0.504,0.0647375) (0.512,0.0643859) (0.52,0.0640298) (0.528,0.0636692) (0.536,0.0633038) (0.544,0.0629336) (0.552,0.0625585) (0.56,0.0621784) (0.568,0.0617932) (0.576,0.0614027) (0.584,0.0610069) (0.592,0.0606057) (0.6,0.0601989) (0.608,0.0597864) (0.616,0.0593682) (0.624,0.0589442) (0.632,0.0585142) (0.64,0.0580781) (0.648,0.057636) (0.656,0.0571876) (0.664,0.0567329) (0.672,0.0562718) (0.68,0.0558042) (0.688,0.0553301) (0.696,0.0548494) (0.704,0.054362) (0.712,0.0538678) (0.72,0.0533669) (0.728,0.0528592) (0.736,0.0523449) (0.744,0.0518242) (0.752,0.0512965) (0.76,0.0507618) (0.768,0.0502201) (0.776,0.0496713) (0.784,0.0491154) (0.792,0.0485525) (0.8,0.0479826) (0.808,0.0474056) (0.816,0.0468217) (0.824,0.0462312) (0.832,0.0456343) (0.84,0.0450312) (0.848,0.0444225) (0.856,0.0438082) (0.864,0.0431888) (0.872,0.0425646) (0.88,0.0419361) (0.888,0.0413034) (0.896,0.0406672) (0.904,0.0400277) (0.912,0.0393853) (0.92,0.0387405) (0.928,0.0380937) (0.936,0.0374451) (0.944,0.0367951) (0.952,0.0361438) (0.96,0.0354915) (0.968,0.0348384) (0.976,0.0341848) (0.984,0.033531) (0.992,0.0328772) (1,0.0322238) (1.008,0.0315711) (1.016,0.0309193) (1.024,0.0302688) (1.032,0.02962) (1.04,0.0289731) (1.048,0.0283285) (1.056,0.0276867) (1.064,0.0270479) (1.072,0.0264126) (1.08,0.0257812) (1.088,0.025154) (1.096,0.0245316) (1.104,0.0239143) (1.112,0.0233025) (1.12,0.0226968) (1.128,0.0220976) (1.136,0.0215054) (1.144,0.0209206) (1.152,0.0203437) (1.16,0.0197753) (1.168,0.0192159) (1.176,0.0186659) (1.184,0.018126) (1.192,0.0175966) (1.2,0.0170784) (1.208,0.0165718) (1.216,0.0160775) (1.224,0.015596) (1.232,0.0151279) (1.24,0.0146738) (1.248,0.0142344) (1.256,0.0138101) (1.264,0.0134017) (1.272,0.0130098) (1.28,0.012635) (1.288,0.0122779) (1.296,0.0119392) (1.304,0.0116195) (1.312,0.0113195) (1.32,0.0110398) (1.328,0.0107811) (1.336,0.0105442) (1.344,0.0103295) (1.352,0.0101379) (1.36,0.00997001) (1.368,0.00982649) (1.376,0.00970803) (1.384,0.00961532) (1.392,0.00954905) (1.4,0.0095099) (1.408,0.00949857) (1.416,0.00951574) (1.424,0.00956209) (1.432,0.00963833) (1.44,0.00974512) (1.448,0.00988315) (1.456,0.0100531) (1.464,0.0102557) (1.472,0.0104915) (1.48,0.0107613) (1.488,0.0110656) (1.496,0.0114053) (1.504,0.0117809) (1.512,0.0121933) (1.52,0.0126433) (1.528,0.013132) (1.536,0.0136602) (1.544,0.014229) (1.552,0.0148394) (1.56,0.0154924) (1.568,0.0161892) (1.576,0.016931) (1.584,0.0177189) (1.592,0.0185543) (1.6,0.0194383) (1.608,0.0203724) (1.616,0.0213578) (1.624,0.0223961) (1.632,0.0234886) (1.64,0.024637) (1.648,0.0258426) (1.656,0.0271071) (1.664,0.0284319) (1.672,0.0298186) (1.68,0.0312687) (1.688,0.0327834) (1.696,0.0343639) (1.704,0.0360109) (1.712,0.0377239) (1.72,0.0395007) (1.728,0.0413367) (1.736,0.0432244) (1.744,0.0442458) (1.752,0.0427688) (1.76,0.0413934) (1.768,0.0400989) (1.776,0.0388698) (1.784,0.037695) (1.792,0.0365661) (1.8,0.0354772) (1.808,0.0344233) (1.816,0.0334009) (1.824,0.032407) (1.832,0.0314394) (1.84,0.030496) (1.848,0.0295754) (1.856,0.0308423) (1.864,0.0321537) (1.872,0.033516) (1.88,0.0349331) (1.888,0.0364077) (1.896,0.0379414) (1.904,0.0395356) (1.912,0.041191) (1.92,0.0429079) (1.928,0.044685) (1.936,0.0465194) (1.944,0.0484062) (1.952,0.0503374) (1.96,0.0523007) (1.968,0.054278) (1.976,0.0562438) (1.984,0.0581645) (1.992,0.0599988) (2,0.0617038)};
\addlegendentry{Optimal Control (with $l_j = 0$)}
\end{axis}
\end{tikzpicture}
         }
         \caption{$\int_\Omega (\sigma-\sigma_Q)^2dx$}
         \label{fig:int_sigma_exp2}
     \end{subfigure}
     \hfill
     \begin{subfigure}[b]{0.49\textwidth}
         \centering
         \resizebox{\linewidth}{!}{
         \begin{tikzpicture}
\begin{axis}[width=12cm, height=6.8cm, xmin=0, xmax=2, xlabel={Time}, ylabel={Volume}, grid=both, ymajorgrids, grid style={gray!30}, legend cell align=left, legend style={font=\normalsize, at={(0.03,0.97)}, anchor=north west, fill=white, fill opacity=0.92, text opacity=1, draw=gray!45, rounded corners=2pt, inner xsep=6pt, inner ysep=4pt}, tick label style={font=\normalsize}, label style={font=\normalsize}, title style={font=\normalsize}, ]
\addplot+[color={rgb,255:red,0;green,0;blue,255}, dashed, line width=1.3pt, mark=none] coordinates {(0.008,0.0396039) (0.016,0.04072) (0.024,0.0418715) (0.032,0.0430569) (0.04,0.0442763) (0.048,0.0455303) (0.056,0.0468197) (0.064,0.0481454) (0.072,0.0495084) (0.08,0.0509096) (0.088,0.0523501) (0.096,0.0538308) (0.104,0.0553529) (0.112,0.0569175) (0.12,0.0585258) (0.128,0.0601788) (0.136,0.0618778) (0.144,0.0636241) (0.152,0.0654189) (0.16,0.0672634) (0.168,0.0691591) (0.176,0.0711073) (0.184,0.0731093) (0.192,0.0751666) (0.2,0.0772807) (0.208,0.079453) (0.216,0.0816851) (0.224,0.0839786) (0.232,0.086335) (0.24,0.088756) (0.248,0.0912433) (0.256,0.0937986) (0.264,0.0964237) (0.272,0.0991203) (0.28,0.10189) (0.288,0.104736) (0.296,0.107658) (0.304,0.11066) (0.312,0.113742) (0.32,0.116908) (0.328,0.12016) (0.336,0.123498) (0.344,0.126927) (0.352,0.130447) (0.36,0.134061) (0.368,0.137772) (0.376,0.141581) (0.384,0.145491) (0.392,0.149506) (0.4,0.153626) (0.408,0.157855) (0.416,0.162195) (0.424,0.166649) (0.432,0.17122) (0.44,0.17591) (0.448,0.180722) (0.456,0.185659) (0.464,0.190723) (0.472,0.195918) (0.48,0.201247) (0.488,0.206713) (0.496,0.212318) (0.504,0.218066) (0.512,0.223959) (0.52,0.230002) (0.528,0.236197) (0.536,0.242547) (0.544,0.249056) (0.552,0.255727) (0.56,0.262564) (0.568,0.269569) (0.576,0.276746) (0.584,0.284099) (0.592,0.291631) (0.6,0.299346) (0.608,0.307246) (0.616,0.315336) (0.624,0.323619) (0.632,0.332098) (0.64,0.340778) (0.648,0.349661) (0.656,0.358752) (0.664,0.368053) (0.672,0.377568) (0.68,0.387302) (0.688,0.397256) (0.696,0.407435) (0.704,0.417843) (0.712,0.428482) (0.72,0.439356) (0.728,0.450468) (0.736,0.461822) (0.744,0.473421) (0.752,0.485268) (0.76,0.497365) (0.768,0.509717) (0.776,0.522326) (0.784,0.535194) (0.792,0.548325) (0.8,0.561721) (0.808,0.575384) (0.816,0.589317) (0.824,0.603523) (0.832,0.618002) (0.84,0.632758) (0.848,0.647791) (0.856,0.663104) (0.864,0.678697) (0.872,0.694573) (0.88,0.710731) (0.888,0.727173) (0.896,0.743899) (0.904,0.76091) (0.912,0.778206) (0.92,0.795785) (0.928,0.813649) (0.936,0.831796) (0.944,0.850225) (0.952,0.868935) (0.96,0.887924) (0.968,0.90719) (0.976,0.926731) (0.984,0.946544) (0.992,0.966626) (1,0.986974) (1.008,1.00758) (1.016,1.02845) (1.024,1.04957) (1.032,1.07094) (1.04,1.09256) (1.048,1.11441) (1.056,1.13649) (1.064,1.15879) (1.072,1.18132) (1.08,1.20405) (1.088,1.22698) (1.096,1.25011) (1.104,1.27342) (1.112,1.29691) (1.12,1.32056) (1.128,1.34437) (1.136,1.36832) (1.144,1.39241) (1.152,1.41661) (1.16,1.44092) (1.168,1.46533) (1.176,1.48981) (1.184,1.51437) (1.192,1.53897) (1.2,1.56362) (1.208,1.58828) (1.216,1.61295) (1.224,1.6376) (1.232,1.66223) (1.24,1.6868) (1.248,1.71131) (1.256,1.73573) (1.264,1.76005) (1.272,1.78423) (1.28,1.80826) (1.288,1.83212) (1.296,1.85578) (1.304,1.87922) (1.312,1.90241) (1.32,1.92532) (1.328,1.94793) (1.336,1.97021) (1.344,1.99213) (1.352,2.01365) (1.36,2.03475) (1.368,2.05538) (1.376,2.07552) (1.384,2.09512) (1.392,2.11415) (1.4,2.13255) (1.408,2.15029) (1.416,2.16731) (1.424,2.18356) (1.432,2.19899) (1.44,2.21353) (1.448,2.22712) (1.456,2.23969) (1.464,2.25116) (1.472,2.26143) (1.48,2.27042) (1.488,2.27802) (1.496,2.2841) (1.504,2.28852) (1.512,2.29112) (1.52,2.29171) (1.528,2.29006) (1.536,2.2859) (1.544,2.2789) (1.552,2.26866) (1.56,2.25466) (1.568,2.23626) (1.576,2.21263) (1.584,2.18275) (1.592,2.14531) (1.6,2.09876) (1.608,2.0414) (1.616,1.97168) (1.624,1.88897) (1.632,1.79476) (1.64,1.69388) (1.648,1.59431) (1.656,1.50443) (1.664,1.42917) (1.672,1.36874) (1.68,1.32046) (1.688,1.28101) (1.696,1.2477) (1.704,1.21865) (1.712,1.19261) (1.72,1.1688) (1.728,1.14671) (1.736,1.12601) (1.744,1.10647) (1.752,1.08793) (1.76,1.07028) (1.768,1.05343) (1.776,1.03731) (1.784,1.02187) (1.792,1.00705) (1.8,0.992817) (1.808,0.979132) (1.816,0.965963) (1.824,0.953279) (1.832,0.941054) (1.84,0.929264) (1.848,0.917886) (1.856,0.906898) (1.864,0.896284) (1.872,0.886023) (1.88,0.876102) (1.888,0.866503) (1.896,0.857213) (1.904,0.848219) (1.912,0.839509) (1.92,0.831071) (1.928,0.822894) (1.936,0.814969) (1.944,0.807284) (1.952,0.799832) (1.96,0.792605) (1.968,0.785593) (1.976,0.778789) (1.984,0.772187) (1.992,0.765779) (2,0.759558)};
\addlegendentry{Uncontrolled}
\addplot+[color={rgb,255:red,255;green,0;blue,0}, solid, line width=0.99pt, mark=none] coordinates {(0.008,0.0396039) (0.016,0.04072) (0.024,0.0418715) (0.032,0.0430569) (0.04,0.0442763) (0.048,0.0455303) (0.056,0.0468197) (0.064,0.0481454) (0.072,0.0495084) (0.08,0.0509096) (0.088,0.0523501) (0.096,0.0538308) (0.104,0.0553529) (0.112,0.0569175) (0.12,0.0585258) (0.128,0.0601788) (0.136,0.0618778) (0.144,0.0636241) (0.152,0.0654189) (0.16,0.0672634) (0.168,0.0691591) (0.176,0.0711073) (0.184,0.0731093) (0.192,0.0751666) (0.2,0.0772807) (0.208,0.079453) (0.216,0.0816851) (0.224,0.0839786) (0.232,0.086335) (0.24,0.088756) (0.248,0.0912433) (0.256,0.0937986) (0.264,0.0964237) (0.272,0.0991203) (0.28,0.10189) (0.288,0.104736) (0.296,0.107658) (0.304,0.11066) (0.312,0.113742) (0.32,0.116908) (0.328,0.12016) (0.336,0.123498) (0.344,0.126927) (0.352,0.130447) (0.36,0.134061) (0.368,0.137772) (0.376,0.141581) (0.384,0.145491) (0.392,0.149506) (0.4,0.153626) (0.408,0.157855) (0.416,0.162195) (0.424,0.166649) (0.432,0.17122) (0.44,0.17591) (0.448,0.180722) (0.456,0.185659) (0.464,0.190723) (0.472,0.195918) (0.48,0.201247) (0.488,0.206713) (0.496,0.212318) (0.504,0.218066) (0.512,0.223959) (0.52,0.230002) (0.528,0.236197) (0.536,0.242547) (0.544,0.249056) (0.552,0.255727) (0.56,0.262564) (0.568,0.269569) (0.576,0.276746) (0.584,0.284099) (0.592,0.291631) (0.6,0.299346) (0.608,0.307246) (0.616,0.315336) (0.624,0.323619) (0.632,0.332098) (0.64,0.340778) (0.648,0.349661) (0.656,0.358752) (0.664,0.368053) (0.672,0.377568) (0.68,0.387302) (0.688,0.397256) (0.696,0.407435) (0.704,0.417843) (0.712,0.428482) (0.72,0.439356) (0.728,0.450468) (0.736,0.461822) (0.744,0.473421) (0.752,0.485268) (0.76,0.497365) (0.768,0.509717) (0.776,0.522326) (0.784,0.535194) (0.792,0.548325) (0.8,0.561721) (0.808,0.575384) (0.816,0.589317) (0.824,0.603523) (0.832,0.618002) (0.84,0.632758) (0.848,0.647791) (0.856,0.663104) (0.864,0.678697) (0.872,0.694573) (0.88,0.710731) (0.888,0.727173) (0.896,0.743899) (0.904,0.76091) (0.912,0.778206) (0.92,0.795785) (0.928,0.813649) (0.936,0.831796) (0.944,0.850225) (0.952,0.868935) (0.96,0.887924) (0.968,0.90719) (0.976,0.926731) (0.984,0.946544) (0.992,0.966626) (1,0.986974) (1.008,1.00758) (1.016,1.02845) (1.024,1.04957) (1.032,1.07094) (1.04,1.09256) (1.048,1.11441) (1.056,1.13649) (1.064,1.15879) (1.072,1.18132) (1.08,1.20405) (1.088,1.22698) (1.096,1.25011) (1.104,1.27342) (1.112,1.29691) (1.12,1.32056) (1.128,1.34437) (1.136,1.36832) (1.144,1.39241) (1.152,1.41661) (1.16,1.44092) (1.168,1.46533) (1.176,1.48981) (1.184,1.51437) (1.192,1.53897) (1.2,1.56362) (1.208,1.58559) (1.216,1.60289) (1.224,1.61557) (1.232,1.62371) (1.24,1.62744) (1.248,1.62732) (1.256,1.62703) (1.264,1.62656) (1.272,1.6259) (1.28,1.62506) (1.288,1.62403) (1.296,1.62281) (1.304,1.62139) (1.312,1.61979) (1.32,1.61798) (1.328,1.61598) (1.336,1.61377) (1.344,1.61136) (1.352,1.60875) (1.36,1.60592) (1.368,1.60289) (1.376,1.59965) (1.384,1.59619) (1.392,1.59251) (1.4,1.58862) (1.408,1.5845) (1.416,1.58016) (1.424,1.57559) (1.432,1.57078) (1.44,1.56575) (1.448,1.56048) (1.456,1.55497) (1.464,1.54921) (1.472,1.5432) (1.48,1.53694) (1.488,1.53043) (1.496,1.52365) (1.504,1.5166) (1.512,1.50928) (1.52,1.50167) (1.528,1.49378) (1.536,1.48559) (1.544,1.4771) (1.552,1.46829) (1.56,1.45916) (1.568,1.44969) (1.576,1.43986) (1.584,1.42968) (1.592,1.41911) (1.6,1.40814) (1.608,1.39675) (1.616,1.38492) (1.624,1.37261) (1.632,1.35979) (1.64,1.34644) (1.648,1.33251) (1.656,1.31795) (1.664,1.30272) (1.672,1.28674) (1.68,1.26996) (1.688,1.2523) (1.696,1.23366) (1.704,1.21396) (1.712,1.19308) (1.72,1.1709) (1.728,1.14732) (1.736,1.12219) (1.744,1.09542) (1.752,1.06689) (1.76,1.03656) (1.768,1.00443) (1.776,0.970571) (1.784,0.935193) (1.792,0.898632) (1.8,0.861363) (1.808,0.823991) (1.816,0.787188) (1.824,0.751625) (1.832,0.717881) (1.84,0.686375) (1.848,0.658251) (1.856,0.633772) (1.864,0.612349) (1.872,0.5934) (1.88,0.576403) (1.888,0.560886) (1.896,0.546417) (1.904,0.53258) (1.912,0.518764) (1.92,0.506488) (1.928,0.496866) (1.936,0.489003) (1.944,0.48225) (1.952,0.476287) (1.96,0.47093) (1.968,0.466069) (1.976,0.461628) (1.984,0.457555) (1.992,0.45381) (2,0.450361)};
\addlegendentry{Optimal Control (with $l_j = 1$)}
\addplot+[color={rgb,255:red,214;green,0;blue,252}, solid, line width=1.08pt, mark=none] coordinates {(0.008,0.0396039) (0.016,0.04072) (0.024,0.0418715) (0.032,0.0430569) (0.04,0.0442763) (0.048,0.0455303) (0.056,0.0468197) (0.064,0.0481454) (0.072,0.0495084) (0.08,0.0509096) (0.088,0.0523501) (0.096,0.0538308) (0.104,0.0553529) (0.112,0.0569175) (0.12,0.0585258) (0.128,0.0601788) (0.136,0.0618778) (0.144,0.0636241) (0.152,0.0654189) (0.16,0.0672634) (0.168,0.0691591) (0.176,0.0711073) (0.184,0.0731093) (0.192,0.0751666) (0.2,0.0772807) (0.208,0.079453) (0.216,0.0816851) (0.224,0.0839786) (0.232,0.086335) (0.24,0.088756) (0.248,0.0912433) (0.256,0.0937986) (0.264,0.0964237) (0.272,0.0991203) (0.28,0.10189) (0.288,0.104736) (0.296,0.107658) (0.304,0.11066) (0.312,0.113742) (0.32,0.116908) (0.328,0.12016) (0.336,0.123498) (0.344,0.126927) (0.352,0.130447) (0.36,0.134061) (0.368,0.137772) (0.376,0.141581) (0.384,0.145491) (0.392,0.149506) (0.4,0.153626) (0.408,0.157855) (0.416,0.162195) (0.424,0.166649) (0.432,0.171219) (0.44,0.175908) (0.448,0.180717) (0.456,0.18565) (0.464,0.190709) (0.472,0.195897) (0.48,0.201218) (0.488,0.206673) (0.496,0.212266) (0.504,0.217999) (0.512,0.223876) (0.52,0.2299) (0.528,0.236074) (0.536,0.2424) (0.544,0.248882) (0.552,0.255523) (0.56,0.262326) (0.568,0.269294) (0.576,0.27643) (0.584,0.283738) (0.592,0.291221) (0.6,0.298882) (0.608,0.306724) (0.616,0.314751) (0.624,0.322965) (0.632,0.331371) (0.64,0.33997) (0.648,0.348767) (0.656,0.357764) (0.664,0.366965) (0.672,0.376373) (0.68,0.385991) (0.688,0.395822) (0.696,0.405869) (0.704,0.416135) (0.712,0.426624) (0.72,0.437338) (0.728,0.447922) (0.736,0.458334) (0.744,0.468549) (0.752,0.480149) (0.76,0.49199) (0.768,0.504072) (0.776,0.516397) (0.784,0.528968) (0.792,0.541788) (0.8,0.554857) (0.808,0.568179) (0.816,0.581519) (0.824,0.594596) (0.832,0.607377) (0.84,0.619832) (0.848,0.631931) (0.856,0.643647) (0.864,0.654953) (0.872,0.665827) (0.88,0.676247) (0.888,0.686194) (0.896,0.695651) (0.904,0.704604) (0.912,0.713039) (0.92,0.720946) (0.928,0.728315) (0.936,0.735509) (0.944,0.742709) (0.952,0.749913) (0.96,0.757119) (0.968,0.764325) (0.976,0.77153) (0.984,0.778732) (0.992,0.785929) (1,0.793119) (1.008,0.800301) (1.016,0.807473) (1.024,0.814633) (1.032,0.821779) (1.04,0.828909) (1.048,0.836022) (1.056,0.843115) (1.064,0.850187) (1.072,0.857236) (1.08,0.86426) (1.088,0.871257) (1.096,0.878225) (1.104,0.885163) (1.112,0.892067) (1.12,0.898937) (1.128,0.905771) (1.136,0.912566) (1.144,0.91932) (1.152,0.926032) (1.16,0.9327) (1.168,0.939321) (1.176,0.945894) (1.184,0.952417) (1.192,0.958888) (1.2,0.965305) (1.208,0.971665) (1.216,0.977967) (1.224,0.98421) (1.232,0.99039) (1.24,0.996507) (1.248,1.00256) (1.256,1.00854) (1.264,1.01445) (1.272,1.0203) (1.28,1.02607) (1.288,1.03176) (1.296,1.03737) (1.304,1.04291) (1.312,1.04837) (1.32,1.05374) (1.328,1.05903) (1.336,1.06423) (1.344,1.06934) (1.352,1.07436) (1.36,1.07929) (1.368,1.08413) (1.376,1.08887) (1.384,1.09351) (1.392,1.09806) (1.4,1.1025) (1.408,1.10684) (1.416,1.11108) (1.424,1.11521) (1.432,1.11923) (1.44,1.12314) (1.448,1.12694) (1.456,1.13062) (1.464,1.13419) (1.472,1.13765) (1.48,1.14098) (1.488,1.1442) (1.496,1.14729) (1.504,1.15037) (1.512,1.154) (1.52,1.15819) (1.528,1.16293) (1.536,1.16822) (1.544,1.17405) (1.552,1.18042) (1.56,1.18733) (1.568,1.19477) (1.576,1.20275) (1.584,1.21127) (1.592,1.22032) (1.6,1.2299) (1.608,1.24002) (1.616,1.25068) (1.624,1.26188) (1.632,1.27363) (1.64,1.28594) (1.648,1.29882) (1.656,1.31229) (1.664,1.32637) (1.672,1.34109) (1.68,1.35649) (1.688,1.37261) (1.696,1.38954) (1.704,1.40735) (1.712,1.42579) (1.72,1.44398) (1.728,1.46187) (1.736,1.47943) (1.744,1.49663) (1.752,1.51348) (1.76,1.53008) (1.768,1.54659) (1.776,1.56312) (1.784,1.5797) (1.792,1.59639) (1.8,1.61319) (1.808,1.63013) (1.816,1.64721) (1.824,1.66444) (1.832,1.68183) (1.84,1.69938) (1.848,1.71712) (1.856,1.73499) (1.864,1.75295) (1.872,1.77087) (1.88,1.78869) (1.888,1.80636) (1.896,1.82387) (1.904,1.84126) (1.912,1.85865) (1.92,1.87636) (1.928,1.89359) (1.936,1.91021) (1.944,1.92615) (1.952,1.94132) (1.96,1.95564) (1.968,1.96903) (1.976,1.98137) (1.984,1.99255) (1.992,2.00242) (2,2.01084)};
\addlegendentry{Optimal Control (with $l_j = 0$)}
\end{axis}
\end{tikzpicture}
         }
         \caption{Tumor volume}
         \label{fig:vol_u_exp2}
     \end{subfigure}
    \caption{Evolution over time.}
    \label{fig:robustes_u_2}
\end{figure}
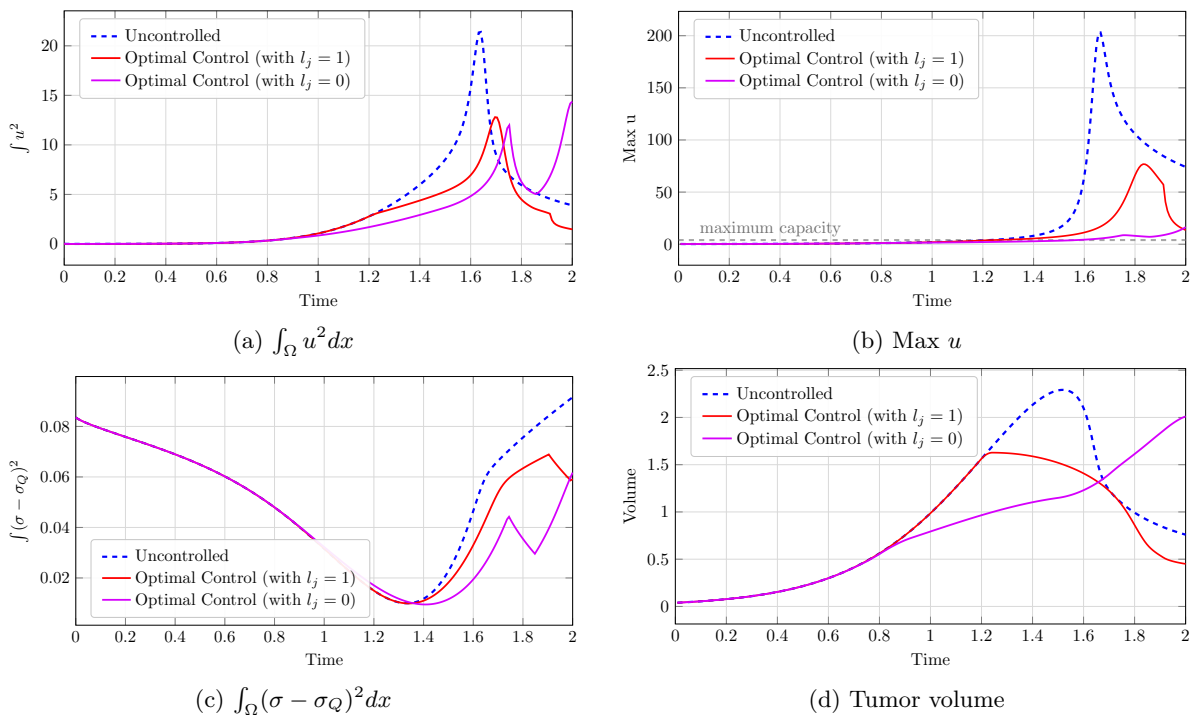

\appendix

\section{Appendix}
\begin{lemm}\label{lema_adjunto}
Let $u_0 \in H^1(\Omega)$ and consider $f_1 \in L^2(L^{2+})$, $\mathbf{v} \in L^4(Q) ^2$, and $f_2 \in L^2(Q)$. Then there exists a unique strong solution $u \in X_2$ of the following initial and boundary value problem:
\begin{equation}\label{eq_append:EqLemma}
\begin{cases}
    \partial_t u - \Delta u + f_1u + \mathbf{v} \cdot\nabla u = f_2, & \text{in }Q,\\
    \partial_{\mathbf{n}}u=0, & \text{on }\Sigma,\\
    u(0) = u_0, & \text{in }\Omega.
\end{cases}
\end{equation}

\end{lemm}
\begin{proof}
    Let $\{e_j\}$ be the basis of $H^1(\Omega)$ of functions such that
    \[
    \begin{cases}
    -\Delta e_j + e_j = \lambda_je_j,  &\text{in }\Omega \\
    \partial_{\mathbf{n}}e_j = 0, &\text{on }\partial\Omega,
    \end{cases}
    \]
    for each $j \in \mathbb{N}$. We then define $V_n = \text{span}\{e_1, e_2, \dots, e_n\}$ and $\mathbb{P}_n$ the projection of $H^1(\Omega)$ on $V_n$. Let $f_{1n},$ $f_{2n}\in C^\infty_0(\mathbb{R};\mathbb{R}^2)$ and $\mathbf{v}_{n}\in (C^\infty_0(\mathbb{R};\mathbb{R}^2))^2$ be mollifier regularizations of $f_1,f_2$ and $\mathbf{v}$ such that
    $$f_{1n}\rightarrow f_1\ \mbox{in}\ L^2(L^{2+}),\ \ \ f_{2n}\rightarrow f_2\ \mbox{in}\ L^2(L^{2}),\ \ \ \mathbf{v}_{n}\rightarrow \mathbf{v}\ \mbox{in}\ L^4(Q)^2.$$
    For each $n\in \mathbb{N}$ we look for the Galerkin approximation
    \[
    u_n(t,x) = \sum_{j=1}^n a_j(t) e_j(x) 
    \]
    such that
    \begin{equation}\label{eq_append:lemmaEq1}
    \langle \partial_t u_n, \varphi\rangle + (\nabla u_n, \nabla \varphi) +(f_{1n}u_n, \varphi) + (\mathbf{v}_n \nabla u_n, \varphi) =  
    (f_{2n},\varphi), \quad \forall \varphi \in V_n,
    \end{equation}
    with $u_n(0) = \mathbb{P}_n u_0$ a.e. in $\Omega$. From the linear ODE theory we have that there is a unique classical global solution $u_n \in C^1([0,T]; V_n)$, satisfying \eqref{eq_append:lemmaEq1} for each $n$.\\

    Next, we obtain a priori estimates for the sequence $\{u_n\}$ that are need for passing the limit as $n$ goes to $\infty.$ Taking $\varphi = u_n-\Delta u_n$ and using that $\|z\|_{H^2} \leq C\, \|z-\Delta z\|_{L^2}$, we obtain
    \[
    \frac{1}{2}\frac{d}{dt} \|u_n\|^2_{H^1} 
    %+ \|\nabla u_n\|^2_{L^2} 
    + \|u_n\|^2_{H^2} 
    = (f_{1n}u_n, u_n-\Delta u_n) 
    + (\mathbf{v}_n\cdot\nabla u_n, u_n - \Delta u_n) 
    + (f_{2n}, u_n-\Delta u_n).
    \]
    Applying the H\"older, Young, and Gagliardo-Nirenberg inequalities, we obtain
    \[
    \begin{split}
        (f_{1n}u_n, u_n-\Delta u_n) &\leq \delta \|u_n\|^2_{H^2} + C_\delta \|f_{1n}\|^2_{L^{2+}} \|u_n\|^2_{H^1},
%         \\
%        (\mathbf{v}\cdot\nabla u_n, u_n)& \leq \delta \|\nabla u\|^2_{H^1} + C_\delta \|\mathbf{v}\|^2_{L^4}\|u_n\|^2_{H^1}, 
      \\
        (\mathbf{v}_n\cdot\nabla u_n, u_n-\Delta u_n)& 
        \leq  \delta \|u_n\|^2_{H^2} + C_\delta \|\mathbf{v}_n\|^2_{L^4} \|\nabla u_n\|^2_{L^4}
        \\ & \leq \delta \|u_n\|^2_{H^2} + C_\delta \|\mathbf{v}_n\|^2_{L^4}\|u_n\|_{H^2}\|\nabla u_n\|_{L^2}
         \\
        &\leq 2\delta\|u_n\|^2_{H^2} 
        + C_\delta \|\mathbf{v}_n\|^4_{L^4}\|\nabla u_n\|^2_{L^2}, 
        \\
        (f_{2n}, u_n-\Delta u_n) & \leq \delta \|u_n\|^2_{H^2} + C_\delta \|f_{2n}\|^2_{L^2}.
    \end{split}
    \]
    Thus, applying the properties of the mollified sequences, taking $\delta>0$ small enough and applying the Gronwall inequality, we can obtain
    \[
    \|u_n\|_{X_2} \leq C(T, |\Omega|, \|\mathbf{v}\|_{L^4(Q)}, \|f_1\|_{L^2(L^{2+})}, \|f_2\|_{L^2(Q)}, \|u_n(0)\|_{L^2(Q)}),
    \]
    where $C$ is independent of $n$. Previous estimate is sufficient to pass to the limit as $n$ approaches infinity, obtaining a solution $u \in X_2$ of \eqref{eq_append:EqLemma}. The uniqueness follows standard arguments.
\end{proof}

\begin{theo}\label{teo_append:linealizado}
    Let $T > 0$ and $\Omega \subset \mathbb{R}^2$ be a bounded domain with smooth enough boundary. Suppose that
    $F_1 \in L^{2}(L^{2}), F_2 \in L^{2}(L^{1+}), F_3 \in L^4(Q), \mathbf{v} \in L^4(Q)^2, F_4 \in L^2(H^{-1}), G_1 \in L^{2}(L^{\infty}), G_2 \in L^{2}(L^{2+}), G_3 \in L^2(Q).$
    Then, the following linear system
    \begin{equation}\label{eq_append:lin}
    \begin{cases}
        \partial_t w - \Delta w + F_1 w + F_2 z + \nabla \cdot (w \mathbf{v} + F_3 \nabla z) = F_4, & \text{in }Q,\\
        \partial_t z - \Delta z + G_1 w + G_2 z = G_3, & \text{in }Q, \\
        (-\nabla w + w \mathbf{v} + F_3 \nabla z)\cdot \mathbf{n} = \partial_{\mathbf{n}}z = 0, & \text{on }\Sigma,\\
        w(0) = z(0) = 0, & \text{in }\Omega,
    \end{cases}
    \end{equation}
    has a unique weak-strong solution $[w,z] \in W_2\times X_2$.
\end{theo}
\begin{proof}
    For the proof we will use the Leray-Schauder Fixed Point Theorem. We consider the operator $\Gamma:L^{4-}(Q) \times L^{\infty-}(Q) \to W_2 \times X_2 \hookrightarrow L^{4-}(Q) \times L^{\infty-}(Q)$ where $\Gamma([\hat w,\hat z]) = [w,z]$ solves the decoupled linear system
    \begin{equation}\label{eq_append:lin1}
    \begin{cases}
        \displaystyle\int_0^T \langle \partial_t w, \varphi\rangle 
        + \int_Q \nabla w \cdot\nabla \varphi + \int_Q F_1 w\varphi
         + \int_Q F_2 z\varphi 
         = \int_Q  w \mathbf{v}\cdot\nabla \varphi 
         + \int_Q F_3 \nabla z\cdot\nabla \varphi 
         + \int_0^T \langle F_4, \varphi\rangle, \\
        \partial_t z - \Delta z + G_1 \hat w + G_2 z = G_3, & \text{in }Q, \\
        %\partial_{\mathbf{n}} w =
         \partial_{\mathbf{n}}z = 0, & \text{on }\Sigma, \\
        w(0) = z(0) = 0, & \text{in }\Omega,
    \end{cases}
    \end{equation}    
    for all $\varphi \in L^2(H^1)$. From the regularity of $G_1, G_2$ and $G_3$ and recalling that $\hat w \in L^{4-}(Q)$ y $\hat z \in L^{\infty-}(Q)$, it holds $G_1 \hat w + G_2 - G_3 \in L^2(Q)$. Then, Lemma \ref{lema_adjunto} with $\mathbf{v} = 0$ gives the existence of a unique $z \in X_2$ of \eqref{eq_append:lin1}$_2$.\\
    
Now, knowing the existence of $z \in X_2,$ then $\nabla z \in W_2 \hookrightarrow L^4(Q)$. In addition, from the regularity of $F_1, F_2, F_3, F_4$ and since $X_2 \hookrightarrow L^\infty(L^{\infty-})$, then $F_1 w + F_2 z - F_3\nabla z- F_4 \in L^2(Q)$.  Therefore, since $\mathbf{v}\in L^4(Q),$ by using a Galerkin procedure, it is not difficult to prove that there is a unique solution $w \in W_2$ of \eqref{eq_append:lin1}$_1$. Thus, the mapping $\Gamma$ is well defined. Furthermore, from the Aubin-Lions Lemma (see \cite{16} Theorem 5.1, p.58) and using Corollary 4 of \cite{23}, then $W_2 \times X_2 \hookrightarrow L^{4-}(Q) \times L^{\infty-}(Q)$ compactly, and therefore $\Gamma$ is compact.\\

Now, it remains to verify that the set of possible fixed points of $\theta \,\Gamma$, $\theta \in [0,1]$ is bounded in $W_2 \times X_2$. Let $\theta \in (0,1]$ (the case $\theta = 0$ is clear). Then, if $[w,z]$ is a fixed point of $\theta \,\Gamma$, it verifies the following system:
    \begin{equation}\label{eq_append:lin2}
    \begin{cases}
        \displaystyle\int_0^T \langle \partial_t w, \varphi\rangle + \int_Q \nabla w \cdot  \nabla \varphi +  \int_Q F_1  w\varphi + \int_Q F_2 z\varphi =  \int_Q   w \mathbf{v}\cdot\nabla \varphi  + \int_Q F_3 \nabla z\cdot\nabla \varphi + \theta \int_0^T \langle F_4, \varphi\rangle,\\
        \partial_t z - \Delta z + \theta G_1 w + G_2 z = \theta G_3, & \text{in }Q, \\
        \partial_{\mathbf{n}} w= \partial_{\mathbf{n}}z = 0, & \text{on }\Sigma, \\
        w(0) = z(0) = 0, & \text{in }\Omega,
    \end{cases}
    \end{equation}
    for all $\varphi \in L^2(H^1)$. Testing \eqref{eq_append:lin2}$_1$ by $w$ and integrating by parts, and using the H\"older, Young and Gagliardo-Niremberg inequalities we get
    \begin{equation}\label{eq_append:lin3}
    \begin{split}
        \frac{1}{2}\frac{d}{dt}\|w\|^2_{L^2} + \|w\|^2_{H^1} \leq & 2\delta \|w\|^2_{H^1} + \delta \|\nabla z\|^2_{H^1} + 3\delta\|\nabla w\|^2_{L^2} + C_\delta( \|F_1\|^2_{L^{2}}+\|\mathbf{v}\|^4_{L^4} + 1)\|w\|^2_{L^2} \\ &+C_\delta \|F_2\|^2_{L^{1+}}\|z\|^2_{H^1}  + C_\delta \|F_3\|^4_{L^4}\|\nabla z\|^2_{L^2} + C_\delta \|F_4\|^2_{H^{-1}}   ,
        \end{split}
    \end{equation}
    for any $\delta>0.$ Now, multiplying \eqref{eq_append:lin2}$_2$ by $z-\Delta z$, integrating by parts over $\Omega$ and applying H\"older and Young inequalities, and using that $\|z\|_{H^2} \leq C\, \|z-\Delta z\|_{L^2}$ we arrive at
    \begin{equation}\label{eq_append:lin4}
    \frac{1}{2}\frac{d}{dt}\|z\|_{H^1}^2 + C\, \|z\|^2_{H^2} \leq \delta \|z\|_{H^2}^2 + C_\delta \|G_1\|^2_{L^\infty} \|w\|_{L^2}^2 + C_\delta \|G_2\|^2_{L^{2+}}\|z\|^2_{H^1} + C_\delta \|G_3\|^2_{L^2}.
    \end{equation}
    Adding \eqref{eq_append:lin3} and \eqref{eq_append:lin4}, taking $\delta > 0$ sufficiently small, using the Gronwall inequality and by the regularity of $F_i$, $G_j$ for $i = 1,2,3, 4$ and $j = 1,2,3$, we obtain
    \[
    \begin{split}
    \|w\|^2_{W_2} + \|z\|^2_{X_2} \leq C(&T, |\Omega|, \|F_1\|_{L^2(L^{2+})}, \|F_2\|_{L^2(L^{1+})}, \|F_3\|_{L^4(Q)},\|F_4\|_{L^2(H^{-1})}, \|G_1\|_{L^2(L^\infty)}, \|G_2\|_{L^2(L^{2+})}, \\ &\|G_3\|_{L^2(Q)}, \|\mathbf{v}\|_{L^4(Q)}),
    \end{split}
    \]
    where $C$ is independent of $\theta$. In addition, it is a straightforward to check the continuity of $\Gamma$. Therefore, the Leray-Schauder Fixed Point Theorem guarantees the existence of the weak-strong solution of \eqref{eq_append:lin}. Moreover, using a classical comparison argument, 
    %we can prove that 
    the pair $[w,z] \in W_2 \times X_{2}$ is the unique solution of system \eqref{eq_append:lin}.
\end{proof}

\section{Appendix}

Let $\Omega \subset \mathbb{R}^2$ a bounded domain with smooth enough boundary, $T>0$, and 
$[\ControlC, \ControlS] \in \mathcal{U}_{ad}$ 
 a convex set of $ [L^\infty(Q)]^2$. Consider the following generic controlled system:
\begin{equation}\label{eq_apend:SistemaGeneral}
    \begin{cases}
        \partial_t u - \Delta u = f_1(u,v) + \ControlC f_2(u,v) - \nabla \cdot (u\nabla v), & \text{in }Q, \\
        \partial_t v - \Delta v = g_1 (u,v) + \ControlS g_2(u,v), & \text{in }Q,\\
        \partial_{\mathbf{n}}u = \partial_{\textbf{n}}v = 0, & \text{on }\Sigma, \\
        u(0) = u_0, \quad v(0) = v_0, & \text{in }\Omega.
    \end{cases}
\end{equation}
The main objective of this section is to analyze the linearized system corresponding to the generic system \eqref{eq_apend:SistemaGeneral}. For this, we will assume the following hypotheses:
\begin{itemize}
    \item[(\text{H1})] There exists a unique weak-strong solution $[u,v] \in W_2 \times X_2$ of the generic  system \eqref{eq_apend:SistemaGeneral}.
    \item[(\text{H2})] The weak-strong solutions of \eqref{eq_apend:SistemaGeneral} are stable with respect to the controls, that is, if $[\ControlC_1, \ControlS_1], [\ControlC_2, \ControlS_2] \in \mathcal{U}_{ad}$ and $[u_1, v_1]$, $[u_2, v_2]\in W_2 \times X_2$ are the corresponding states, then there exists $C>0,$ such that
    \[
    \|u_1-u_2\|_{W_2} + \|v_1- v_2\|_{X_2} \leq C(\|\ControlC_1 - \ControlC_2\|_{L^\infty(Q)} + \|\ControlS_1 - \ControlS_2\|_{L^\infty(Q)}).
    \]
\end{itemize}
From (H1) it follows that there  exists a unique weak-strong solution  $[u,v]\in W_2 \times X_2$. Using (H2) we compute the derivative of states with respect to the controls using a convex perturbation strategy. Let $[u^\varepsilon, v^\varepsilon]$ and $[u, v]$ solutions of the controlled general system \eqref{eq_apend:SistemaGeneral} associated with the controls $[\ControlC^\varepsilon, \ControlS^\varepsilon]$ and $[\ControlC, \ControlS]$ respectively, where $[\ControlC^\varepsilon, \ControlS^\varepsilon] = [\ControlC + \varepsilon(\bar \ControlC - \ControlC), \ControlS + \varepsilon(\bar \ControlS - \ControlS)] \in \mathcal{U}_{ad}$ for any $[\bar \ControlC, \bar \ControlS] \in \mathcal{U}_{ad}$ and $\varepsilon \in [0,1]$. Note that, letting $\varepsilon \to 0$, by construction 
$$\ControlC^\varepsilon \to \ControlC\quad \text{in } L^\infty(Q)\quad and \quad \ControlS^\varepsilon \to \ControlS\quad  \text{in } L^\infty(Q),
$$ and as a consequence of (H2), we get
\[
[u^\varepsilon, v^\varepsilon] \to [u,v] \quad \text{ strongly in } W_2\times X_2.
\]
Let us define  
$$w^\varepsilon = \frac{u^\varepsilon - u}{\varepsilon}, 
\quad z^\varepsilon = \frac{v^\varepsilon - v}{\varepsilon}.
$$
 Then, using that 
$$ \frac{\ControlC ^\varepsilon - \ControlC}{\varepsilon}=\bar \ControlC - \ControlC
\quad 
\text{and}
\quad 
\frac{\ControlS ^\varepsilon - \ControlS}{\varepsilon}=\bar \ControlS - \ControlS,
$$  
one has that $[w^\varepsilon,z^\varepsilon]$ satisfies the following system
\begin{equation}\label{eq_append:linealizadoEpsilon}
    \begin{cases}
        \partial_t w^\varepsilon - \Delta w^\varepsilon 
        = \frac{1}{\varepsilon}\Big(f_1(u^\varepsilon, v^\varepsilon)
         - f_1(u,v)\Big) + f_2(u^\varepsilon, v^\varepsilon)(\bar \ControlC - \ControlC) + \frac{1}{\varepsilon}\Big (f_2(u^\varepsilon, v^\varepsilon) - f_2(u,v)\Big )\ControlC \\ \hspace{2.4cm}- \nabla \cdot (w^\varepsilon\nabla v + u^\varepsilon\nabla z^\varepsilon), & \text{in }Q, \\
        \partial_t z^\varepsilon - \Delta z^\varepsilon = \frac{1}{\varepsilon}\Big(g_1(u^\varepsilon, v^\varepsilon) - g_1(u,v)\Big) + g_2(u^\varepsilon, v^\varepsilon)(\bar\ControlS - \ControlS) + \frac{1}{\varepsilon}\Big(g_2(u^\varepsilon, v^\varepsilon)-g_2(u,v)\Big)\ControlS, &\text{in }Q, \\
        \partial_{\mathbf{n}} w^\varepsilon = \partial_{\mathbf{n}}z^\varepsilon = 0, & \text{on }\Sigma,\\
        w^\varepsilon(0) = 0, \quad z^\varepsilon(0) = 0, & \text{in }\Omega.
    \end{cases}
\end{equation}
\begin{theo}\label{teo:Append_derivada}
Assume that $f_i,g_i \in C^2([0,\infty)^2)$ for $i=1,2$, such that for any $[u,v] \in W_2\times X_2$ with $\partial_{\mathbf{n}} v = 0$ on $\Sigma$, one has $f_2(u,v) \in L^2(L^{\infty-})$, $g_2(u,v) \in L^2(Q)$, and their first partial derivatives satisfy
\begin{align*}
    &\partial_uf_1(u,v) \in L^{2}(L^{2+}), && \partial_u f_2(u,v) \in L^{2}(L^{2+}), && \partial_u g_1(u,v) \in L^{2}(L^{\infty}), && \partial_u g_2(u,v) \in L^{2}(L^\infty),\\
    &\partial_v f_1(u,v) \in L^{2+}(Q), && \partial_v f_2(u,v) \in L^{2+}(Q), && \partial_v g_1(u,v) \in L^{2+}(Q), && \partial_v g_2(u,v) \in L^{2+}(Q),
\end{align*}
and their seconds partial derivatives satisfy
\begin{align*}
    &\partial_{uu}f_1(u,v), \partial_{uu}f_2(u,v) \in L^4(L^{4+}), && \partial_{vv}f_1(u,v), \partial_{vv}f_2(u,v) \in L^2(L^{1+}),\\ & \partial_{uv}f_1(u,v),\partial_{uv}f_2(u,v) \in L^2(L^{2+})\cap L^4(L^{4/3+}), &&
    \partial_{uu}g_1(u,v), \partial_{uu}g_2(u,v) \in L^\infty(Q),\\ & \partial_{vv}g_1(u,v), \partial_{vv}g_2(u,v) \in L^2(L^{2+}), && \partial_{uv}g_1(u,v),\partial_{uv}g_2(u,v) \in L^4(L^{4+}).
\end{align*}
    Then, under the hypotheses (H1) and (H2), there is the limit $\lim\limits_{\varepsilon\to 0}[w^\varepsilon, z^\varepsilon] = [w,z]$ in $W_2 \times X_2$, where $[w,z]$ is the solution of the following linearized (or sensitivity) system
    \begin{equation}\label{eq_append:linealizado}
        \begin{cases}
            \partial_t w - \Delta w  - (\partial_u f_1 (u,v) + \partial_uf_2(u,v)\ControlC)w - (\partial_v f_1(u,v) + \partial_v f_2(u,v)\ControlC)z \\ \qquad+ \nabla \cdot (w\nabla v + u\nabla z)  = f_2(u,v) (\bar \ControlC - \ControlC), & \text{in }Q,\\
            \partial_t z - \Delta z  - (\partial_u g_1(u,v) + \partial_ug_2(u,v) \ControlS)w - (\partial_v g_1(u,v) + \partial_vg_2(u,v)\ControlS)z = g_2(u,v) (\bar \ControlS - \ControlS), &\text{in }Q, \\
            \partial_{\mathbf{n}}w = \partial_{\mathbf{n}} z = 0, &\text{on }\Sigma, \\
            w(0) = 0, \quad z(0) = 0, &\text{in }\Omega.
        \end{cases}
    \end{equation}
\end{theo}
\begin{remark}
Note that the hypotheses on $f_i$ and $g_i$ for $i=1,2$ and their first derivatives imposed in Theorem 
\ref{teo:Append_derivada} 
are sufficient to ensure that for each $\varepsilon \in (0,1]$, using the Theorem \ref{teo_append:linealizado}, there exists a unique solution $[w^\varepsilon, z^\varepsilon] \in W_2 \times X_2$ of the system \eqref{eq_append:linealizadoEpsilon} associated with $[u^\varepsilon, v^\varepsilon, \ControlC^\varepsilon, \ControlS^\varepsilon],$  and a unique solution $[w,z]$ of the linearized system \eqref{eq_append:linealizado} associated with $[u,v,\ControlC, \ControlS]$.
\end{remark}
\begin{proof}
    Using an extension of the mean value Lagrange Theorem (see \cite[Appendix]{extTeoLagrange}), it holds
    \[
    f_i(u^\varepsilon, v^\varepsilon)-f_i(u,v) = \partial_u f_i (\tilde u^\varepsilon, \tilde v^\varepsilon)(u^\varepsilon-u) + \partial_v f_i(\tilde u^\varepsilon, \tilde v^\varepsilon)(v^\varepsilon - v),
    \]
    \[
    g_i(u^\varepsilon, v^\varepsilon)-g_i(u,v) = \partial_u g_i (\tilde u^\varepsilon, \tilde v^\varepsilon)(u^\varepsilon-u) + \partial_v g_i(\tilde u^\varepsilon, \tilde v^\varepsilon)(v^\varepsilon - v),
    \]
for $i = 1,2$, where $\tilde u^\varepsilon, \tilde v^\varepsilon$ are intermediate functions between those of $u^\varepsilon$ and $u$, and 
$v^\varepsilon$ and $v$, respectively. Then
    \[
    [\tilde u^\varepsilon, \tilde v^\varepsilon] \to [u, v] \quad \text{ strongly in } \quad  W_2 \times X_2.
    \]
    Therefore, the system \eqref{eq_append:linealizadoEpsilon} is rewriting as:
    \begin{equation}\label{eq_append:linealizadoEpsilon2}
        \begin{cases}
            \partial_t w^\varepsilon - \Delta w^\varepsilon - \Big(\partial_u f_1(\tilde u^\varepsilon, \tilde v^\varepsilon) + \partial_u f_2(\tilde u^\varepsilon, \tilde v^\varepsilon)\ControlC\Big)w - \Big(\partial_v f_1(\tilde u^\varepsilon, \tilde v^\varepsilon) + \partial_v f_2(\tilde u^\varepsilon, \tilde v^\varepsilon)\ControlC\Big)z \\ \hspace{2.5cm }= f_2(u^\varepsilon, v^\varepsilon)(\bar\ControlC - \ControlC) - \nabla \cdot (w^\varepsilon \nabla v + u^\varepsilon \nabla z^\varepsilon) , & \text{in }Q, \\
            \partial_t z^\varepsilon - \Delta z^\varepsilon - \Big(\partial_u g_1(\tilde u^\varepsilon, \tilde v^\varepsilon) + \partial_u g_2(\tilde u^\varepsilon, \tilde v^\varepsilon)\ControlS \Big)w - \Big(\partial_v g_1(\tilde u^\varepsilon, \tilde v^\varepsilon) + \partial_v g_2(\tilde u^\varepsilon, \tilde v^\varepsilon)\ControlS\Big)z \\ \hspace{2.5cm } = g_2(u^\varepsilon, v^\varepsilon)(\bar\ControlS - \ControlS), & \text{in }Q,\\
            \partial_{\mathbf{n}} w^\varepsilon = \partial_{\mathbf{n}}z^\varepsilon = 0, & \text{on }\Sigma, \\
            w^\varepsilon(0) = 0, \quad z^\varepsilon (0)=0, & \text{in }\Omega.
        \end{cases}
    \end{equation}
    Now, let us define the error functions $e_w = w^\varepsilon - w$ and $e_z =  z^\varepsilon - z$. Therefore, $[e_w, e_z]$ satisfies the following system (in the weak-strong sense of $W_2\times X_2$):    
    \begin{equation}\label{eq_append:error1}
        \begin{cases}
            \partial_t e_w - \Delta e_w + e_w= (A_1 +1)e_w + A_2e_z + A_3 w + A_4 z + A_5 - \nabla \cdot (e_w \nabla v + u^\varepsilon \nabla e_z + (u^\varepsilon-u)\nabla z), & \text{in }Q, \\
            \partial_t e_z - \Delta e_z + e_z = B_1 e_w + (B_2+1)e_z + B_3 w + B_4 z + B_5, & \text{in }Q, \\
            \partial_{\mathbf{n}} e_w = \partial_{\mathbf{n}}e_z = 0, & \text{on }\Sigma, \\
            e_w(0) = 0, \quad e_z(0) = 0, & \text{in }\Omega,
        \end{cases}
    \end{equation}
    where
    \[
    \begin{split}
        &A_1 = \partial_u f_1 (\tilde u^\varepsilon, \tilde v^\varepsilon) + \partial_u f_2(\tilde u^\varepsilon, \tilde v^\varepsilon)\ControlC,\\ &A_2 = \partial_v f_1 (\tilde u^\varepsilon, \tilde v^\varepsilon) + \partial_v f_2(\tilde u^\varepsilon, \tilde v^\varepsilon)\ControlC, \\
        & A_3 = \Big(\partial_u f_1(\tilde u^\varepsilon, \tilde v^\varepsilon) - \partial_u f_1(u,v)\Big) + \Big(\partial_u f_2(\tilde u^\varepsilon, \tilde v^\varepsilon) - \partial_uf_2( u,  v)\Big)\ControlC, \\
        & A_4 = \Big(\partial_v f_1(\tilde u^\varepsilon, \tilde v^\varepsilon) - \partial_v f_1(u,v)\Big) + \Big(\partial_v f_2(\tilde u^\varepsilon, \tilde v^\varepsilon) - \partial_vf_2( u,  v)\Big)\ControlC, \\
        & A_5 = \Big(f_2( u^\varepsilon,  v^\varepsilon)-f_2(u,v)\Big)(\bar \ControlC - \ControlC),
    \end{split}
    \]
    and
    \[
    \begin{split}
        &B_1 = \partial_u g_1 (\tilde u^\varepsilon, \tilde v^\varepsilon) + \partial_u g_2(\tilde u^\varepsilon, \tilde v^\varepsilon)\ControlS,\\ &B_2 = \partial_v g_1 (\tilde u^\varepsilon, \tilde v^\varepsilon) + \partial_v g_2(\tilde u^\varepsilon, \tilde v^\varepsilon)\ControlS, \\
        &B_3 = \Big(\partial_u g_1(\tilde u^\varepsilon, \tilde v^\varepsilon) - \partial_u g_1(u,v)\Big) + \Big(\partial_u g_2(\tilde u^\varepsilon, \tilde v^\varepsilon) - \partial_ug_2( u,  v)\Big)\ControlS, \\
        &B_4 = \Big(\partial_v g_1(\tilde u^\varepsilon, \tilde v^\varepsilon) - \partial_v g_1(u,v)\Big) + \Big(\partial_v g_2(\tilde u^\varepsilon, \tilde v^\varepsilon) - \partial_v g_2( u,  v)\Big)\ControlS, \\
        &B_5 = \Big(g_2(u^\varepsilon, v^\varepsilon)-g_2(u,v)\Big)(\bar \ControlS - \ControlS).
    \end{split}
    \]
    Multiplying \eqref{eq_append:error1}$_1$ by $e_w$ and \eqref{eq_append:error1}$_2$ by $e_z - \Delta e_z$, integrating by parts on $\Omega$ and adding, we get
    \[
    \begin{split}
    \frac{1}{2}\frac{d}{dt}\Big(\|e_w\|^2_{L^2} & + \|e_z\|^2_{H^1}\Big) + \|\nabla e_w\|^2_{L^2} + \|\nabla e_z\|^2_{L^2}+ \|\Delta e_z\|^2_{L^2} = \int_\Omega (A_1 + 1)e_w^2 + \int_\Omega A_2 e_z e_w + \int_\Omega A_3 w e_w \\ &+ \int_\Omega A_4 z e_w + \int_\Omega A_5 e_w + \int_\Omega e_w \nabla v \cdot \nabla e_w + \int_\Omega u^\varepsilon \nabla e_z \cdot \nabla e_w + \int_\Omega (u^\varepsilon - u)\nabla z\cdot  \nabla e_z\\ & + \int_\Omega B_1 e_w (e_z - \nabla e_z)  + \int_\Omega (B_2 + 1)e_z(e_z - \nabla e_z) + \int_\Omega B_3 w (e_z - \nabla e_z) \\ & + \int_\Omega B_4 z (e_z - \nabla e_z) + \int_\Omega B_5 (e_z - \nabla e_z).
    \end{split}
    \]
    Applying the H\"older and Young inequalities and using that $H^1(\Omega) \hookrightarrow L^{\infty-}(\Omega)$, we obtain that
    \begin{align*}
        & \int_\Omega A_1 e_w^2 \leq \delta \|e_w\|^2_{H^1}+ C_\delta \|A_1\|^2_{L^2}\|e_w\|^2_{L^2}, && \int_\Omega B_1e_w(e_z-\Delta e_z)\leq \delta \|e_z\|^2_{H^2} + C_\delta \|B_1\|^2_{L^4}\|e_w\|^2_{L^4},\\
        & \int_\Omega A_2 e_w e_z \leq \delta \|e_w\|^2_{H^1} + C_\delta \|A_2\|^2_{L^{1+}}\|e_z\|^2_{H^1}, &&\int_\Omega B_2 e_z(e_z - \Delta e_z) \leq \delta \|e_z\|^2_{H^2} + C_\delta \|B_2\|^2_{L^{2+}}\|e_z\|^2_{H^1}, \\
        & \int_\Omega A_3 w e_w \leq \delta \|e_w\|^2_{H^1} + C_\delta \|A_3\|_{L^{2+}}^2\|w\|^2_{L^2}, && \int_\Omega B_3 w(e_z-\Delta e_z) \leq \delta \|e_z\|^2_{H^2} + C_\delta \|B_3\|^2_{L^4}\|w\|^2_{L^4}, \\
        & \int_\Omega A_4 z e_w \leq \delta \|e_w\|^2_{H^1} + C_\delta C^2 \|A_4\|^2_{L^{1+}}\|z\|_{H^1}^2, && \int_\Omega B_4 z(e_z-\Delta e_z) \leq \delta \|e_z\|_{H^2}^2 +C_\delta \|B_4\|^2_{L^{2+}}\|z\|^2_{H^1} \\
        & \int_\Omega A_5 e_w \leq \delta \|e_w\|^2_{H^1} + C_\delta \|A_5\|^2_{L^{1+}}, &&  \int_\Omega B_5(e_z-\Delta e_z) \leq \delta \|e_z\|^2_{H^2} + C_\delta \|B_5\|^2_{L^2},
    \end{align*}
    Similarly, applying the Gagliardo-Nirenberg inequality, we get
    \[
    \begin{split}
        & \int_\Omega e_w \nabla v \cdot\nabla e_w \leq C\|e_w\|^{1/2}_{L^2}\|e_w\|^{1/2}_{H^1}\|\nabla v\|_{L^4}\|\nabla e_w\|_{L^2} \leq \delta \|e_w\|^2_{H^1} + C_\delta \|\nabla v\|^4_{L^4}\|e_w\|^2_{L^2},
        \\
        & \int_\Omega u^\varepsilon \nabla e_z \cdot \nabla e_w \leq C\|u^\varepsilon\|_{L^4}\|\nabla e_z\|^{1/2}_{L^2}\|\nabla e_z\|^{1/2}_{H^1}\|\nabla e_w\|_{L^2} \leq \delta (\|\nabla e_z\|^2_{H^1} + \|\nabla e_w\|^2_{L^2}) + C_\delta \|u^\varepsilon\|_{L^4}^4 \|\nabla e_z\|^2_{L^2}, \\
        & \int_\Omega (u^\varepsilon-u)\nabla z \cdot\nabla e_w \leq \delta\|e_w\|^2_{H^1} + C_\delta \|u^\varepsilon-u\|^2_{L^4}\|\nabla z\|_{L^4}^2.
    \end{split}
    \]
    Thus, taking $\delta$ sufficiently small and using the Gronwall inequality, we have
    \[
    \begin{split}
    \|e_w\|^2_{W_2} + &\|e_z\|_{X_2}^2 \leq C\Big(\|A_3\|^2_{L^{2}(L^{2+})}\|w\|^2_{L^\infty(L^2)}+\|A_4\|^2_{L^2(L^{1+})}\|z\|^2_{L^\infty(H^1)} + \|A_5\|^2_{L^2(L^{1+})} \\ & + \|u^\varepsilon-u\|^2_{L^4(Q)}\|\nabla z\|^2_{L^4(Q)} + \|B_3\|_{L^4(Q)}\|w\|^2_{L^4(Q)}+ \|B_4\|^2_{L^2(L^{2+})}\|z\|^2_{L^\infty(H^1)} + \|B_5\|^2_{L^2(Q)}\Big).
    \end{split}
    \]
    Now, using the mean value Lagrange Theorem again, and using hypothesis (H2), we obtain 
    \[
        \begin{split}
            \|A_3\|_{L^2(L^{2+})} & \leq  \|\partial_u f_1(\tilde u^\varepsilon, \tilde v^\varepsilon)-\partial_uf_1(u,v)\|_{L^2(L^{2+})} + \|\partial_u f_2(\tilde u^\varepsilon, \tilde v^\varepsilon) - \partial_u f_2(u,v)\|_{L^2(L^{2+})}\|\ControlC\|_{L^\infty(Q)} \\
        & = \|\partial_{uu}f_1(\tilde {\tilde u}^\varepsilon, \tilde{\tilde v}^\varepsilon)(\tilde u^\varepsilon - u) + \partial_{uv}f_1(\tilde {\tilde u}^\varepsilon, \tilde{\tilde v}^\varepsilon)(\tilde v^\varepsilon - v)\|_{L^2(L^{2+})} \\ & \qquad + \|\partial_{uu}f_2(\tilde {\tilde u}^\varepsilon, \tilde{\tilde v}^\varepsilon)(\tilde u^\varepsilon - u) + \partial_{uv}f_2(\tilde {\tilde u}^\varepsilon, \tilde{\tilde v}^\varepsilon)(\tilde v^\varepsilon - v)\|_{L^2(L^{2+})}\|\ControlC\|_{L^\infty(Q)} \\
        & \leq \|\partial_{uu}f_1(\tilde {\tilde u}^\varepsilon, \tilde{\tilde v}^\varepsilon)\|_{L^4(L^{4+})}\|\tilde u^\varepsilon - u\|_{L^4(Q)} + \|\partial_{uv}f_1(\tilde {\tilde u}^\varepsilon, \tilde{\tilde v}^\varepsilon)\|_{L^2(L^{2+})}\|\tilde v^\varepsilon - v\|_{L^\infty(L^{\infty-})} \\
        & \qquad + \|\partial_{uu}f_2(\tilde {\tilde u}^\varepsilon, \tilde{\tilde v}^\varepsilon)\|_{L^4(L^{4+})}\|\tilde u^\varepsilon - u\|_{L^4(Q)}\|\ControlC\|_{L^\infty(Q)} \\ & \qquad + \|\partial_{uv}f_2(\tilde {\tilde u}^\varepsilon, \tilde{\tilde v}^\varepsilon)\|_{L^2(L^{2+})}\|\tilde u^\varepsilon - u\|_{L^\infty(L^{\infty-})}\|\ControlC\|_{L^\infty(Q)}\\
        & \overset{\varepsilon \to 0}{\longrightarrow} 0.
        \end{split}
    \]
    Similarly, it is easy to prove that
    \[
    \|A_4\|_{L^2(L^{1+})}, \|A_5\|_{L^2(L^{1+})}, \|B_3\|_{L^4(Q)}, \|B_4\|_{L^2(L^{2+})}, \|B_5\|^2_{L^2(Q)} \overset{\varepsilon\to 0}{\longrightarrow}0.
    \]
    Thus we conclude that
    \[
    (e_w, e_z) \to (0,0)\quad \text{strongly in }\quad W_2\times  X_2,
    \]
        that is,
    \[
    \lim_{\varepsilon \to 0} \frac{u^\varepsilon - u}{\varepsilon} = w \quad \text{in}\quad W_2 \qquad \text{and} \qquad \lim_{\varepsilon \to 0} \frac{v^\varepsilon - v}{\varepsilon} = z \quad \text{in}\quad X_2.
    \]
\end{proof}

\noindent
{\footnotesize
\textbf{Acknowledgements. }First author thanks to the Asociación Universitaria Iberoamericana de Posgrados (AUIP) for the support to the research stay at Universidad de Sevilla during february of 2026. Second author thanks Grant I+D+I PID2023-149182NB-I00 funded by MICIU/AEI/10.13039/501100011033 and ERDF/EU and IMUS-Maria de Maeztu grant CEX2024-001517-M - Apoyo a Unidades de Excelencia María de Maeztu, funded by MICIU/AEI/ 10.13039/501100011033, and The Vicerrector\'ia Acad\'emica of the Universidad Industrial de Santander.  The third author thanks the Vicerrector\'ia de Investigaci\'on y Extensi\'on of the Universidad Industrial de Santander by the support.
}
%%%%%%%%%%%%%%%%%%%%%%%%%%%%%%%%%%%%%%%%%%%%%%%%%%%%%%%%%%%%%
%%%%%%%%%%%%%%%%%%%%%%%%%%%%%%%%%%%%%%%%%%%%%%%%%%%%%%%%%%%%%
% BIBLIOGRAFIA
%%%%%%%%%%%%%%%%%%%%%%%%%%%%%%%%%%%%%%%%%%%%%%%%%%%%%%%%%%%%%
%%%%%%%%%%%%%%%%%%%%%%%%%%%%%%%%%%%%%%%%%%%%%%%%%%%%%%%%%%%%%

\end{document}